\documentclass[11pt,normalem]{article} % use larger type; default would be 10pt
\usepackage[utf8]{inputenc} % set input encoding (not needed with XeLaTeX)

\usepackage[margin=1.25in]{geometry} % to change the page dimensions
\usepackage{graphicx} % support the \includegraphics command and options
\usepackage{booktabs} % for much better looking tables
\usepackage{array} % for better arrays (eg matrices) in maths
\usepackage{paralist} % very flexible & customisable lists (eg. enumerate/itemize, etc.)
\usepackage{verbatim} % adds environment for commenting out blocks of text & for better verbatim
\usepackage{mathrsfs} 
\usepackage{amssymb}
\usepackage{amsthm}
\usepackage{amsmath,amsfonts,amssymb}
\usepackage{esint}
\usepackage{graphics}
\usepackage{enumerate}
\usepackage{mathtools}
\usepackage{xfrac}
\usepackage{bbm}
\usepackage{subcaption}
\usepackage{stmaryrd} % \llbracket

\let\oldsquare\square % Must go before mathabx

\usepackage{mathabx}

\renewcommand{\square}{\oldsquare}

\usepackage{tocloft}

\usepackage[usenames,dvipsnames]{xcolor}
\usepackage[colorlinks=true, pdfstartview=FitV, linkcolor=blue, citecolor=blue, urlcolor=blue]{hyperref}
\usepackage[normalem]{ulem}

\newcommand{\NN}{\mathbb{N}}
\newcommand{\CC}{\mathbb{C}}

\newcommand{\sh}{\mathcal{H}}

\numberwithin{equation}{section}

\newtheorem{theorem}{Theorem}[section]

\newtheorem{corollary}[theorem]{Corollary}
\newtheorem{proposition}[theorem]{Proposition}
\newtheorem{lemma}[theorem]{Lemma}
\theoremstyle{definition}
\newtheorem{definition}[theorem]{Definition}

\newtheorem{remark}[theorem]{Remark}

\let\originalleft\left
\let\originalright\right
\renewcommand{\left}{\mathopen{}\mathclose\bgroup\originalleft}
\renewcommand{\right}{\aftergroup\egroup\originalright}

\newcommand{\vertiii}{\vert\kern-0.3ex\vert\kern-0.25ex\vert}

\newcommand*{\N}{\ensuremath{\mathbb{N}}}
\newcommand*{\Z}{\ensuremath{\mathbb{Z}}}
\newcommand*{\R}{\ensuremath{\mathbb{R}}}

\newcommand{\eps}{\varepsilon}
\renewcommand*{\Re}{\ensuremath{\mathrm{Re\,}}}

\renewcommand*{\hat}{\widehat}

\newcommand{\G}{\mathbf{G}}

\newcommand{\Li}{\mathrm{Li}_{\frac{1}{2}}\,}

\DeclareSymbolFont{boldoperators}{OT1}{cmr}{bx}{n}
\SetSymbolFont{boldoperators}{bold}{OT1}{cmr}{bx}{n}

\DeclareMathOperator*{\argmin}{argmin}

\newcommand{\RR}{\mathbb{R}}

\newcommand{\e}{\varepsilon}

\makeatletter
\newcommand\avsuminner[2]{%
	{\sbox0{$\m@th#1\sum$}%
		\vphantom{\usebox0}%
		\ooalign{%
			\hidewidth
			\smash{\,\rule[.23em]{8.8pt}{1.1pt} \relax}%
			\hidewidth\cr
			$\m@th#1\sum$\cr
		}%
	}%
}
\makeatother

\makeatletter
\newcommand\avsuminnerr[2]{%
	{\sbox0{$\m@th#1\sum$}%
		\vphantom{\usebox0}%
		\ooalign{%
			\hidewidth
			\smash{\,\rule[.23em]{6pt}{0.7pt} \relax}%
			\hidewidth\cr
			$\m@th#1\sum$\cr
		}%
	}%
}
\makeatother

\def\Xint#1{\mathchoice
	{\XXint\displaystyle\textstyle{#1}}%
	{\XXint\textstyle\scriptstyle{#1}}%
	{\XXint\scriptstyle\scriptscriptstyle{#1}}%
	{\XXint\scriptscriptstyle\scriptscriptstyle{#1}}%
	\!\int}
\def\XXint#1#2#3{{\setbox0=\hbox{$#1{#2#3}{\int}$}
		\vcenter{\hbox{$#2#3$}}\kern-.5\wd0}}

\def\fint{\Xint-}

\makeatletter 
\newcommand{\negphantom}{\v@true\h@true\negph@nt} 
\newcommand{\neghphantom}{\v@false\h@true\negph@nt} 
\newcommand{\negph@nt}{\ifmmode\expandafter\mathpalette 
	\expandafter\mathnegph@nt\else\expandafter\makenegph@nt\fi} 
\newcommand{\makenegph@nt}[1]{% 
	\setbox\z@\hbox{\color@begingroup#1\color@endgroup}\finnegph@nt} 
\newcommand{\finnegph@nt}{% 
	\setbox\tw@\null 
	\ifv@ \ht\tw@\ht\z@\dp\tw@\dp\z@\fi \ifh@\wd\tw@-\wd\z@\fi\box\tw@} 
\newcommand{\mathnegph@nt}[2]{% 
	\setbox\z@\hbox{$\m@th #1{#2}$}\finnegph@nt} 
\makeatother

\newcommand{\red}[1]{\textcolor{blue}{#1}}

\usepackage{titlesec}

\newcommand{\addperiod}[1]{#1.}
\titleformat{\section}
{\normalfont\Large}{\thesection.}{0.5em}{}
\titleformat*{\subsection}{\normalfont\large}%{\bfseries}
\titleformat{\subsubsection}[runin]
{\bfseries}
{\thesubsubsection.}
{0.5em}
{\addperiod}
\titleformat*{\subsubsection}{\bfseries}
\titleformat*{\paragraph}{\bfseries}
\titleformat*{\subparagraph}{\large\bfseries}

\title{\bf \Large Interaction energies in paranematic colloids}

\author{Dmitry Golovaty
	\thanks{Department of Mathematics, University of Akron.
		{\footnotesize \href{mailto:dmitry@uakron.edu}{dmitry@uakron.edu}.}
	}
	\and 
	Jamie Taylor
	\thanks{Department of Mathematics, CUNEF University.
		{\footnotesize \href{mailto:jamie.taylor@cunef.edu}{jamie.taylor@cunef.edu}.}
	}
 \and
	Raghavendra Venkatraman
	\thanks{Courant Institute of Mathematical Sciences, New York University.
		{\footnotesize \href{mailto:raghav@cims.nyu.edu}{raghav@cims.nyu.edu}.}
	}\and 
Arghir Zarnescu
\thanks{Basque Center for Applied Mathematics, Ikerbasque Foundation and "Simion Stoilow" Institute of the Romanian Academy.
	{\footnotesize \href{mailto:azarnescu@bcamath.org}{azarnescu@bcamath.org}.}
}
}
\date{\today}

\usepackage[nottoc,notlot,notlof]{tocbibind}

\begin{document}

\maketitle
	
\begin{abstract}
We consider a 2D system of colloidal particles embedded in a paranematic---an isotropic phase of a nematogenic medium above the temperature of the nematic-to-isotropic transition. In this state, the nematic order is induced by the boundary conditions in a narrow band around each particle and it decays exponentially in the bulk. 

We develop rigorous asymptotics of the linearization of the appropriate variational model that allow us to describe weak far-field interactions between the colloidal particles in two dimensional paranematic suspensions. We demonstrate analytically that decay rates of solutions to the full nonlinear and linear problems are similar and verify numerically that the interactions between the particles in these problems have similar dependence on the distance between the particles.  We go beyond the existing literature by considering the next order term in the expansion and discover that the interaction can be either repulsive or attractive. 
Finally, we perform Monte-Carlo simulations for a system of colloidal particles in a paranematic and describe the statistical properties of this system.
\end{abstract}
	
	\setcounter{tocdepth}{2}  
	\tableofcontents
		\section{Introduction and main results}
%	Our goal is to initiate the study of interaction energies between colloidal particles in the paranematic regime.   {yada yada}. The main model is nonlinear, but we study a (formally) linearized version of it around the isotropic state, in detail, and then demonstrate by numerical experiments that the nonlinear version of the problem shares a number of qualitative features of our linear analysis. 
	
%	To fix ideas we now introduce a simpler vector-valued model that retains all principal features of the Landau-de Gennes approach.
 
We aim to initiate the study of interaction energies between colloidal particles in a nematic liquid crystal environment. There exists a significant body of physics literature on this topic (see for instance \cite{smalyukh2018liquid,senyuk2019high,smalyukh2020knots}), mostly based on simulations of certain variational models. The analytical intuition behind these interactions follows ideas developed in the seminal paper \cite{poulin1997novel}. In this work, interactions between colloidal particles are established based on a suitable linearisation at infinity and on formal analogies with a classical theme, namely interactions of electrostatic multipoles. A rigorous understanding of these interactions is still missing in the case of several particles, while for the case of a single particle it was considered in the recent work \cite{alama2023far}.
	
The main goal in this paper is to provide a rigorous underpinning to the intuition developed in the physical literature, aiming to obtain explicit estimates quantifying the interaction in the case of several particles expressed in terms of the geometric and material parameters of the problem. The models typically used to describe this physical setting are nonlinear but following formal ideas in \cite{PhysRevE.60.4210,PhysRevE.66.041705}, we reduce the problem to the linearization around the isotropic state and discuss the precise analytical meaning in which the solution to the resulting linear problem approximate the minimizers of the nonlinear problem. Further, we conduct the detailed analytical study of the linear problem, and then show via numerical experiments that the nonlinear version of the problem shares a number of qualitative features with our linear analysis. 
	
In the long-term, we will be interested in understanding a Landau-de Gennes model of nematic liquid crystals. The main features of this model---based on a tensor-valued order parameter---are presented in Appendix~\ref{sec:LDGmodel}. In the current paper we focus on a simpler, vector-valued model in two dimensions that retains the relevant features of the Landau-de Gennes approach. To this end, suppose that the liquid crystal is described by  $u:\Omega\to\R^2$, where $\Omega\subset\R^2$ is  an open, smooth and not necessarily bounded domain that models the container occupied by the nematic liquid crystal. Let 
	\begin{equation}
		\label{eq:pot0}
		W(u)=k(T){|u|}^2-2{|u|}^4+{|u|}^6
	\end{equation}
be the bulk potential, the minima of which describe a physical system that may undergo a phase transition at some critical temperature $T=T^*$, i.e. in mathematical terms, the type and number of minima change at this temperature.  

An examination of $W$ reveals that it has exactly one minimum at the isotropic state $u=0$ when $k(T)>\frac{4}{3}.$ When $\frac{4}{3}\geq k(T)>1$ there is a global minimum at $u=0$ and a local minimum at $|u|=\alpha(T)$ that represents a metastable ordered state. When $k(T)=1$ both minima have equal depth, while $|u|=\alpha(T)$ and $u=0$ become the global minimum and a local minimum, respectively, when $1>k(T)>0.$ When $k(T)<0,$ the circle $|u|=\alpha(T)$ is the global minimal set, while $u=0$ is the local maximum of $W$. We call the temperature $T^*$ satisfying $k(T^*)=1$ the temperature of the phase transition between the isotropic and the ordered states. 
	
We are interested in the behavior of minimizers of the functional 
\begin{equation}
    \label{eq:pot}
	E_\e(u,\Omega)=\int_\Omega{\frac{1}{2}{|\nabla u|}^2+\frac{1}{4\e^4}W(u)},
\end{equation}
where we assume that $u\in H^1(\Omega)$ satisfies Dirichlet boundary data on $\partial\Omega$. We will be focused on domains $\Omega$ that are exterior to a collection of colloidal particles, and seek to understand the inter-particle interactions as mediated by the background ordered state. As a first step in this program, we are interested in a so-called \emph{paranematic} regime when $k(T)>\frac{4}{3}.$ In this regime, the potential $W$ is convex and has a single minimum at the isotropic state $u=0.$ Nonetheless, whilst the ground state is isotropic in the bulk, some residual nematic ordering may still be induced by the boundary conditions. {In the remainder of this paper we fix $k(T)>\frac{4}{3}$ and refer to \eqref{eq:pot} as the {\it paranematic energy functional.}}

Note that interactions between spherical particles immersed in an isotropic phase of a nematogenic fluid were investigated in the physical literature in the past by considering formal asymptotics \cite{PhysRevE.60.4210}-\nocite{PhysRevE.61.2831,PhysRevE.71.062701,PhysRevLett.86.3915,galatola2003interaction}\cite{PhysRevE.66.041705} in three dimensions. Here we will focus instead on rigorous understanding of the regime of two spherical colloids in $\R^2$ when the domain is the whole space. To fix ideas, we  suppose thus that there are two identical, spherical colloidal particles~$\hat{B}_1^\e$ and ~$\hat{B}^\e_2,$ each of radius~$1$, that are separated by distance $2b\e^2 > 0.$ We also take $\Omega=\R^2\setminus \overline{\hat{B}_1^\e \cup \hat{B}_2^\e}$ {  and assume that the admissible competitors satisfy the Dirichlet conditions $u = g_i$ on $\partial \hat{B}_i^\e$ for $i=1,2$. Under the paranematic interaction energy between the two particles, we understand the difference $$E_{int}:=\min E_\e(u,\Omega)-\min E_\e(u,\R^2 \setminus\hat{B}_1^\e)-\min E_\e(u, \R^2 \setminus\hat{B}_2^\e), $$
where the second and the third minima  taken among the competitors satisfying the boundary conditions $u = g_1$ on $\partial \hat{B}_1^\e$ and $u = g_2$ on $\partial \hat{B}_2^\e,$ respectively, represent the {\it self-energies} of the particles.}  
	
We show in Proposition~\ref{t.nonlinear} that in the paranematic regime, the unique solution to the nonlinear Euler-Lagrange problem
\begin{equation} \label{e.PDE.0}
	\begin{aligned}
			\Delta u &= \frac{1}{4\e^4}\nabla_uW(u), \quad u \in H^1\left(\R^2 \setminus \overline{\hat{B}_1^\e \cup \hat{B}_2^\e}\right)\\
			u & = g_1 \qquad \mbox{ on } \partial \hat{B}_1^\e\\
			u &= g_2 \qquad \mbox{ on } \partial \hat{B}_2^\e\,. 
	\end{aligned}
\end{equation}
for \eqref{eq:pot} {\em has the same rate of decay} as that of a solution of the corresponding linearization of \eqref{e.PDE.0} around the state $u\equiv0,$ namely
	\begin{equation} \label{e.PDE.1}
		\begin{aligned}
			\Delta u &= \frac{k(T)}{2\e^4}u, \quad u \in H^1\left(\R^2 \setminus \overline{\hat{B}_1^\e \cup \hat{B}_2^\e}\right)\\
			u & = g_1 \qquad \mbox{ on } \partial \hat{B}_1^\e\\
			u &= g_2 \qquad \mbox{ on } \partial \hat{B}_2^\e\,. 
		\end{aligned}
	\end{equation}
This observation allows us to conjecture that far-field paranematic-mediated interactions between two particles in a nematogenic medium should depend on the distance between the particles in a way similar to that for the particles in the corresponding linear problem.
 \begin{comment}
     To be precise, we prove
	\begin{theorem} \label{t.nonlinear}
		For every $\e>0,$ suppose that $u_\e$ and $v_\e$ solve \eqref{e.PDE.0} and \eqref{e.PDE.1}, respectively, subject to the same boundary conditions. Then there exists a constant $C>0,$ indepenent of $\e$, such that
		\begin{equation} \label{e.claim1}
			\left\|u_\e-v_\e\right\|_{L^2(\R^2 \setminus \hat{B}^\e_1 \cup \hat{B}^\e_2)}\leq C\e\left\|g\right\|_{L^2(\partial \hat{B}^\e_1 \cup \partial \hat{B}^\e_2)}\,,
		\end{equation}  
  and 
  \begin{equation}
      \label{e.claim2}
      \|\nabla (u_\e - v_\e)\|_{L^2(\R^2 \setminus \hat{B}^\e_1 \cup \hat{B}^\e_2) } \leqslant \frac{C}{\e}\|g\|_{L^2(\partial \hat{B}^\e_1 \cup \partial \hat{B}^\e_2)}\,.
  \end{equation}
	\end{theorem}
 \begin{remark}
		\label{r.general}
		As we will demonstrate in the sequel, the specifics of the geometry in \eqref{e.PDE.0} (or \eqref{e.PDE.1}) is not at all necessary, and is only for ease of stating the result. We will explain how the conclusions of the above theorem stay valid for more general geometries that permit for more than two, possibly non-spherical particles, of varying sizes, etc.
	\end{remark}
 \end{comment}
	
By rescaling, we can assume that $k(T)=2$ in \eqref{eq:pot0}, hence the linear PDE we will consider in the sequel is 
\begin{equation} \label{e.PDE}
	\begin{aligned}
			\Delta U &= \frac{1}{\e^4}U, \quad U \in H^1\left(\R^2 \setminus \overline{\hat{B}_1^\e \cup \hat{B}_2^\e}\right)\\U & = G_1 \qquad \mbox{ on } \partial \hat{B}_1^\e\\
			U &= G_2 \qquad \mbox{ on } \partial \hat{B}_2^\e\,. 
		\end{aligned}
\end{equation}
It is clear that for the equation~\eqref{e.PDE}, the vectorial nature of $U$ is unimportant and, therefore, whenever possible we will assume $U$ to be a scalar. 
	
{ For concreteness, we set $\hat{B}_1^\e$ and $\hat{B}_2^\e$ to be open unit disks centered at $(0, 1 + b\e^2)$ and $(0, -1 - b\e^2),$ respectively, so that the distance of separation between the disks is $2b \e^2 > 0.$ To understand this choice of geometry, note that the solutions of \eqref{e.PDE} satisfying the nonzero Dirichlet data on $\partial \hat{B}_1^\e \cup \partial \hat{B}_2^\e$, decay exponentially fast within the distance~$\sim\e^2$ away from the boundaries of the disks. 

Assuming that $\e\ll1$ and that two disks are on the distance~$O(1)$ from each other, e.g., when $b\sim\frac1{\e^2}$, the solutions of \eqref{e.PDE} are non-vanishing within two narrow ~$\e^2-$wide non-overlapping "coronas" surrounding the disks. In this regime, we can think of the disks as not paranematically interacting. Alternatively, when $b\sim1,$ the particles are close to touching, their coronas intersect and the energy of the resulting paranematic interaction is comparable to their self-energies.

In what follows we will focus on the intermediate regime when $\e^2\ll b\e^2\ll1$. In this case, the particles are almost touching and their exponentially decaying paranematic coronas overlap significantly (i.e.,~$b\e^2\ll1)$, but the overlap occurs at a lengthscale that is much larger than the screening lengthscale of each corona (i.e.,~$b\e^2 \gg \e^2)$. Our main interest will be on estimating the energy of paranematic interaction between the disks which in this case can be thought of as a lower order correction to the self-energy of the particles. Note that we are not interested in computing the limit of the energy as $\e\to0$ and $b\to\infty$.} 
	
It will often be convenient to rescale the problem via the change of variables $x' = \frac{x}{\e^2}, y' = \frac{y}{\e^2}$ and subsequently drop the primes. Then, setting $B_1^\e $ to be the disk of radius $\frac{1}{\e^2}$ centered at $(0,\frac{1}{\e^2} + b)$ and $B_2^\e$ to be the disk of radius $\frac{1}{\e^2}$ centered at $(0, - \frac{1}{\e^2} - b),$ with a slight abuse of notation, we have

\begin{equation} \label{e.PDEblownup}
		\begin{aligned}
			\Delta u &= u, \quad\quad u\in H^1\left(\R^2 \setminus \overline{B_1^\e \cup B_2^\e}\right)\\
			u & = g_1 \qquad \mbox{ on } \partial B_1^\e\\
			u &= g_2 \qquad \mbox{ on } \partial B_2^\e\,.
		\end{aligned}
	\end{equation}

In the above, functions with lower case letters $(u,g_1, g_2)$ represent scaled versions of their upper-case counterparts $(U,G_1,G_2)$. We point out that in the blown up variables, the separation between the disks is $2b > 0.$ 

 %\red{\bf Comment JMT: We start at the energy, consider the EL equation, linearlise the EL equation, then reinterpret the EL equation as an energy problem (eqs 1.2-1.7). As we are studying asymptotics of the energy, moreso than asymptotics of the minimisers themselves, I think it would be a good idea to view the energy we consider as a linearlised (quadraticised?) energy, if possible, without using the PDE as a stepping stone.}
	
For the rescaling as above, we observe that the natural quadratic energies associated to the two settings are equal: 
	\begin{equation*}
		\frac12\int_{\R^2 \setminus \overline{\hat{B}_1^\e \cup \hat{B}_2^\e}} \biggl(|\nabla U|^2 + \frac{1}{\e^4} U^2 \biggr)\,dx = \frac12\int_{\R^2 \setminus \overline{B_1^\e \cup B_2^\e}} \biggl(|\nabla u|^2 + u^2\biggr) \,dx\,.
	\end{equation*}
Note that, although formally these energies can be thought of as a leading order approximation of \eqref{eq:pot} when the supremum norm of $u$ is small, this is not true in the current case as the boundary data $g_i,$ $i=1,2$ is of order $1.$

It will be shown by direct energy comparison with a competitor, that the energy of the unique solution to~\eqref{e.PDEblownup} given by
	\begin{equation}\label{def:Feps}
		F_\e(u) := \frac12\int_{\Omega_\e} \Bigl(|\nabla u|^2 + u^2\Bigr) \,dx\,, \quad \Omega_\e := \R^2 \setminus (B_1^\e \cup B_2^\e)\,. 
	\end{equation}
satisfies 
	\begin{equation} \label{e.competitor}
		F_\e(u) \leqslant \frac{C}{\e^2}\|g\|_{L^2}^2\,. 
	\end{equation}
For the linear problem, the goal in this paper is to give a precise energy expansion of the first two terms of the minimum energy $E_\e(u),$ in terms of the parameter $\e,$ and quantify the expansion. To be precise, let us note that the problems \eqref{e.PDE} and \eqref{e.PDEblownup} are associated with variational principles, and the solutions to these PDEs arise as unique minimizers of strictly convex energies. Focusing on \eqref{e.PDEblownup} with the associated energy $F_\e(u),$ we set 
	\begin{equation}\label{def:keps}
		\kappa_\e (g_1,g_2) := \min \Bigl\{  F_\e(w) : w \in H^1(\Omega_\e),  w = g_i \,\, \mbox{ on } \partial B_i^\e\Bigr\}\,. 
	\end{equation}
Our first main result, to be provided in Section~\ref{sec:twoparticlesconst}	 concerns the first two terms in an asymptotic expansion of $\kappa_\e(g_1,g_2)$ in powers of $\e.$ and for constant boundary conditions $g_1,g_2\in\R$. The leading terms are of order $O(\e^{-2})$, and correspond to the energy of each individual particle. However, due to the presence of two particles (rather than one), the particles interact, and there is a correction to the leading order energy which occurs at order $O(\e^{-1}).$ Computing this interaction energy exactly is the main contribution of our work. More precisely, in Subsection~\ref{ss.two.const} we will establish

\begin{theorem}
		\label{t.2balls.constbc}
Let $B_1^\eps=\{x\in\R^2, d(x,(0,\eps^{-2}+b))<\eps^{-2}\}$ and $B_2^\eps=\{x\in\R^2, d(x,(0,-\eps^{-2}-b))<\eps^{-2}\}.$ For constant boundary conditions $g_1,g_2\in \R$ consider the energy $F_\eps$ defined as in \eqref{def:Feps} and its minimum $\kappa_\eps$ as defined in \eqref{def:keps}.  {Then, there exists a constant~$C$ independent of~$\e$ and~$b,$ such that for~$b,\e>0$ such that~$b \gg 1,\  b \e^2 \ll 1,$} we have the  estimate
		\begin{equation}
			\label{e.exp.2balls.constbc}
			\left| \kappa_\e (b) - { \frac{\pi}{\e^2} (g_1^2 + g_2^2) +  2g_1 g_2 \frac{e^{-2b}\sqrt{\pi}}{\e} } \right| \leqslant C \frac{e^{-4b} + b\e^2 e^{-2b}}{\e}=o\left(\frac{e^{-2b}}{\e}\right)
		\end{equation}
\end{theorem}

\begin{remark} \label{r.dmitry}
In the other parameter regime that we will consider in Section~\ref{s.O1separation}, we let $\e \to 0^+$ for the fixed separation $b = O(1)$. Here we will only provide the formal asymptotics in order to keep the paper length manageable, although rigorous statements should be obtainable in this case as well. More specifically, in this case we have

\begin{multline*}
\kappa_\e=\frac{2\pi}{\e^2}\left(g_1^2+g_2^2\right)+\frac{\sqrt{\pi/2}}{\e}\left[2e^{-4b}+\frac{1}{\sqrt{2}}e^{-8b}\right](g_1^2+g_2^2)\\-\frac{4\sqrt{\pi}}{\e}\left[e^{-2b}+\frac{1}{\sqrt{3}}e^{-6b}\right]g_1g_2+O(1),
\end{multline*}

\end{remark}

Moving on to the more general case of non-constant boundary conditions, we will show in Section~\ref{ss.two.var} the following:

	\begin{theorem}
		\label{t.mt}

    Let $B_1^\eps=\{x\in\R^2, d(x,(0,\eps^{-2}+b))<\eps^{-2}\}$ and $B_2^\eps=\{x\in\R^2, d(x,(0,-\eps^{-2}-b))<\eps^{-2}\}.$  For smooth functions $\hat{g}_1,\hat{g}_2:\mathbb{S}^1\to{ \mathbb{R}},$ we define $g_i:\partial B_i^\eps\to{ \mathbb{R}}$ via $$ g_i(x):=\hat{g}_i \left(\e^2x-{(-1)}^{i+1}\left(1+\e^2b\right)\right),\quad i=1,2,$$and consider the energy $F_\eps$ defined as in \eqref{def:Feps} and its minimum $\kappa_\eps$ as defined in \eqref{def:keps}.  {Then, there exists~$C>0$ independent of~$\e,b$ such that for~$b,\e>0$ such that~$b \gg 1,\  b \e^2 \ll 1,$ we have}	
		\begin{align*}
			&\left| \kappa_\e(b) - { \frac{\pi}{\e^2} \Biggl(\fint_{\partial \mathbb{S}^1} |\hat{g}_1|^2 \,d\sh^1 + \fint_{\partial \mathbb{S}^1} |\hat{g}_2|^2\,d \sh^1\Biggr) +  \frac{2e^{-2b}\sqrt{\pi}}{\e}  g_1(p) g_2(q)}\right| \\
			&\quad \quad \leqslant C  \frac{(e^{-4b} +b\e^2e^{-2b} ) \left(\|\hat{g}_1\|_{H^{1} (\mathbb{S}^1)}^2 + \|\hat{g}_2\|_{H^1(\mathbb{S}^1)}^2\right)}{\e}\,,
		\end{align*}
	where $p \in \partial B_1^\e$ and $q \in \partial B_2^\e$ denote the points on the two respective circles that are closest to each other.
		\end{theorem}

Theorems~\ref{t.2balls.constbc} and \ref{t.mt} (and Remark~\ref{r.dmitry}) describe the fine asymptotic behavior at the level of energies, to the solutions to the simple linear PDE problem~\eqref{e.PDE.1}. The subtlety in~\eqref{e.PDE.1} lies in the fact that the (exterior) domain becomes singular when $\e \to 0^+$ and consists of two touching disks. In order to place our results in a broader context, let us note that linear PDE and systems with piecewise constant coefficients (with possibly high contrast) in domains exterior to two nearly touching obstacles has received a lot of attention \cite{t1,t2,t3,t4}. This literature concerns itself with gradient estimates on the solution to the PDE, and their character in the region of closest contact between the two touching inclusions. We think of these as describing the \emph{leading order} behavior in a pointwise sense, close to the contact points of the obstacles. In this paper we are instead interested in the \emph{next order interaction effects} that we capture through the careful asymptotic analysis of a more global quantity, namely, the energy associated to the PDE~\eqref{e.PDE.1}. We discover that the interaction can be either repulsive or attractive, unlike previously suggested in the physics literature where only attractive effects (flocculation) were predicted \cite{PhysRevE.60.4210,PhysRevE.61.2831,PhysRevE.71.062701}.

Further, in Section~\ref{sec:multiple}, we will briefly explain how the results for two-particle interactions can be extended to the case of several particles, under suitable assumptions.

In Section~\ref{sec:numerics} we will use numerics to explore the similarities and differences at the level of minimizers and energy scaling, between the solutions of the linear problem corresponding to the simple quadratic potential $\frac12 u^2$ in the energy $F_\e(u)$ and the solutions of the nonlinear problem corresponding to a more physical potential $W(u)$ in the paranematic regime.

 The results presented here provide first analytical steps towards rigorous understanding of multiple particles interactions. In order to offer a glimpse into the future explorations, we provide in Figure~\ref{fig.MHMC} some Monte Carlo simulations results based on the ideas developed here. These show configurations of several particles with boundary conditions having different topological degrees and the details are provided  in Section~\ref{sec:MonteCarlo}. 
 
\begin{figure}[ht]
\begin{subfigure}[h]{0.3\textwidth}
\begin{center}
\includegraphics[width=0.99\textwidth]{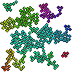}
\caption{Degree $2$}\label{subfig.MHMC.2}
\end{center}
\end{subfigure}\begin{subfigure}[h]{0.3\textwidth}
\begin{center}
\includegraphics[width=0.99\textwidth]{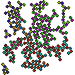}
\caption{Degree $3$}\label{subfig.MHMC.3}
\end{center}
\end{subfigure}
\begin{subfigure}[h]{0.3\textwidth}
\begin{center}
\includegraphics[width=0.99\textwidth]{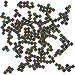}
\caption{Degree $5$}\label{subfig.MHMC.5}
\end{center}
\end{subfigure}
\caption{Some Monte-Carlo simulations for multiple particles}\label{fig.MHMC}
\end{figure} 

	\section{Two particles and constant boundary conditions}
\label{sec:twoparticlesconst}

 \label{ss.two.const}
	The focus of this section is on the case when the boundary conditions in \eqref{e.PDEblownup}  are constant, so that $g_1, g_2 \in \R.$ Recall that $\Omega_\e= \RR^2 \setminus \overline{B_1^\e \cup B_2^\e}$ is the domain exterior to two large balls of the radius $1/\e^2$ each and situated at distance $2b$ away from each other.  We will provide  an energy expansion in two cases:
	\begin{enumerate}
		\item that holds for a \emph{fixed} {$\e$ that is sufficiently small} and $b,$ {that is sufficiently large} {  provided~$b \e^2 \ll 1,$} and 
		\item that holds {for a fixed~$b > 0$, but} in the limit $\e \to 0^+,$ {so that the separation between the particles is of the order of~$\e^2.$ }
	\end{enumerate}

 The first case, to be treated in the next subsection will be studied rigorously, providing all the details, while for the other case, to be treated in Subsection~\ref{s.O1separation}  we will only provide the formal asymptotics, in order to shorten the presentation.
	
	\subsection{Interaction energies between particles when their separation satisfies: { $1 \ll b \ll \frac1{\e^2}$}}
	
	%{  The goal of this section is to prove our main theorems when in the setting of equations (1.6) there are two particles of radius~$\frac1{\e^2}$ whose closures are~$2b$ apart, where~$b$ is large, but not too large, in the sense of~$1 \ll b \ll \frac1{\e^2}.$ Note that in this regime~$1 \gg b\e^2 \gg \e^2$ so that the separation between unit size particles is small (but not too small), with respect to the scaling in~(1.5). The reason to focus on this regime is easy to understand from the scaling of unit size particles as in (1.5). n this setting, the equations take the form~$\Delta u =  Ku$ for a large \textbf{positive} constant~$K = \frac1{\e^4}.$ Consequently, starting from satisfying the Dirichlet data at the boundaries, solutions decay exponentially fast within a lengthscale~$K^{-\frac12} = \e^2.$ The setting that we are focused on consists of two such exponentially decaying coronas overlapping significantly (i.e.,~$b\e^2 \ll 1$), but the overlap occurs on a lengthscale much larger than the screening lengthscale of the corona (i.e.,~$b\e^2 \gg \e^2)$. In this setting, if one were to consider the limit as~$ b \to \infty$, the limiting geometry of the domains is \textbf{singular}, and consists of two touching unit circles. Our estimates are  valid not just in the limit~$b \to \infty$, but for fixed~$\e,b$ provided~$1 \ll b \ll \frac1{\e^2}$.} 
	  The main result of this section is an expansion of the energy of $U$ in terms of $\e$, as stated in Theorem~\ref{t.2balls.constbc}.

The proof of Theorem \ref{t.2balls.constbc} is contained in a sequence of Lemmas. For $i=1,2,$ let 
	\begin{equation}
		\label{eq:psi}
		\Psi_i(x) := \frac{K_0 \big(|x - a_i^\e|  \big)}{K_0\big(\frac{1}{\e^2} \big)}\,, \quad x \in \Omega_\e
	\end{equation} where $K_0$ is a modified Bessel function of the second kind (see Appendix~\ref{s.app1} for details).
One can check that this is 	the solution of the single particle exterior problem
	\begin{equation}
	\label{eq:sipa}
		\begin{aligned}
			&\Delta \Psi_i =  \Psi_i \quad \quad \mbox{ in } \R^2\backslash B_i^\e,\\
			&\Psi_i = 1 \quad \quad \quad\  \mbox{ on } \partial B_i^\e
		\end{aligned}	
	\end{equation}

 Let us define $\sigma :\{1,2\} \to \{1,2\}$ via $\sigma(1) = 2, \sigma(2) = 1.$ Suppose that $Z_i\in H^1\left(\Omega_\e\right)$ solve 
	\begin{equation} \label{e.ZjPDE}
		\begin{aligned}
			&\Delta Z_i =  Z_i \quad \quad \mbox{ in } \Omega_\e,\\
			&Z_i = 0 \quad \ \quad \quad \mbox{ on } \partial B_i^\e, \\
			&Z_i = -\Psi_{i} \quad \quad \mbox{ on } \partial B_{\sigma(i)}^\e
		\end{aligned}
	\end{equation}
	for $i=1,2.$ Then it is easily seen that the unique solution to \eqref{e.PDEblownup} is given by 
	\begin{align} \label{e.sol}
		U = g_1 (\Psi_1 + Z_1) + g_2 (\Psi_2 + Z_2). 
	\end{align}
	 If for any $\xi_1, \xi_2 \in H^1(\Omega_\e)$ we define
	\begin{align*}
		&\langle \xi_1, \xi_2\rangle_\e := \langle \xi_1, \xi_2\rangle_{H^1(\Omega_\e)}= \int_{\Omega_\e} \xi_1 \xi_2 + \nabla \xi_1 \cdot \nabla \xi_2
	\end{align*}
	and denote $\|\xi\|_\e^2 := \langle \xi,\xi \rangle_\e,$ we observe that  
	\begin{align*}
		F_\e(U) = \frac 12\|U\|_{\e}^2.
	\end{align*}
Then, 
	\begin{equation} \label{e.decompofenergy}
		\begin{aligned}
			&2F_\e(U) = \|U\|_{\e}^2 = g_1^2 \|\Psi_1 + Z_1\|_\e^2 + 2g_1 g_2 \langle \Psi_1 + Z_1, \Psi_2 + Z_2\rangle_\e + g_2^2 \|\Psi_2 + Z_2\|_\e^2\\
			&\qquad = \begin{pmatrix}
				g_1 & g_2
			\end{pmatrix} \begin{pmatrix}
				\|\Psi_1 + Z_1\|_\e^2 & \langle \Psi_1 + Z_1 , \Psi_2 + Z_2\rangle_\e\\
				\langle \Psi_1 + Z_1 , \Psi_2 + Z_2\rangle_\e & \|\Psi_2 + Z_2\|_\e^2
			\end{pmatrix} \begin{pmatrix}
				g_1 \\
				g_2
			\end{pmatrix}.
		\end{aligned}
	\end{equation}
	%\begin{lemma}
	%   Suppose that $\Psi_i,\ i=1,2$ is as defined in \eqref{eq:psi} and $\alpha>0.$ Then
	%   \[\int_{\partial B_{\sigma(i)}^\e}\Psi_i^\alpha(s)ds=.\]
	%\end{lemma}
	
	Our first lemma expresses each of the terms in the above matrix in terms of certain boundary integrals. Naturally, this is done using integration by parts-- for this purpose we let $\nu_i,i = 1,2$ denote unit normals that point \emph{towards the centers of the discs} $B_i^\e.$ In particular, for the exterior domain $\Omega_\e,$ these represent outward unit normals. We will collectively refer to both these normals (i.e., as outward unit normal to $\partial \Omega_\e$) via $\nu$.\\
	\begin{lemma}
		\label{l.boundaryintegrals1}
		We have the following identities
		\begin{align} \label{e.diag}
			\|\Psi_2 + Z_2\|_\e^2 = \|\Psi_1 + Z_1\|_\e^2 = - \frac{2\pi}{\e^2} \frac{K_0^\prime \bigl( \frac{1}{\e^2} \bigr)}{K_0\bigl( \frac{1}{\e^2}\bigr)} - \int_{\partial B_2^\e} \Psi_1 \biggl( \frac{\partial \Psi_1}{\partial \nu_2} + \frac{\partial Z_1}{\partial \nu_2} \biggr)\,d \sh^1 ,
		\end{align}
		and  
		\begin{align} \label{e.offdiag}
			\langle \Psi_1 + Z_1, \Psi_2 + Z_2\rangle_\e =\int_{\partial B_2^\e} \frac{\partial \Psi_1}{\partial \nu_2}\,d \sh^1 - \int_{\partial B_2^\e} \Psi_1 \frac{\partial \Psi_2}{\partial \nu_2}\,d \sh^1 - \int_{\partial B_1^\e}\Psi_2  \frac{\partial Z_1}{\partial \nu_1}\,d \sh^1. 
		\end{align}
	\end{lemma}
	\begin{proof}
		{To prove \eqref{e.diag},} we compute $\|\Psi_1 + Z_1\|_\e^2,$ with the other term being symmetrical. Computing, and using the PDE and boundary conditions satisfied by $\Psi_1 $ and $Z_1,$ we observe
		\begin{equation} \label{e.diag1}
			\begin{aligned}
				&\| \Psi_1 + Z_1\|_\e^2 = \|\Psi_1\|_\e^2 + 2 \langle \Psi_1 , Z_1\rangle_\e + \|Z_1\|_\e^2 \\
				&\quad \quad = \int_{\Omega_\e} \left(\Psi_1^2 + |\nabla \Psi_1|^2 + 2 \Psi_1 Z_1 + 2\nabla \Psi_1 \cdot \nabla Z_1 + |Z_1|^2 + |\nabla Z_1|^2 \right)\,dx\\
				&\quad \quad = \int_{\Omega_\e} \Psi_1 \left( \Psi_1 - \Delta \Psi_1 \right)\,dx + \int_{\partial \Omega_\e} \Psi_1 \frac{\partial \Psi_1}{\partial \nu} \,d \sh^1\\
				&\quad \quad \ + 2 \int_{\Omega_\e} Z_1 \left( \Psi_1 - \Delta \Psi_1 \right) \,dx + 2 \int_{\partial \Omega_\e} Z_1 \frac{\partial \Psi_1}{\partial \nu}\,d \sh^1\\
				&\quad \quad \  + \int_{\Omega_\e} Z_1 (Z_1 - \Delta Z_1)\,dx + \int_{\partial \Omega_\e} Z_1 \frac{\partial Z_1}{\partial \nu}\,d \sh^1\\
				&\quad \quad =\int_{\partial B_1^\e} \frac{\partial \Psi_1}{\partial \nu_1}\,d \sh^1 + \int_{\partial B_2^\e}\Psi_1 \frac{\partial \Psi_1}{\partial \nu_2}\,d \sh^1 - 2 \int_{\partial B_2^\e} \Psi_1 \frac{\partial \Psi_1}{\partial \nu_2} \,d \sh^1 - \int_{\partial B_2^\e} \Psi_1 \frac{\partial Z_1}{\partial \nu_2}\,d \sh^1\\
				&\quad \quad = \int_{\partial B_1^\e} \frac{\partial \Psi_1}{\partial \nu_1}\,d \sh^1 -\int_{\partial B_2^\e}\Psi_1 \left( \frac{\partial \Psi_1}{\partial \nu_2} + \frac{\partial Z_1}{\partial \nu_2}\right)\,d \sh^1\\
				&\quad \quad = -\frac{2\pi}{\e^2} \frac{K_0^{\prime} \left(\frac{1}{\e^2} \right)}{K_0\left( \frac{1}{\e^2}\right)}-\int_{\partial B_2^\e}\Psi_1 \left( \frac{\partial \Psi_1}{\partial \nu_2} + \frac{\partial Z_1}{\partial \nu_2}\right)\,d \sh^1. 
			\end{aligned}
		\end{equation}
		Similarly, {to prove \eqref{e.offdiag},} we notice
		\begin{equation}
			\begin{aligned}
				&\langle \Psi_1 + Z_1,\Psi_2 + Z_2\rangle_\e = \int_{\Omega_\e} (\Psi_1 + Z_1) (\Psi_2 + Z_2) + (\nabla \Psi_1 + \nabla Z_1)\cdot (\nabla \Psi_2 + \nabla Z_2)\,dx \\
				&\quad = \int_{\Omega_\e}  (\Psi_1 - \Delta \Psi_1) \Psi_2 \,dx + \int_{\partial \Omega_\e} \frac{\partial \Psi_1}{\partial \nu}\Psi_2 \,d\sh^1 \\
				&\quad \quad + \int_{\Omega_\e} (\Psi_1 - \Delta \Psi_1)Z_2
				\,dx + \int_{\partial \Omega_\e} Z_2\frac{\partial \Psi_1}{\partial \nu}\, \sh^1    \\
				&\quad \quad + \int_{\Omega_\e} (\Psi_2 - \Delta \Psi_2)Z_1 \,dx + \int_{\partial \Omega_\e} Z_1 \frac{\partial \Psi_2}{\partial \nu}\,d \sh^1 \\
				&\quad \quad + \int_{\Omega_\e} (Z_1 - \Delta Z_1)Z_2 \,dx + \int_{\partial \Omega_\e} Z_2 \frac{\partial Z_1}{\partial \nu}\,d \sh^1\\
				&\quad = \int_{\partial B_1^\e} \frac{\partial \Psi_1}{\partial \nu_1}\Psi_2 \,d \sh^1 + \int_{\partial B_2^\e} \frac{\partial \Psi_1}{\partial \nu_2}\,d \sh^1 \\
				&\quad \quad - \int_{\partial B_1^\e} \Psi_2 \frac{\partial \Psi_1}{\partial \nu_1} \,d \sh^1 - \int_{\partial B_2^\e} \Psi_1 \frac{\partial \Psi_2}{\partial \nu_2}\,d \sh^1 - \int_{\partial B_1^\e}\Psi_2  \frac{\partial Z_1}{\partial \nu_1}\,d \sh^1\\
				&\quad = \int_{\partial B_2^\e} \frac{\partial \Psi_1}{\partial \nu_2}\,d \sh^1 - \int_{\partial B_2^\e} \Psi_1 \frac{\partial \Psi_2}{\partial \nu_2}\,d \sh^1 - \int_{\partial B_1^\e}\Psi_2  \frac{\partial Z_1}{\partial \nu_1}\,d \sh^1.
			\end{aligned}
		\end{equation}
	\end{proof}
	Next, we have a lemma that controls the normal derivative of the function $Z$ in $H^{-1/2}(\partial \Omega_\e)$ by the energy. The underlying subtlety, is of course, that the domain $\Omega_\e$ varies in $\e,$ and we must obtain estimates that are uniform in $\e.$ 
	\begin{lemma} \label{l.normal.der}
		The functions $Z_j$ have $H^1$ norms bounded by 
		\begin{align}
			%\left\|\frac{\partial Z_j}{\partial \nu}\right\|_{H^{-\sfrac{1}{2} }(\partial \Omega_\e)} \leqslant
			\biggl\{ \int_{\Omega_\e} |Z_j|^2 + |\nabla Z_j|^2 \,dx \biggr\}^{\sfrac{1}{2}}{\lesssim \frac{1}{\sqrt{\e}}K_0(2b) {\lesssim \frac{e^{-2b}}{b\sqrt{\e}}}},
		\end{align}
		 where the last inequality holds for $b \gg 1$  and $\e^2 b\ll 1$.
	\end{lemma}
	\begin{proof}
		% By definition of the $H^{-\sfrac 12}$ norm, we have 
		% \begin{equation} \begin{aligned}
				%   &\biggl\| \frac{\partial Z_j}{\partial \nu} \biggr\|_{H^{-\sfrac{1}{2}}(\partial \Omega_\e)} = \sup_{\phi \in H^1(\Omega_\e), \|\phi\|_{H^1(\Omega_\e)} \leqslant 1} \biggl[ \int_{\Omega_\e} \nabla Z_j \cdot \nabla \phi + \Delta Z_j \phi \,dx \biggr]\\
				%     &\quad =  \sup_{\phi \in H^1(\Omega_\e), \|\phi\|_{H^1(\Omega_\e)} \leqslant 1} \biggl[ \int_{\Omega_\e} \nabla Z_j \cdot \nabla \phi + Z_j \phi \,dx \biggr] \\
				%     &\quad \leqslant \|Z_j \|_{H^1(\Omega_\e)} = \sqrt{2E_\e(Z_j)}.
				% \end{aligned}
			% \end{equation}
		% where we used Cauchy-Schwarz. 
		We first make the observation that the functions $Z_j$ satisfying the PDE \eqref{e.ZjPDE} are the unique minimizers of the $H^1$ norm, subject to their own boundary conditions. Therefore, the desired estimate follows by the construction of a competitor and comparing energies. %\\
		%\emph{Construction of competitor and its energy:} 
		Without loss of generality, we fix $j = 1$. Our competitor $\zeta \in H^1(\Omega_\e)$ must be constructed satisfying the boundary conditions for $\zeta = Z_1$ on $\partial \Omega_\e,$ so that $\zeta = 0$ on $\partial B_1$ and $\zeta = - \Psi_1$ on $\partial B_2.$ We let $\eta: (0,\infty) \to [0,1]$ be a $C^1$ function that satisfies $\eta(t) \equiv 1$ for $t \in \bigl[ \sfrac{1}{\e^2}, \sfrac{1}{\e^2} + \sfrac{b}{2}\bigr],$ $\eta(t) \equiv 0$ when $t \geqslant \sfrac{1}{\e^2} + b,$ and $|\eta^\prime| \leqslant \frac{2}{b},$ and set 
		\begin{align*}
			\zeta(x) := - \Psi_1(x) \eta(|x - a_2^\e|)\,,
		\end{align*}
  where we recall that $a_2^\e$ is the center of $B_2^\e.$
		Then 
		\begin{align*}
			\nabla \zeta = - \eta(|x-a_2^\e|) \nabla \Psi_1(x) - \Psi_1(x) \eta^\prime(|x-a_2^\e|)\frac{x-a_2^\e}{|x-a_2^\e|},
		\end{align*}
		so that, pointwise, we have the bound 
		\begin{align*}
			|\nabla \zeta(x)| \leqslant  |\nabla \Psi_1(x)| + \frac{2}{b} |\Psi_1(x)|,
		\end{align*}
		with support in the set $\sfrac{1}{\e^2}\leqslant |x-a_2^\e| \leqslant \sfrac{1}{\e^2} + b.$ Then, the energy of $\zeta$ is easily calculated: 
		\begin{equation} \label{e.energyofzeta}
			\begin{aligned}
				&F_\e(\zeta) \lesssim \int_{\sfrac{1}{\e^2}\leqslant |x-a_2^\e| \leqslant \sfrac{1}{\e^2} + b} \biggl( \frac{C}{b^2}|\Psi_1|^2 + |\nabla \Psi_1|^2\biggr) \,dx \\
				&\quad =  \biggl|\frac{C}{b^2} \int_{\sfrac{1}{\e^2}\leqslant |x-a_2^\e| \leqslant \sfrac{1}{\e^2} + b}|\Psi_1|^2 \,dx + \int_{|x-a_2^\e| = \sfrac{1}{\e^2}} \Psi_1(x) \frac{\partial \Psi_1(x)}{\partial \nu} \,d \sh^1\\
				&\quad \quad \quad \quad  - \int_{|x-a_2^\e| = \sfrac{1}{\e^2} + b} \Psi_1(x) \frac{\partial \Psi_1(x)}{\partial \nu} \,d \sh^1 \biggr|,
			\end{aligned}
		\end{equation}
		where we plugged in the PDE satisfied by $\Psi_1$ and integrated by parts as before; the signs in front of the boundary integrals reflect our choice that the corresponding unit normals point towards $a_2^\e.$ Each of these integrals are  explicitly estimated {using properties of Bessel functions~$K_0$}, and the triangle inequality then implies that 
		\begin{equation*}
			|F_\e(\zeta)| \lesssim \frac{K_0^2(2b)}{\e}\,.
		\end{equation*}
		This completes the proof  {using the large argument asymptotics of~$K_0$ (see \eqref{eq.besseldecay}).}  
	\end{proof}
	In the next Lemma, we use Lemma \ref{l.normal.der} to control the boundary integral on the right-hand side of \eqref{e.diag}. 
	\begin{lemma}
		\label{l.diag.terms} For all $b \gg 1$  and $\e^2 b\ll 1$, we have the estimate 
		\begin{equation} \label{e.quadtermbnd.diag}
			\begin{aligned}
				\biggl|  \int_{\partial B_2^\e} \Psi_1 \biggl( \frac{\partial \Psi_1}{\partial \nu_2} + \frac{\partial Z_1}{\partial \nu_2}\biggr)\biggr| {\lesssim\frac{K_0^2(2b)}{\e}  {\lesssim \frac{e^{-4b}}{b\e}}}\,.
			\end{aligned}
		\end{equation}
	\end{lemma}
	\begin{proof}
		\emph{Step 1.} By the triangle inequality, 
		\begin{equation*}
			\begin{aligned}
				\biggl|  \int_{\partial B_2^\e} \Psi_1 \biggl( \frac{\partial \Psi_1}{\partial \nu_2} + \frac{\partial Z_1}{\partial \nu_2}\biggr)\biggr| \leqslant \biggl| \int_{\partial B_2^\e} \Psi_1 \frac{\partial \Psi_1}{\partial \nu_2}\,d \sh^1\biggr| + \biggl| \int_{\partial B_2^\e} \Psi_1 \frac{\partial Z_1}{\partial \nu_2}\,d\sh^1 \biggr| =: R_1 + R_2. 
			\end{aligned}
		\end{equation*}
		The previous lemma shows that the term $R_1$ is controlled by $C\frac{K_0^2(2b)}{\e}, $ so that the proof of the Lemma is completed if we show the same bound for the term $R_2$. 
		
		\emph{Step 2.} First we make the observation that the prescribed boundary conditions on $Z_1$ imply that 
		\begin{align*}
			\int_{\partial B_2^\e} \Psi_1 \frac{\partial Z_1}{\partial \nu_2}\,d\sh^1=\int_{\partial B_2^\e} -Z_1 \frac{\partial Z_1}{\partial \nu_2}\,d\sh^1=-\int_{\partial \Omega_\e} Z_1 \frac{\partial Z_1}{\partial \nu_2}\,d\sh^1
		\end{align*}
		Then as $\Delta Z_1=Z_1$, integrating by parts we see that 
		\begin{align*}
			\int_{\partial \Omega^\e} Z_1 \frac{\partial Z_1}{\partial \nu_2}\,d\sh^1=\int_{\Omega^\e}|\nabla Z_1|^2+|Z_1|^2\,dx= \|Z_1\|_{H^1(\Omega^\e)}^2.
		\end{align*}
		Thus $R_2= \|Z_1\|_{H^1(\Omega^\e)}^2$, and {by} Lemma \ref{l.normal.der} we thus have that 
		
		\begin{align*}
			R_2= \|Z_1\|_{H^1(\Omega^\e)}^2{\lesssim \frac{K_0(2b)^2}{\e} {\lesssim\frac{e^{-4b}}{b\e}}}\,.
		\end{align*}
         {This completes the proof using large argument asymptotics of~$K_0$ \eqref{eq.besseldecay}.}

		% In this step we use Lemma \ref{l.normal.der} to control the term $R_2.$ By that lemma,
		% \begin{align*}
			%     |R_2| \leqslant \|\Psi_1\|_{H^{\sfrac{1}{2}}(\partial B_2)} \Bigl\|\frac{\partial Z_1}{\partial \nu_2}\Bigr\|_{H^{-\sfrac12}(\partial \Omega)} \leqslant C \frac{K_0(b)}{b\e} \|\Psi_1\|_{H^{\sfrac12}(\partial B_2)},
			% \end{align*}
		% so our task is to estimate $\|\Psi_1\|_{H^{\sfrac12}(\partial B_2)}.$ For this, recalling that 
		% \begin{align*}
			%     &\|\Psi_1\|_{H^{\sfrac12}(\partial B_2)}^2 = \inf_{f \in H^1(\RR^2 \setminus \overline{B}_2): \mathrm{tr}\, f|_{\partial B_2} = \Psi_1} \int_{\RR^2 \setminus \overline{B}_2} (f^2 + |\nabla f|^2 )\,dx \\
			%     &\quad \leqslant \int_{\RR^2 \setminus \overline{B}_2} |\Psi_1|^2 + |\nabla \Psi_1|^2\,dx \\
			%     &\quad = \int_{\partial B_2} \Psi_1 \frac{\partial \Psi_1}{\partial \nu_2}\,d \sh^1 \leqslant C \frac{K_0^2(b)}{b^2\e^2},
			% \end{align*}
		% by \eqref{e.bdryterm1}.

		%%%%%%%%%%%%%%%%%%%%%%%%%%%%%%%%%%detailed proof 
		
		%%%%%%%%%%%%%%%%%%%%%%%%%%%%%%%%%%%%%%%%%%%%%%%%%%%%%%%% end of old detailed proof
	\end{proof}
	\begin{lemma}
		\label{l.offdiag.terms}
		 {Assume $b\gg 1$ and $b\e^2 \ll 1$. }The off-diagonal terms from \eqref{e.offdiag} have the asymptotic expansion
		\begin{multline} \label{l.statementoffdiag}
				\Biggl|	\int_{\partial B_2^\e} \frac{\partial \Psi_1}{\partial \nu_2}\,d \sh^1 - \int_{\partial B_2^\e} \Psi_1 \frac{\partial \Psi_2}{\partial \nu_2}\,d \sh^1 - \int_{\partial B_1^\e}\Psi_2  \frac{\partial Z_1}{\partial \nu_1}\,d \sh^1 - \frac{2e^{-2b} \sqrt{\pi}}{\e}\Biggr| \\
                {\lesssim \frac{e^{-4b}}{b \e} {+b\e^2 \frac{e^{-2b}}{\e} + \e}}\,. 
		\end{multline}
	\end{lemma}
	\begin{proof}
		\emph{Step 1.} We proceed by a similar argument to Step 2 of Lemma \ref{l.diag.terms}. First we note that \begin{align*}
			\int_{\partial B_1^\e} \Psi_2 \frac{\partial Z_1}{\partial \nu_1}\,d \sh^1 = \int_{\partial B_1^\e} -Z_2 \frac{\partial Z_1}{\partial \nu_1}\,d \sh^1=\int_{\partial \Omega_\e} -Z_2 \frac{\partial Z_1}{\partial \nu_1}\,d \sh^1
		\end{align*}
		Thus we may estimate this via 
		\begin{align*}
			\left|\int_{\partial B_1^\e} \Psi_2 \frac{\partial Z_1}{\partial \nu_1}\,d \sh^1\right|\leqslant \|Z_2\|_{H^{\sfrac{1}{2}}(\partial \Omega)_\e} \Bigl\|\frac{\partial Z_1}{\partial \nu}\Bigr\|_{H^{-\sfrac12}(\partial \Omega_\e)} .
		\end{align*}
		
		We now utilise the fact that $\Delta Z_i = Z_i$ and the definition of the $H^{\sfrac{1}{2}}$ norm to conclude that $\|Z_2\|_{H^{\sfrac{1}{2}}(\partial \Omega_\e)}=\|Z_2\|_{H^{1}(\Omega_\e)}$, and we use Proposition \ref{prop.normal.deriv.equal.h1} to conclude that $\Bigl\|\frac{\partial Z_1}{\partial \nu}\Bigr\|_{H^{-\sfrac12}(\partial \Omega)}=\|Z_1\|_{H^{1}(\Omega_\e)}$. Thus by taking the estimations of the $H^1$ norms of $Z_1,Z_2$ from Lemma \ref{l.normal.der}, we have that 
		\begin{align*}
			\left|\int_{\partial B_1^\e} \Psi_2 \frac{\partial Z_1}{\partial \nu_1}\,d \sh^1\right|{\lesssim \frac{K_0(2b)^2}{\e} {\lesssim\frac{e^{-4b}}{b\e}}}.
		\end{align*}
		\emph{Step 2.} Towards evaluating the first two terms, first we write $\partial B_2^\e = \partial B_2^{\e+} \cup \partial B_2^{\e-},$ where the $\pm$ respectively denote the upper/lower hemispheres (i.e., $y > -b - \frac{1}{\e^2}$ and $y < -b -\frac{1}{\e^2}$ respectively). It is clear that the contribution of $\partial B_2^{\e-}$ is exponentially small by prior arguments, so we focus on the contribution of $\partial B_2^{\e+}$ from the first two terms. We parameterize $\partial B_2^{\e+}$ as a graph over the $x-$ axis: 
		\begin{align*}
			y = - b - \frac{1}{\e^2} + \sqrt{\frac{1}{\e^4} - x^2}, \quad \quad |x| \leqslant \frac{1}{\e^2},
		\end{align*}
		where we choose the positive square root since we want the upper semicircle of $\partial B_2^\e.$ We note that for $(x,y) \in \partial B_2^\e,$
		\begin{align*}
			\nu_2 = -\e^2 \biggl(x,y + b + \frac{1}{\e^2}\biggr).
		\end{align*}
		As a sign check, we note that  at $x = 0, y = -b$ the normal $\nu_2 = -\e^2 (0, \frac{1}{\e^2}) = (0,-1).$
		Since 
		\[\nabla \Psi_1(x,y) = \frac{1}{K_0 \bigl( \frac{1}{\e^2}\bigr)} K_0^\prime\biggl( \sqrt{x^2 + \bigl(y - b - \frac{1}{\e^2})^2\bigr)}\biggr) \frac{(x,y - b - \frac{1}{\e^2})}{\sqrt{x^2 + \bigl(y - b - \frac{1}{\e^2}\bigr)^2}}\,,\]
		we arrive at
		\begin{equation} \label{e.nromalder}
			\begin{aligned}
				\nu_2 \cdot \nabla \Psi_1 & = -\frac{\e^2}{K_0( \tfrac{1}{\e^2})} K_0^\prime\biggl( \biggl\{x^2 + \biggl\{2b + \frac{2}{\e^2} -  \sqrt{\frac{1}{\e^4} - x^2} \biggr\}^2\biggr\}^{\sfrac12} \biggr)\\
				& \quad \quad \times \left[ \frac{ x^2 + \biggl(b + \frac{1}{\e^2} - \sqrt{\frac{1}{\e^4} - x^2} \biggr)^2 - (b+ \frac{1}{\e^2})^2}{ \biggl\{x^2 + \biggl\{2b + \frac{2}{\e^2} -  \sqrt{\frac{1}{\e^4} - x^2} \biggr\}^2\biggr\}^{\sfrac{1}{2}} } \right],
			\end{aligned}
		\end{equation}
		holding for all $(x,y) \in \partial B_2^\e.$ Note, with the parametrization $x \mapsto \bigl(x, -b - \frac{1}{\e^2} + \sqrt{\frac{1}{\e^4} - x^2}\bigr),$ the speed of the curve $\partial B_2^{\e+}$ is $\sqrt{1 +  \bigl(\frac{x}{\sqrt{\frac{1}{\e^4}-x^2}}\bigr)^2} = \bigl( \frac{1}{1-\e^4 x^2}\bigr)^{\sfrac12}.$ 
		
		\smallskip
		
		\emph{Step 3.} We will split the integral as $[0,\frac{M}{\e}],$  and between $[\frac{M}{\e},\frac{1}{\e^2}],$ for some $M \leqslant \frac{\sqrt{3}}{2} \frac1{\e} < \frac1{\e},$ to be fixed later. We compute each of these contributions separately.  For the first integral in \eqref{l.statementoffdiag} we obtain 
		\begin{equation} \label{e.firstterm}
			\int_{\partial B_2^{\e+}} \frac{\partial \Psi_1}{\partial \nu_2}\,d \sh^1 = \biggl(2\int_{0}^{\frac{M}{\e}} + \int_{|x| \geqslant \frac{M}{\e}} \biggr) \frac{\partial \Psi_1}{\partial \nu_2}\,d \sh^1\,.
		\end{equation}
		Towards computing it we note by the binomial theorem that for $|x| \leqslant \frac{M}{\e}$,
		\begin{align*}
			\biggl\{x^2 + \biggl\{2b + \frac{2}{\e^2} -  \sqrt{\frac{1}{\e^4} - x^2} \biggr\}^2\biggr\}^{\sfrac12} &= \biggl\{ \frac{1}{\e^4}  +  \frac{4}{\e^4}(1 + b \e^2)\Bigl( 1 + b \e^2 - \sqrt{1 - \e^4 x^2} \Bigr)  \biggr\}^{\sfrac12} \\
			&=: \biggl\{ \frac1{\e^4} + C_M(x) \biggr\}^{\sfrac12}
		\end{align*}
		where the function $C_M$  satisfies 
		\begin{enumerate}
			\item[(i)] $C_M(0) = \frac{4 b}{\e^2}(1 + b\e^2),$ 
			\item[(ii)] $C_M(\frac{\sqrt{3}}{2\e }) \geqslant \frac{3}{\e^4},$ so that $\sqrt{\e^{-4} + C_M(x)} \geqslant \frac{2}{\e^2}$ for $x =  \frac{\sqrt{3}}{2\e },$ 
			and \item[(iii)] $C_M(x)$ is an increasing function for all $x \in [0,\frac1{\e^2}],$ so that if $x \geqslant  \frac{\sqrt{3}}{2\e },$ then the preceding quantity, i.e., $\sqrt{\e^{-4} + C_M(x)} \geqslant \frac{2}{\e^2}.$ 
		\end{enumerate}
		In addition, we have 
		\begin{align*}
			&\frac{ x^2 + \biggl(b + \frac{1}{\e^2} - \sqrt{\frac{1}{\e^4} - x^2} \biggr)^2 - (b+ \frac{1}{\e^2})^2}{ \biggl\{x^2 + \biggl\{2b + \frac{2}{\e^2} -  \sqrt{\frac{1}{\e^4} - x^2} \biggr\}^2\biggr\}^{\sfrac{1}{2}} } \frac{1}{\sqrt{1 - \e^4 x^2}} \\
			&\quad \quad = \frac{- \frac{1}{\e^4} + \frac{2}{\e^4}\bigl(1 - (1 + b\e^2) \sqrt{1 - \e^4 x^2} \bigr)}{ \sqrt{\frac1{\e^4} + C_M(x)}}  \frac{1}{\sqrt{1 - \e^4 x^2}}\,.
		\end{align*}
		It follows then that the first term in \eqref{e.firstterm} contributes (see \eqref{e.nromalder})
		\begin{align*}
			&2 \int_0^{\frac{M}{\e}} \frac{\partial \Psi_1}{\partial \nu_2} \,d \sh^1 \\ &\quad = - 2\e^2 \int_0^{\frac{M}{\e}} \frac{1}{K_0(\frac{1}{\e^2})}K_0^\prime \biggl( \sqrt{\frac{1}{\e^4} + C_M(x)}\biggr) \frac{- \frac{1}{\e^4} + \frac{2}{\e^4}\bigl(1 - (1 + b\e^2) \sqrt{1 - \e^4 x^2} \bigr)}{ \sqrt{\frac1{\e^4} + C_M(x)}}  \frac{1}{\sqrt{1 - \e^4 x^2}}\,dx\\
			&\quad = - 2 \int_0^{\frac{M}{\e}} \frac{1}{K_0(\frac{1}{\e^2})} K_0^\prime \biggl( \frac{1}{\e^2}\sqrt{1 + \e^4 C_M(x)}\biggr) \frac{-1 + 2\bigl(1 - (1 + b\e^2) \sqrt{1 - \e^4 x^2} \bigr)}{\sqrt{1 + \e^4 C_M(x)}} \frac{1}{\sqrt{1 - \e^4 x^2}}\,dx \\
			&\quad = 2 \int_0^{\frac{M}{\e}}  \frac{1}{K_0(\frac{1}{\e^2})} K_0^\prime \biggl( \frac{1}{\e^2}\sqrt{1 + \e^4 C_M(x)}\biggr) \,dx + R_1,
		\end{align*}
		where the remainder $R_1$ satisfies 
		\begin{equation*}
			|R_1| \leqslant 2 b\e^2 \Biggl|\int_0^{\frac{M}{\e}}   \frac{1}{K_0(\frac{1}{\e^2})} K_0^\prime \biggl( \frac{1}{\e^2}\sqrt{1 + \e^4 C_M(x)}\biggr) \,dx\Biggr| \,.
		\end{equation*}
		 {As~$b \e^2 \ll 1,$ this is a remainder term. } Finally, {by using~\eqref{e.besselderivatives} and Lemma~\ref{l.besselfct1}} , we obtain 
		\begin{align*}
			&2\biggl| \int_0^{\frac{M}{\e}} \frac{1}{K_0(\frac{1}{\e^2})} K_0^\prime \biggl( \frac{1}{\e^2}\sqrt{1 + \e^4 C_M(x)}\biggr) \,dx\biggr|\\
			&\quad \approx 2 \biggl| \int_0^{\frac{M}{\e}} \exp \Bigl( \frac{1}{\e^2} \bigl( 1 - \sqrt{1 + \e^4 C_M(x)}\bigr) \Bigr)\sqrt{\frac{1}{\sqrt{1 + \e^4C_M(x)}}}\,dx \biggr|\,.
		\end{align*}
		Inserting the definition of $C_M(x) = \frac{4}{\e^4} (1 + b\e^2) \bigl( 1 + b\e^2 - \sqrt{1 - \e^4 x^2} \bigr),$ and making the change of variables $s = \e^2 x,$ we find that the last integral simplifies to
		\begin{align*}
			&	\frac{2}{\e^2} \int_0^{M \e} \!\!\exp \Bigl( \frac{- \frac{4}{\e^2}(1 + b\e^2) (1 + b\e^2 - \sqrt{1 - s^2})}{1 + \sqrt{1 + 4(1 + b\e^2)(1 + b\e^2 - \sqrt{1 - s^2})}}\Bigr)   \Bigl( 1 {+} 4 (1 {+} b\e^2)(1 {+} b\e^2 {-} \sqrt{1 {-} s^2})\Bigr)^{-\frac14}       \,ds\\
			&\quad \approx \frac{2}{\e^2}\int_0^{M\e} \exp\Bigl( - \frac{2}{\e^2}(b\e^2  + \frac12 s^2)\Bigr)\,ds \\
			&\quad = \frac{2e^{-2b}}{\e^2}\int_0^{M\e} \exp \Bigl(- \frac{s^2}{\e^2}\Bigr)\,ds \\
			&\quad = \frac{2e^{-2b}}{\e} \int_0^{M} \exp(-s^2)\,ds = \frac{e^{-2b}\sqrt{\pi}}{\e} \Bigl[ 1 - \frac{2}{\sqrt{\pi}} \int_M^\infty \exp(-s^2)\,ds \Bigr] = \frac{e^{-2b}\sqrt{\pi}}{\e} \bigl(1 - \mathrm{erfc} (M) \bigr) \,.
		\end{align*}
		In each of the preceding two displays, the $\approx$ (approximate) sign means that the left-hand and right-hand sides differ by $O(\e).$
		To conclude the computation of the leading term in \eqref{e.firstterm}, we observe that the complimentary error function $\mathrm{erfc}$ satisfies the asymptotics 
		\begin{equation*}
			\Bigl| \mathrm{erfc}(M)  - \frac{e^{-M^2}}{M \sqrt{\pi}} \Bigr| \leqslant CM^{-3} e^{-M^2}\,.
		\end{equation*}
		At this point, we must choose $M$ so that $C M^{-3}e^{-M^2} \ll 1$ as $\e \to 0.$  This is, for example, guaranteed with the choice\footnote{the precise prefactor is not important, but is chosen to simplify the arithmetic in our computation of the tails}  $M = \frac{\sqrt{3}}{2} \frac{1}{\e},$ so that  combining the preceding displays we find 
		\begin{equation*}
			\Biggl|	2 \int_0^{\frac{M}{\e}} \frac{\partial \Psi_1}{\partial \nu_2} \,d \sh^1 - \frac{\sqrt{\pi} e^{-2b}}{\e} \Biggr| {\lesssim \e +   \frac{e^{-2b}}{\e^4}e^{-\frac{C}{\e^2}} \lesssim \e}\,. 
		\end{equation*}
		For the tail term in \eqref{e.firstterm}, noting from the properties of the function $C_M,$ that $$\sqrt{\e^{-4} + C_M(\frac{\sqrt{3}}{2\e})} \geqslant \frac{2}{\e^2},$$ and that $C_M$ is increasing, we find 
		\begin{equation*}
			\biggl| \int_{|x| \geqslant \frac{M}{\e}} \frac{\partial \Psi_1}{\partial \nu_2}\,d \sh^1\biggr| \leqslant C \int_{\frac{M}{\e}}^{\frac{1}{\e^2}} \exp \bigl( \frac1{\e^2} (1 - \sqrt{1 + \e^4 C_M(x)})\bigr)\,dx \leqslant \frac{1}{\e^2} \exp\biggl( - \frac{1}{\e^2} \biggr) \leqslant C \e\,,
		\end{equation*}
		for all $\e$ small enough. 
		
		\smallskip 
		
		\emph{Step 4.} The last step in the proof is to evaluate the asymptotics of the term 
		\begin{equation*}
			\int_{\partial B_2^\e} \Psi_1 \frac{\partial \Psi_2}{\partial \nu_2}\,d \sh^1\,.
		\end{equation*}
		This is easier than Step 3, since 
		\begin{equation*}
			\frac{\partial \Psi_2}{\partial \nu_2}= -\frac{1}{K_0\bigl(\frac{1}{\e^2}\bigr)} K_0^\prime \Bigl(\frac{1}{\e^2}\Bigr) \,.
		\end{equation*}
		Consequently, we find, {using~\eqref{e.besselderivatives} and Lemma~\ref{l.besselfct1}}, that
		\begin{align*}
			&\int_{\partial B_2^\e} \Psi_1 \frac{\partial \Psi_2}{\partial \nu_2}\,d \sh^1 \\
			&\quad = -2\frac{1}{K_0\bigl(\frac{1}{\e^2}\bigr)} K_0^\prime \Bigl(\frac{1}{\e^2}\Bigr) \int_0^{\frac{1}{\e^2}} \frac{1}{K_0\bigl(\frac{1}{\e^2}\bigr)} K_0 \Bigl(\frac{1}{\e^2}\sqrt{1 + \e^4 C_M(x)} \Bigr) \frac{1}{\sqrt{1 - \e^4 x^2}}\,dx  + \frac{1}{\e^2}\exp \Bigl( - \frac{1}{\e^2}\Bigr)\\
			&\approx \frac{e^{-2b}\sqrt{\pi}}{\e} + C\e + \frac{1}{\e^2}\exp \Bigl( - \frac{1}{\e^2}\Bigr)\,.
		\end{align*}
		The proof of the proposition is completed by combining Steps 1 through 4. 
	\end{proof}

	\begin{proof}[Proof of Theorem \ref{t.2balls.constbc}]
		 {Theorem \ref{t.2balls.constbc} follows by taking the representation formula of \eqref{e.decompofenergy}, with the asymptotics of the diagonal terms obtained in Lemma \ref{l.diag.terms} under the observation that $b\gg 1$ and $b\e^2\ll 1$, $\left|\frac{K_0'(\e^{-2})}{K_0(\e^{-2})}+1\right|\leqslant C\e^2$. The off-diagonal terms are obtained similarly in Lemmas \ref{l.boundaryintegrals1} and \ref{l.offdiag.terms}.}
	\end{proof}

	\subsection{Interaction energies for O(1) separation between particles}
 \label{s.O1separation} 
In this section we use formal asymptotics to compute the energy of interaction between two particles when $b=O(1)$ and $\e\to0^{+}$. This amounts to computing various terms in \eqref{e.diag} and \eqref{e.offdiag}. Since the problem is rotationally invariant, in this section we find it convenient to orient the particles horizontally (Fig.~\ref{fig:setup}), rather than vertically. Given $b,\e>0,$ consider two disk-like particles $B_1^\e$ and $B_2^\e$ of radius $1/\e^2$ where the first particle is centered at the origin and the distance between the particles is equal $2b$ as shown in Fig.~\ref{fig:setup}.

 \begin{figure}[ht]
    \centering
    \includegraphics[width=3.5in]{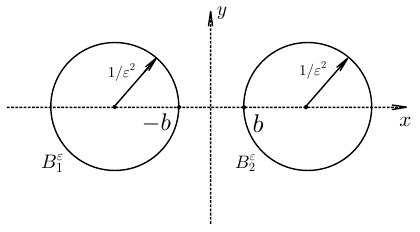}
    \caption{Geometry of the problem}
    \label{fig:setup}
\end{figure}

We begin by introducing polar coordinates associated with the center of the particle $B_1^\e$ so that
\[x=-\e^{-2}-b+\rho\cos{\theta},\quad y=\rho\sin{\theta},\]
then
\[\partial B_1^\e=\left\{(\rho,\theta)\left|\,\rho=\e^{-2},\,\theta\in[0,2\pi)\right.\right\}\] and \[\partial B_2^\e=\left\{(\rho,\theta)\left|\,\rho^2-4(\e^{-2}+b)\rho\cos{\theta}+4{(\e^{-2}+b)}^2-\e^{-4}=0,\,\theta\in[0,2\pi)\right.\right\}.\]
Now suppose that $\e\ll1$ and let $t:=\e^{-1}\,\theta.$ Then, if $t=O(1),$ we have $\theta=O\left(\e\right)$ so that \[\cos{\theta}=1-\frac{\e^2t^2}{2}+O\left(\e^{4}\right)\]
and the equation for $\partial B_2^\e$ is 
\[\rho^2-4(\e^{-2}+b)\rho\left(1-\frac{\e^2t^2}{2}\right)+4{(\e^{-2}+b)}^2-\e^{-4}=0,\]
up to the order $O\left(\e\right).$ Solving this equation for $\rho$, gives an asymptotic expression for the boundary of the right disk, i.e.,
\begin{equation}
\label{eq:stw}
\rho|_{\partial B_2^\e}=\e^{-2}+2b+t^2,
\end{equation}
valid up to $O\left(\e\right),$ while the boundary of the left disk is given by
\[\rho|_{\partial B_1^\e}=\e^{-2}.\]
We now solve the problem \eqref{eq:sipa} for the left disk
\begin{equation}
\label{eq:as1}
\left\{\begin{aligned}&-\Delta \Psi_1+\Psi_1=0\mbox{ in }\mathbb{R}^2\backslash B_1^\e,\\&\Psi_1|_{\partial B_1^\e}=1.\end{aligned}\right.
\end{equation}
Using the radially symmetric ansatz $\Psi_1=\Psi_1(\rho)$ and the fact that $\rho>\e^{-2}\gg1,$ we find that 
\begin{equation}
\label{eq:psi1}
\Psi_1(\rho)\sim e^{\e^{-2}-\rho}.
\end{equation}
to leading order. Note that this result matches the expression \eqref{eq:psi} for small $\e.$ Therefore
\begin{equation}
\label{eq:psi1b2}
\Psi_1|_{\partial B_2^\e}=\Psi_1\left(\e^{-2}+2b+t^2\right)=e^{-2b-t^2}
\end{equation}
and this expression decays exponentially fast in $t$ as one moves away from the point on $\partial B_2^\e$ that is the closest to $\partial B_1^\e.$

Now we can solve the problem \eqref{e.ZjPDE}. We have
\begin{equation}
\label{eq:3.8bis}
\left\{
\begin{array}{ll}
-Z_{1,\rho\rho}-\frac{1}{\rho}Z_{1,\rho}-\frac{1}{\e^2\rho^2}Z_{1,tt}+Z_1=0,&(\rho,t)\in\Omega_\e\\Z_1\left(\rho,\e t\right)=0,&\rho=\e^{-2},\\Z_1\left(\rho,\e t\right)=-e^{-2b-t^2},&(\rho,t)\in\partial B_2^\e.
\end{array}
\right.
\end{equation}
Because $\rho>\e^{-2},$ to leading order the equation in \eqref{eq:3.8bis} takes the form
\[-Z_{1,\rho\rho}+Z_1=0,\]
with the general solution
\[Z_1=C_1e^{-\rho}+C_2e^\rho.\]
Substituting this $Z_1$ into the boundary conditions from \eqref{eq:3.8bis} gives
\[\left\{\begin{array}{l}C_1e^{-\e^{-2}}+C_2e^{\e^{-2}}=0,\\C_1e^{-\e^{-2}-2b-t^2}+C_2e^{\e^{-2}+2b+t^2}=-e^{-2b-t^2},\end{array}\right.\]
so that
\[C_1=\frac{e^{\e^{-2}}}{e^{4b+2t^2}-1},\quad C_2=-\frac{e^{-\e^{-2}}}{e^{4b+2t^2}-1},\]
hence
\begin{equation}
\label{eq:spq0}
Z_1=\frac{2\sinh{(\e^{-2}-\rho)}}{e^{4b+2t^2}-1}
\end{equation}
solves \eqref{eq:3.8bis} to leading order in $\e.$ It then follows that
\begin{equation}
\label{eq:spq}
\left.\frac{\partial Z_1}{\partial \nu_{1}}\right|_{\partial B_1^\e}=-\left.\frac{\partial Z_1}{\partial \rho}\right|_{\partial B_1^\e}=\frac{2}{e^{4b+2t^2}-1}.
\end{equation}

Next, we observe that 
\[\Psi_1\left(\rho(\theta)\right)=\Psi_2\left(\bar\rho(\bar\theta)\right)\]
(cf. Fig.~\ref{fig:setup1}).
\begin{figure}
    \centering
    \includegraphics[width=3.5in]{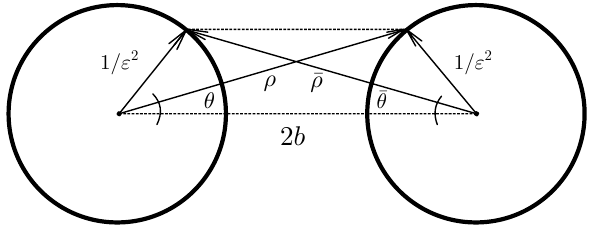}
    \caption{Setup for the calculation of $z$}
    \label{fig:setup1}
\end{figure}
Recalling that $\theta=\e t,$ we set $\bar\theta=\e\bar{t},$ use the law of sines and \eqref{eq:stw} to find that
\[\frac{\e^{-2}+2b+t^2}{\sin{(\e\bar{t})}}=\frac{1}{\e^2\sin{(\e t)}},\]
hence
\[\left(\e^{-2}+2b+t^2\right)t=\e^{-2}\bar{t},\]
because $\e^{-2}t,\e^{-2}\bar{t}\ll1.$ We conclude that
\[\bar{t}=t+O\left(\e^2\right)\]
and
\begin{equation}
\label{eq:spr}
\Psi_2|_{\partial B_1^\e}=\Psi_1\left(\rho\left(\e t\right)\right)+O\left(\e^2\right)=e^{-2b-t^2}+O\left(\e^2\right).
\end{equation}

We are now ready to determine the $\e\to0^{+}$ asymptotics of terms that appear in (\ref{e.diag}-\ref{e.offdiag}) in a series of six Steps:

\smallskip\noindent {\it Step (1).} Using \eqref{eq:psi1}, we obtain 
\[\int_{\partial B_1}\frac{\partial\Psi_1}{\partial\nu_1}\sim{\frac{2\pi}{\e^2}}\]

\smallskip\noindent {\it Step(2).} From \eqref{eq:spq} and \eqref{eq:spr}, we deduce
\begin{multline}
\label{eq:psi2dz1dnu}
\int_{\partial B_1^\e}\Psi_2\frac{\partial Z_1}{\partial\nu_1}\,d \sh^1=\e^{-2}\int_{-\pi}^\pi\Psi_2\frac{\partial Z_1}{\partial\nu_1}\,d\theta=\e^{-1}\int_{-\pi/\e}^{\pi/\e}\Psi_2\frac{\partial {Z}_1}{\partial\nu_1}\,dt\\\sim \frac{2}{\e}\int_{-\infty}^{\infty}\frac{e^{-6b}e^{-3t^2}}{1-e^{-4b}e^{-2t^2}}\,dt\sim\frac{\sqrt{2\pi}}{\e}\Theta_3\left(e^{-2b}\right),
\end{multline}
where
\begin{equation}
\label{eq:theta}
\Theta_k(x):=\sqrt{\frac2\pi}\int_{-\infty}^\infty\frac{x^ke^{-kt^2}}{1-x^2e^{-2t^2}}\,dt,
\end{equation}
where $x\in(0,1)$ and $k\in\mathbb{N}$. Note that ${\Theta_2(x)=\Li\left(x^2\right)},$ where $\Li$ is the polylogarithmic function~({see~\cite{AS,Leb}}) and
\begin{equation}
\label{eq:thelim}
\lim_{x\to0}{\frac{\Theta_k(x)}{x^k}}=\sqrt{\frac2k}.
\end{equation}

\smallskip\noindent{\it Step(3).} Using \eqref{eq:psi1b2}, we have
\begin{multline}
\label{eq:psi1dpsi2dnu}
\int_{\partial B_2^\e}\Psi_1\frac{\partial \Psi_2}{\partial\nu_2}\,d \sh^1=\e^{-2}\frac{\partial \Psi_2}{\partial\nu_2}(\e^{-2})\int_{-\pi}^\pi\Psi_1\,d\theta\\=\frac{1}{\e^2}\frac{\partial \Psi_2}{\partial\nu_2}(\e^{-2})\int_{-\pi/\e}^{\pi/\e}\Psi_1\,dt\sim \frac{e^{-2b}}{\e^{2}}\int_{-\infty}^{\infty}e^{-t^2}\,dt=\frac{\sqrt{\pi}e^{-2b}}{\e}.
\end{multline}

\begin{figure}
    \centering
    \includegraphics[width=3.5in]{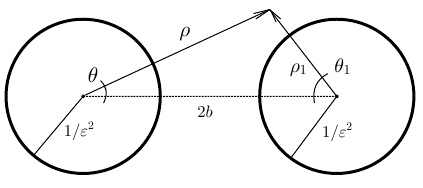}
    \caption{Setup for the calculation in (5).}
    \label{fig:setup2}
\end{figure}

\smallskip\noindent{\it Step (4).} Using Fig.~\ref{fig:setup2} and the law of cosines, we have
\begin{equation}
\label{eq:sph}
\rho^2=4{(\e^{-2}+b)}^2+\rho_1^2-4\rho_1(\e^{-2}+b)\cos{\theta_1},
\end{equation}
then $\Psi_1$ can be written as
\[
\bar{\Psi}_1(\rho_1,\theta_1)=\Psi_1(\rho(\rho_1,\theta_1),\theta(\rho_1,\theta_1))\sim e^{\e^{-2}-\sqrt{4{(\e^{-2}+b)}^2+\rho_1^2-4\rho_1(\e^{-2}+b)\cos{\theta_1}}}
\]
so that
\begin{multline*}
\frac{\partial\Psi_1}{\partial\nu_2}=-\left.\frac{\partial\bar{\Psi}_1}{\partial\rho_1}\right|_{\rho_1=\e^{-2}}\\\sim \frac{2(1+b\e^{2})\cos{\theta_1}-1}{\sqrt{4{(1+b\e^{2})}^2+1-4(1+b\e^{2})\cos{\theta_1}}}e^{\e^{-2}(1-\sqrt{4{(1+b\e^{2})}^2+1-4(1+b\e^{2})\cos{\theta_1}})}
\end{multline*}
and
\[\frac{\partial\Psi_1}{\partial\nu_2}\sim -e^{-2b-t_1^2}.
\]
Here the last step follows because
\[1-2\cos{\theta_1}\sim -1+\e^2 t_1^2\]
and
\[4{(\e^{-2}+b)}^2+\e^{-4}-4\e^{-2}(\e^{-2}+b)\cos{\theta_1}\sim \e^{-2}+2b+t_1^2,\]
where $t_1:=\theta_1/\e.$ Therefore
\begin{equation}
\label{eq:sps}
\int_{\partial B_2^\e}\frac{\partial\Psi_1}{\partial\nu_2}\,d \sh^1\sim-\frac{e^{-2b}}{\e}\int_{-\pi/\e}^{\pi/\e}e^{-t_1^2}\,dt_1\sim-\frac{\sqrt{\pi}}{\e}e^{-2b}
\end{equation}

\smallskip\noindent{\it Step(5).} Using \eqref{eq:psi1b2}, we find
\begin{equation}
\label{eq:spu}
\int_{\partial B_2^\e}\Psi_1\frac{\partial\Psi_1}{\partial\nu_2}\,d \sh^1\sim-\frac{e^{-4b}}{\e}\int_{-\pi/\e}^{\pi/\e}e^{-2t^2}\,dt\sim-\frac{\sqrt{\pi/2}}{\e}e^{-4b}
\end{equation}

\smallskip\noindent{\it Step(6).} The last term that we need to estimate is $\int_{\partial B_2^\e}\Psi_1\frac{\partial Z_1}{\partial\nu_2}.$ From \eqref{eq:spq0}, we have that
\[\bar{Z}_1(\rho_1,t_1)=Z_1\left(\rho(\rho_1,t_1),t(\rho_1,t_1)\right)\sim\frac{2\sinh{(\e^{-2}-\rho(\rho_1,t_1))}}{e^{4b+2t(\rho_1,t_1)^2}-1},\]
where $t_1=\theta_1/\e$ and $(\rho_1,\theta_1)$ are as in Fig.~\ref{fig:setup2}. Further, \[\rho(\rho_1,t_1)=\sqrt{4{(\e^{-2}+b)}^2+\rho_1^2-4\rho_1(\e^{-2}+b)\cos{\theta_1}}\] per \eqref{eq:sph} and
\[2(\e^{-2}+b)\sin{\theta}=\rho_1\sin{(\theta+\theta_1),}\]
by the law of sines, hence
\[\frac{\partial Z_1}{\partial \nu_2}=-\left.\frac{\partial \bar{Z}_1}{\partial \rho_1}\right|_{\rho_1=\e^{-2}}\sim-\frac{2e^{-4b-2t_1^2}}{1-e^{-4b-2t_1^2}}\cosh{(2b+t_1^2)}.\] It follows that
\begin{align}
\label{eq:spf}
\int_{\partial B_2^\e}\Psi_1\frac{\partial Z_1}{\partial\nu_2}\,d \sh^1&\sim -\frac{2}{\e}\int_{-\pi/\e}^{\pi/\e}\frac{e^{-6b-3t_1^2}}{1-e^{-4b-2t_1^2}}\cosh{(2b+t_1^2)}\\
&\sim-\frac{\sqrt{\pi/2}}{\e}\left(\Li\left(e^{-4b}\right)+\Theta_4\left(e^{-2b}\right)\right).
\end{align}
Now collecting the energy contributions in Steps (1) through (6) and using (\ref{e.diag}-\ref{e.offdiag}), we find that
\begin{multline}
\label{eq:expans}
	\kappa_\e=g_1^2{\|\Psi_1+z_1\|}^2+g_2^2{\|\Psi_2+z_2\|}^2+2g_1g_2\langle\Psi_1+z_1,\Psi_2+z_2\rangle\\=\frac{2\pi}{\e^2}\left(g_1^2+g_2^2\right)+\frac{\sqrt{\pi/2}}{\e}\left[\Li\left(e^{-4b}\right)+\Theta_4\left(e^{-2b}\right)+e^{-4b}\right](g_1^2+g_2^2)\\-\frac{4\sqrt{\pi}}{\e}\left[e^{-2b}+\frac{1}{\sqrt{2}}\Theta_3\left(e^{-2b}\right)\right]g_1g_2+O(1)\,.
\end{multline}
Note that, using \eqref{eq:thelim}, this expression reduces to
\begin{multline*}
\kappa_\e=\frac{2\pi}{\e^2}\left(g_1^2+g_2^2\right)+\frac{\sqrt{\pi/2}}{\e}\left[2e^{-4b}+\frac{1}{\sqrt{2}}e^{-8b}\right](g_1^2+g_2^2)\\-\frac{4\sqrt{\pi}}{\e}\left[e^{-2b}+\frac{1}{\sqrt{3}}e^{-6b}\right]g_1g_2+O(1),
\end{multline*}
when $b\gg1,$ matching the asymptotics of $\kappa_\e$ established in Theorem~\ref{t.2balls.constbc}.

\begin{remark}
Consider a single particle $B_{1/\e^2}$ of the radius $1/\e^2$ centered at the origin and let $k(T)=2$ in \eqref{eq:pot0} so that
\[W(u)=2{|u|}^2-2{|u|}^4+{|u|}^6.\]
We can use asymptotic arguments developed in this section to compare the rates of decay of solutions to the nonlinear
	\begin{equation} \label{e.PDE.0.1}
		\begin{aligned}
			\Delta u &= \frac{1}{4}\nabla_uW(u), \quad u \in H^1\left(\R^2 \setminus \overline{B}_{1/\e^2}\right)\\
			u & = 1 \qquad \mbox{ on } \partial B_{1/\e^2}\,. 
		\end{aligned}
	\end{equation}
 and linear
	\begin{equation} \label{e.PDE.1.1}
		\begin{aligned}
			\Delta v &= v, \quad v \in H^1\left(\R^2 \setminus \overline{B}_{1/\e^2} \right)\\
			v & = 1 \qquad \mbox{ on } \partial B_{1/\e^2} 
		\end{aligned}
	\end{equation}
problems when $\e$ is small. Indeed, suppose that $(\rho,\theta)$ are polar coordinates associated with the center of a particle $B_{1/\e^2}$ and $\e\ll1$. From the proof of Lemma \ref{l.singleKm} for $m=0$ and \eqref{eq:psi1}, we have that the solution of \eqref{e.PDE.1.1} is
\begin{equation}
    \label{eq:non1}
    v = \frac{K_0 (\rho)}{K_0 \big( \frac{1}{\e^2} \big)}\sim e^{1/\e^2-\rho},
\end{equation}
because $\rho>\frac{1}{\e^2}$ and $\e\ll1.$ Assuming that $u=u(\rho)$ in the same asymptotic regime $u$ satisfies
\begin{equation} \label{e.PDE.0.2}
    \begin{aligned}
			u_{\rho\rho} &= \frac{1}{4}\nabla_uW(u),\\
			u\left(\frac{1}{\e^2}\right) & = 1, \quad \lim_{\rho\to\infty}u=0\,. 
	\end{aligned}
\end{equation}
to leading order in $\e.$ This problem has an explicit solution
\[u=\frac{\sqrt{2}}{\sqrt{1+\cosh{2\left(\rho-\frac{1}{\e^2}\right)}+\sqrt{2}\sinh{2\left(\rho-\frac{1}{\e^2}\right)}}}\sim\frac{2}{\sqrt{1+\sqrt{2}}}e^{1/\e^2-\rho},\]
when $\rho\gg\frac{1}{\e^2}$ and it would be reasonable to expect that the interaction between two well-separated particles would be the same to leading order, up to a multiplicative constant.
 
\end{remark}

It is possible to quantify the tail behaviors of $u$ and $v$ rigorously for general boundary conditions. At its heart is a convexity argument. We begin by noticing that with $W$ as in~\eqref{eq:pot0}  and $k(T) > \frac{4}{3},$ the function $W$ is uniformly convex and $C^2,$ so that for any $p,q \in \R^2,$ we have 
 \begin{equation} \label{e.c0}
 	\bigl(\nabla W(p) - \nabla W(q)\bigr) \cdot (p-q) \geqslant c_0 |p - q|^2\,,
 \end{equation}
 for some $c_0 > 0$, that only depends on $k(T) > \frac43.$ Then we have, 

 \begin{proposition} \label{t.nonlinear}
     Let $v_\e$ denote the unique $H^1(\R^2 \setminus \overline{\hat{B}^\e_1 \cup \hat{B}^\e_2})$ solution to the nonlinear system ~\eqref{e.PDE.0}, and let $u_\e$ the corresponding unique solution to the linear system~\eqref{e.PDE.1}, both with the same boundary condition. Then, we claim that there exists {$0<\alpha\leq 1$ and $C>0$  } that only depend on $c_0,$ from~\eqref{e.c0} such that 
 \begin{equation}
     \label{e.claim}
     \int_{|x|\geqslant C_0|\log \e|} e^{2\alpha H_\e(x)}|u_\e-v_\e|^2\,dx \leqslant C{\e^{\frac{8}{2+\alpha}}} \,, 
 \end{equation}
 with  {$\frac{8}{2+\alpha}\geq\frac{8}{3}$ and}
 \begin{equation}
     \label{eq:he}
     H_\e(x) = \frac1{\e^2} \mathrm{dist}(x,\overline{\hat{B}^\e_1 \cup \hat{B}^\e_2}).
 \end{equation}
 \end{proposition} 
 \smallskip 
 
      \begin{proof}[Proof of the Proposition~\ref{t.nonlinear}]
      Let $\eta$ be a positive smooth compactly supported test function that will be subsequently chosen. 
 	We compute that 
 	\begin{equation*}
 		\Delta u_\e - \frac{1}{2\e^4}\nabla_u W(u_\e) = \Delta (u_\e - v_\e) + \frac{1}{2\e^4}\bigl( \nabla_u W(v_\e) - \nabla_u W(u_\e) \bigr)\,,
 	\end{equation*}
 	so that taking the dot product of both sides with $\eta^2(v_\e - u_\e),$ integrating on $\R^2 \setminus \overline{\hat{B}^\e_1 \cup \hat{B}^\e_2}$, using the uniform convexity of $W,$ integrating by parts, and using that $u=v$ on the boundary, yields 
 	\begin{equation*}
  \begin{aligned}
     & C\int |\nabla \eta|^2 |u - v|^2+\int \eta^2 (v_\e - u_\e) \cdot \biggl( \Delta u_\e - \frac{1}{2\e^4}\nabla_u W(u_\e)\biggr)\,dx \\
      &\qquad \qquad \geqslant  \frac12 \int\eta^2 |\nabla u_\e - \nabla v_\e|^2 + \frac{c_0}{2\e^4} \int \eta^2 |v_\e-u_\e|^2\,dx\,. 
  \end{aligned}
 	\end{equation*}
 As~$|u_\e|\leqslant 1$, we estimate 
 	\begin{equation*}
 		\biggl|\Delta u_\e - \frac{1}{2\e^4} \nabla_u W(u_\e) \biggr| = \frac{1}{2\e^4}\bigl|\bigl( 8|u_\e|^2 - 6|u_\e|^4 \bigr)u_\e\bigr| \leqslant \frac{C}{\e^4}|u_\e|^3\,,
 	\end{equation*}
 	it follows by Cauchy-Schwarz that 
 	\begin{equation} \label{e.mainestimate}
  \begin{aligned}
     &\int\eta^2 |\nabla u_\e-\nabla v_\e|^2 + \frac{c_0}{2\e^4}\int \eta^2  |u_\e - v_\e|^2\,dx\\
     &\quad \quad \leqslant C\biggl( \frac1{\e^4}\int \eta^2|v_\e - u_\e|^2\biggr)^{\frac12} \biggl( \frac1{\e^4}\int \eta^2|u_\e|^6\biggr)^{\sfrac12} + C \int |\nabla \eta|^2|v_\e-u_\e|^2\,dx\\
     &\quad \quad \leqslant \frac{c_0}{4\e^4} \int \eta^2|u_\e - v_\e|^2 \,dx + \frac{C}{\e^4}\int \eta^2 |u_\e|^6\,dx + C \int |\nabla \eta|^2 |v_\e - u_\e|^2\,dx\,,
  \end{aligned}
 	\end{equation}
  so that upon rearranging, we get 
\begin{equation} \label{e.mainest2}
 \int\eta^2 |\nabla u_\e-\nabla v_\e|^2 +   \frac{c_0}{4\e^4}\int \eta^2 |u_\e - v_\e|^2\,dx \leqslant \frac{C}{2\e^4}\int \eta^2 |u_\e|^2 + C \int |\nabla \eta|^2 |v_\e - u_\e|^2\,dx \,.
\end{equation}
Let $\alpha > 0$ be a parameter that will be subsequently chosen. Further, we let $\chi\in C^\infty_c(\R^2),$ be a test function that will be chosen, which will satisfy  $0 \leqslant \chi\leqslant 1,$ and $\chi \equiv 0$ in the disk $B_4$ of radius $4$ centered at the origin. Finally, we set $\eta$ is defined by 
\begin{equation*}
    \eta(x) = e^{\alpha H_\e} \chi\,,
\end{equation*}
 where $H_\e$ is given in \eqref{eq:he}. As $|g| \leqslant 1,$ we note that $|u_\e(x)| \leqslant Ce^{-H_\e}$ (for example by examining the representation formula using the Green's function of the operator $\Delta - \e^{-4}$), and that on the support of $\chi,$ we have 
 \[
|\nabla \eta|\leqslant e^{\alpha H_\e} \bigl(\alpha  \chi |\nabla H_\e| + |\nabla \chi| \bigr) \leqslant e^{\alpha H_\e} \biggl( \alpha\frac{\chi}{\e^2} + |\nabla \chi|\biggr) = \frac{\alpha}{\e^2}\eta + e^{\alpha H_\e}|\nabla \chi|\,.
 \]
Inserting this in the estimate~\eqref{e.mainest2} we obtain 
\begin{equation*}
\begin{aligned}
    & \int \eta^2|\nabla u_\e - \nabla v_\e|^2 + \frac{c_0}{4\e^4}\eta^2 |u_\e-v_\e|^2\,dx  \lesssim \frac1{\e^4}\int \eta^2|u_\e|^6\,dx + \int \biggl(\frac{\alpha^2}{\e^4} \eta^2 + e^{2\alpha H_\e} |\nabla \chi|^2 \biggr)|v_\e-u_\e|^2\,dx\,.
\end{aligned}
\end{equation*}
Choosing $\alpha^2 = {\min}(1,\frac{c_0}{8}),$ we obtain 
\begin{equation*}
    \int e^{2\alpha H_\e} \chi^2 |\nabla u_\e- \nabla v_\e|^2 + \frac{c_0}{8\e^4}\int e^{2\alpha H_\e} \chi^2|u_\e - v_\e|^2 \lesssim \frac1{\e^4} \int e^{(2\alpha - 6) H_\e }\chi^2\,dx +  \int e^{2\alpha H_\e}|\nabla \chi|^2 |v_\e - u_\e|^2\,dx \,.
\end{equation*}
To conclude, we simply choose a sequence of dyadic annuli $\mathcal{A}_k := \{ |x| \in  [2^{k-1} R_0, 2^{k+2} R_0]\}$, and a corresponding sequence of choices $\chi = \chi_k$, with $\chi_k(x) \equiv 1$ for $|x| \in [2^k R_0, 2^{k+1} R_0],$ and $\chi_k(x) \equiv 0$ when $|x| \leqslant 2^{k-1}R_0$ or if $|x| \geqslant 2^{k+2} R_0,$ and $|\nabla \chi_k| \leqslant \frac{C}{R_0 2^k}.$ Summing over $k \in \mathbb{N},$ and buckling one last time, we find 
\begin{equation*}
    \int_{|x| \geqslant R_0} e^{2\alpha H_\e} \biggl( |\nabla u_\e - \nabla v_\e|^2 + \frac{c_0}{16 \e^4} |u_\e - v_\e|^2\biggr)\,dx \lesssim \frac{C}{\e^4} e^{-4R_0} +  \frac{C}{R_0^2} \int_{|x|\leqslant R_0} e^{2\alpha H_\e} |u_\e - v_\e|^2\,.
\end{equation*}
Finally, since $|u_\e| \leqslant 1$ and $|v_\e|\leqslant 1,$ we conclude by the triangle inequality and multiplying through by $\e^4$, that 
\begin{equation*}
    \int_{|x| \geqslant R_0} e^{2\alpha H_\e} |u_\e - v_\e|^2\,dx \lesssim  C e^{-4R_0} + C \e^4 e^{2\alpha R_0}\,. 
\end{equation*}
{It remains to make an optimal choice of~$R_0.$ Letting~$R_0 = k |\log \e|$ for a $k > 0$ to be chosen, the last estimate yields 
\begin{equation*}
    \int_{|x|\geqslant k |\log \e|} e^{2\alpha H_\e} |u_\e - v_\e|^2 \,dx \lesssim C\bigl( \e^{4k} + \e^{4 - 2\alpha k}\bigr)\,.
\end{equation*}
Balancing the two terms yields the optimal choice of~$k = \frac2{2+\alpha},$ and this choice, in turn, completes the proof of the proposition. 
}

\end{proof}

	\section{Two particles and nonconstant boundary conditions}
	
 \label{ss.two.var}
	In this section, we continue working with the geometry of Section \ref{ss.two.const}, but consider, instead, variable boundary conditions. In other words, we seek to obtain an expansion, in powers of $\e$ of the energy $F_\e$ for the problem \eqref{e.PDEblownup}, when $g_1,g_2$ are nonconstant. 
 Without loss of generality, we assume that the functions $g_i : \partial B_i^\e \to \RR^2$ have uniformly and absolutely summable Fourier developments with respect to local polar coordinates on the circles $\partial B_i^\e.$ To be precise, parametrizing~$\partial B_1$ by~$\{ (0, \sfrac{1}{\e^2} + b) + \sfrac{1}{\e^2}(\cos \theta, \sin \theta): \theta \in [0,2\pi)\},$ we assume that 
	\begin{align*}
		g_1 \bigl( (0, \sfrac{1}{\e^2} + b) + \sfrac{1}{\e^2}(\cos \theta, \sin \theta)\bigr) = \sum_{m\in \Z}g_m^{(1)} e^{im \theta}, \quad \quad \theta \in [0,2\pi),
	\end{align*}
	for some Fourier coefficients~$\{g_m^{(1)}\}_{m\in \Z} \in \ell^2 (\Z;\CC^2).$
	Similarly, parametrizing~$\partial B_2$ by $\{(0, - \sfrac{1}{\e^2}-b) + \sfrac{1}{\e^2}(\cos \theta, \sin \theta): \theta \in [0,2\pi)\},$ we assume that 
	\begin{align*}
		g_2 \bigl( (0, - \sfrac{1}{\e^2}-b) + \sfrac{1}{\e^2}(\cos \theta, \sin \theta)\bigr) = \sum_{m \in \Z} g_m^{(2)} e^{im\theta}, \quad \quad \theta \in [0,2\pi),
	\end{align*}
	for some Fourier coefficients $\{g_m^{(2)}\}_{m\in \Z} \in \ell^2(\Z;\CC^2).$
	
	As in the case of constant boundary conditions, we perform a splitting of the solution. To be precise, we introduce, for $j = 1,2,$ the functions 
	\begin{align*}
		\Psi_m^{(j)}(x) := \frac{K_m(|x-a_j|)}{K_m(\sfrac{1}{\e^2})} e^{im \theta_j}, \quad x \in \R^2 \setminus  \overline{B}_1 \cup \overline{B}_2
	\end{align*}
	where $\theta_j \in [0,2\pi)$ denotes the polar angle with pole $a_j$ so that for any $x \neq a_j,$ we have $(x-a_j) = |x-a_j| e^{i\theta_j}.$ Here, $K_m$ denotes the modified Bessel function of second kind of order $m,$ and by Lemma \ref{l.singleKm}, the function $\Psi_m^{(j)}$ captures the behavior of a single colloid. Finally, as in the case of the constant boundary conditions, we define $Z_m^{(j)}\in H^1\left(\Omega_\e\right),\ j=1,2,\ m\in\mathbb{N}$ to be the unique solutions to the problems 
	\begin{equation} \label{e.Zmjdef}
		\begin{aligned}
			\Delta Z_m^{(j)} &= Z_m^{(j)}  \ \ \quad \quad \mbox{ in } \Omega_\e,\\
			Z_m^{(j)} &= 0 \quad \ \ \ \quad \quad \mbox{ on } \partial B_j\\
			Z_m^{(j)} &= - \Psi^{(j)}_m \quad \quad \mbox{ on } \partial B_{\sigma(j)}\,.
		\end{aligned}
	\end{equation}
	where we recall that the transposition map~ $\sigma: \{1,2\} \to \{1,2\}$ is defined via~ $\sigma(1) = 2, \sigma(2) = 1.$ 
	
	Then, it is clear that the unique solution to \eqref{e.PDEblownup} is given by the formula 
	\begin{equation} \label{e.uexpansion}
		\begin{aligned}
			U = \sum_{j=1}^2 \sum_{m \in \Z} g_m^{(j)}(\Psi_m^{(j)} + Z_m^{(j)})
		\end{aligned}
	\end{equation}
	We will see shortly that the infinite sum in \eqref{e.uexpansion} does indeed converge in $H^1(\Omega_\e)$ and is therefore well-defined. In order to focus on the essential issues for the time being, let us suppose that there exists $M_0 \in \NN,$ such that 
	\begin{equation} \label{e.compactsptdata}
		g_m^{(j) } \equiv  0 \quad \quad \mbox{ for all } |m | \geqslant M_0, j = 1,2.
	\end{equation}
	This makes the infinite sums in \eqref{e.uexpansion}, in fact, finite. With this assumption, in what follows we will freely interchange various integrals and sums, keeping careful track of the dependence of errors on the tail parameter $M_0,$ and send $M_0 \to \infty$ at the end of the proof of Theorem \ref{t.mt} below. 
	
	Next, let us note that identical to \eqref{e.decompofenergy}, in this case, too, the energy associated to $u$ admits a splitting. To be precise, we have 
	\begin{lemma} \label{l.splitting.nonconst}
	Under the assumption \eqref{e.compactsptdata},	we have the decomposition 
		\begin{equation} \label{e.quadraticformdef}
			\begin{aligned}
			F_\e(U)= 
				\mathfrak{Re} \left[\begin{pmatrix}
					\ldots g_m^{(1)} \ldots | \ldots g_m^{(2)} \ldots 
				\end{pmatrix} \mathcal{A}_{M_0} \begin{pmatrix}
					\vdots\\ \overline{g_n^{(1)}} \vdots\\ ---\\  \vdots \overline{g_n^{(2)}} \vdots 
				\end{pmatrix} \right],
			\end{aligned}
		\end{equation}
	with the matrix $\mathcal{A}_{M_0} \in \R^{(4 M_0 + 1) \times (4M_0 + 1)}$ being given by 
	\begin{equation*}
		{ \frac{1}{2}}\begin{pmatrix}
			\begin{matrix}
				\ddots & \ddots & \ddots \\
				& & \\
				\ddots &  \int_{\partial B_1} e^{im \theta_1} \overline{\frac{\partial}{\partial \nu_1} (\Psi_n^{(1)} + Z_n^{(1)})} \,d \sh^1 & \ddots \\
				& & \\
				\ddots & \ddots & \ddots 
			\end{matrix}
			& \vline  & 
			\begin{matrix}
				\ddots & \ddots & \ddots \\
				& & \\
				\ddots &   \int_{\partial B_1} e^{im\theta_1} \overline{\frac{\partial }{\partial \nu_1} (\Psi_n^{(2)} + Z_n^{(2)})} \,d \sh^1 & \ddots \\
				& &\\
				\ddots & \ddots & \ddots 
			\end{matrix} \\ 
			\hline \\
			\begin{matrix}
				\ddots & \ddots & \ddots \\
				& & \\
				\ddots &  \int_{\partial B_2} e^{im \theta_2} \overline{\frac{\partial }{\partial \nu_2}(\Psi_n^{(1)} + Z_n^{(1)})} \,d \sh^1 & \ddots \\
				& & \\
				\ddots & \ddots & \ddots 
			\end{matrix}
			& \vline  & 
			\begin{matrix}
				\ddots & \ddots & \ddots \\
				& & \\
				\ddots &   \int_{\partial B_2} e^{im \theta_2} \overline{\frac{\partial}{\partial \nu_2} (\Psi_n^{(2)} + Z_n^{(2)})} \,d \sh^1& \ddots \\
				& & \\
				\ddots & \ddots & \ddots 
			\end{matrix} 
		\end{pmatrix} \,.
	\end{equation*}
	\end{lemma}
	\begin{proof}
		Indeed, plugging in the representation formula \eqref{e.uexpansion} for the solution $U$ into the energy and integrating by parts, we arrive at 
		\begin{equation} \label{e.Fourierdev}
			\begin{aligned}
				2    F_\e (U) = \int_{\Omega_\e} |\nabla U|^2 + |U|^2 \,dx = { \mathfrak{Re} }\biggl[\int_{\partial B_1} U \cdot \overline{\frac{\partial U}{\partial \nu_1}} \,d \sh^1 + \int_{\partial B_2}U \cdot \overline{\frac{\partial U}{\partial \nu_2}}\,d \sh^1\biggr]. 
			\end{aligned}
		\end{equation}
		For each $j = 1,2,$ on the boundary $\partial B_j,$ we note that $U = \sum_{m \in \Z} g_m^{(j)} e^{im \theta_j},$ since on $\partial B_{\sigma(j)}$ we have $\Psi_m^{(j)} + Z_m^{(j)} = 0$ by construction. Inserting the Fourier development into \eqref{e.Fourierdev}, and rewriting in as a quadratic form with the matrix being written in block form, we find 
		\begin{equation}
			\begin{aligned}
				&2F_\e(U) =  { \mathfrak{Re} }\biggl[ \int_{\partial B_1} \biggl( \sum_{m \in \Z} g_m^{(1)} e^{im\theta_1}\biggr)\cdot \overline{\biggl( \sum_{n \in \Z} \sum_{j=1}^2 g_n^{(j)} \frac{\partial}{\partial \nu_1} (\Psi_n^{(j)} + Z_n^{(j)})\biggr)}\,d \sh^1\biggr] \\ &\quad \quad +  { \mathfrak{Re} }\biggl[\int_{\partial B_2 }  \biggl( \sum_{m \in \Z} g_m^{(2)} e^{im\theta_2}\biggr)\cdot \overline{\biggl( \sum_{n \in \Z} \sum_{j=1}^2 g_n^{(j)} \frac{\partial}{\partial \nu_2} (\Psi_n^{(j)} + Z_n^{(j)})\biggr)}\,d \sh^1 \biggr] \,.
			\end{aligned}
		\end{equation}
	Expanding and rewriting in matrix form completes the proof of the lemma. 
	\end{proof}
	As before, our main task is to estimate the asymptotics as $\e \to 0^+$ of the entries of the matrix $\mathcal{A}_{M_0}.$ We accomplish this in a series of Lemmas. Our first lemma is the analog of Lemma \ref{l.normal.der} for the present nonconstant boundary conditions case and has a similar proof, as we demonstrate. 
	\begin{lemma}
		\label{l.zn1est}
		The functions $Z_n^{(j)}$ satisfy the estimate 
		\begin{align} \label{e.self}
			\biggl\|  \frac{\partial Z_n^{(j)}}{ \partial \nu_j}\biggr\|_{H^{-\sfrac{1}{2}}(\partial B_j)} + 
			\biggl\|  \frac{\partial Z_n^{(j)}}{\partial \nu_{\sigma(j)}}\biggr\|_{H^{-\sfrac{1}{2}}(\partial B_{\sigma(j)})} \leqslant \frac{C(1+|n|)^{\sfrac12}}{\sqrt{\e}} K_n(2b).
		\end{align}
	\end{lemma}
	
	\begin{proof}
		As mentioned before, the proof of this lemma proceeds similarly to that of Lemma \ref{l.normal.der}. Without loss of generality, fix $j = 1.$\\
		By the definition of $H^{-\sfrac12}(\partial \Omega_\e )$ (see Appendix~\ref{sec:Sobolev}) we have 
		\begin{equation} \label{e.h-halfbnd}
			\begin{aligned}
				&\biggl\| \frac{\partial Z_n^{(1)}}{\partial \nu}\biggr\|_{H^{-\sfrac12}(\partial \Omega_\e)} = \sup_{\phi \in H^1(\Omega_\e): \|\phi\|_{H^1(\Omega_\e)} \leqslant 1} \biggl[ \int_{\Omega_\e } \nabla Z_n^{(1)} \cdot \nabla \phi + \Delta Z_n^{(1)} \phi \,dx \biggr] \\
				&\quad = \sup_{\phi \in H^1(\Omega_\e): \|\phi\|_{H^1(\Omega_\e)} \leqslant 1} \biggl[ \int_{\Omega_\e } \nabla Z_n^{(1)} \cdot \nabla \phi +  Z_n^{(1)} \phi \,dx \biggr] \\
				&\quad \leqslant \|Z_n^{(1)}\|_{H^1(\Omega_\e)} = \sqrt{2F_\e(Z_n^{(1)})},
			\end{aligned}
		\end{equation}
		and therefore, the desired estimate follows by constructing a competitor to the variational problem of minimizing the energy $F_\e$ subject to the boundary conditions of $Z_n^{(1)}.$  Our competitor  construction and estimation of its energy proceeds as before. 
		
		\emph{Construction of competitor $\zeta$  for $Z_n^{(1)}$ and estimating the energy of the $\zeta$:}  Our competitor $\zeta \in H^1(\Omega_\e)$ must be constructed satisfying the boundary conditions for $\zeta = Z_n^{(1)}$ on $\partial \Omega_\e,$ so that $\zeta = 0$ on $\partial B_1$ and $\zeta = - \Psi_n^{(1)}$ on $\partial B_2.$ We let $\eta: (0,\infty) \to [0,1]$ be a $C^1$ function that satisfies $\eta(t) \equiv 1$ for $t \in \bigl[ \sfrac{1}{\e^2}, \sfrac{1}{\e^2} + \sfrac{b}{2}\bigr],$ $\eta(t) \equiv 0$ when $t \geqslant \sfrac{1}{\e^2} + b,$ and $|\eta^\prime| \leqslant \frac{2}{b},$ and set 
		\begin{align*}
			\zeta(x) := - \Psi_n^{(1)}(x) \eta(|x - a_2|).
		\end{align*}
		Then 
		\begin{align*}
			\nabla \zeta = - \eta(|x-a_2|) \nabla \Psi_n^{(1)}(x) - \Psi_n^{(1)}(x) \eta^\prime(|x-a_2|)\frac{x-a_2}{|x-a_2|},
		\end{align*}
		so that, pointwise, we have the bound 
		\begin{align*}
			|\nabla \zeta(x)| \leqslant   |\nabla \Psi_n^{(1)}(x)| + \frac{2}{b}|\Psi_n^{(1)}(x)| ,
		\end{align*}
		with support in the set $\sfrac{1}{\e^2}\leqslant |x-a_2| \leqslant \sfrac{1}{\e^2} + b.$ Then, the energy of $\zeta$ is easily calculated: 
		\begin{equation} \label{e.energyofzetan}
			\begin{aligned}
			&F_\e(\zeta) \leqslant C\int_{\sfrac{1}{\e^2}\leqslant |x-a_2| \leqslant \sfrac{1}{\e^2} + b} \bigl( |\Psi_n^{(1)}|^2 + |\nabla \Psi_n^{(1)}|^2\bigr) \,dx \\
				&\quad = C \biggl|\int_{|x-a_2| = \sfrac{1}{\e^2} + b} \Psi_n^{(1)}(x) \overline{ \frac{\partial \Psi_n^{(1)}(x)}{\partial \nu}} \,d \sh^1 - \int_{|x-a_2| = \sfrac{1}{\e^2}} \Psi_n^{(1)}(x) \overline{ \frac{\partial \Psi_n^{(1)}(x)}{\partial \nu}} \,d \sh^1 \biggr|,
			\end{aligned}
		\end{equation}
		where we plugged in the PDE satisfied by $\Psi_n^{(1)}$ and integrated by parts as before; the signs in front of the boundary integrals reflect our choice that the corresponding unit normals point towards $a_2.$ Recalling that 
		\begin{align*}
			\Psi_{n}^{(1)}(x) = \frac{K_n(|x-a_1|)}{K_n(\sfrac{1}{\e^2})} e^{in\theta_1},
		\end{align*}
		it is clear that 
		\begin{align*}
			|\nabla \Psi_n^{(1)}(x)| \leqslant \frac{|K_n^\prime(|x-a_1|)|}{K_n(\sfrac{1}{\e^2})} + \frac{|n|}{|x-a_1|} \frac{K_n(|x-a_1|)}{K_n(\sfrac{1}{\e^2})} , \quad \quad x \in \RR^2 \setminus \overline{B}_2. 
		\end{align*}
		\emph{Estimating the first boundary integral in \eqref{e.energyofzetan}.} For this term, as before, we estimate $|\frac{\partial \Psi_n^{(1)}}{\partial \nu}|\leqslant |\nabla \Psi_n^{(1)}|,$ and use the geometric observation that when $|x-a_2| = \sfrac{1}{\e^2},$ parametrizing $x = a_2 + \sfrac{1}{\e^2}(\cos \theta,\sin\theta), \theta \in (-\pi,\pi]$ with $\theta = 0$ along the vertical, we have the lower bound 
		\begin{align*}
			|x-a_1| \geqslant \sfrac{1}{\e^2} + 2b + \sfrac{1}{\e^2}(1-\cos \theta).
		\end{align*}
		As $K_n$ and $-K_n^\prime$ are both monotone decreasing functions, and moreover, since on the circle~ $|x-a_2| = \sfrac{1}{\e^2},$ we have $\,d\sh^1 = \sfrac{1}{\e^2}\,d\theta$ in the above parametrization, we find 
		\begin{equation*}
			\begin{aligned}
				&\biggl|\int_{|x-a_2| = \sfrac{1}{\e^2}} \Psi_n^{(1)}\overline{ \frac{\partial \Psi_n^{(1)}(x)}{\partial \nu}} \,d \sh^1\,d \sh^1 \biggr| \\ &\quad \leqslant \frac{2}{\e K_n^2(\sfrac{1}{\e^2})} \Biggl|\int_0^\pi  K_n\bigl( \sfrac{1}{\e^2} + 2b + \sfrac{1}{\e^2}(1- \cos \theta)\bigr) \biggl[ -K_n^\prime( \sfrac{1}{\e^2} + 2b + \sfrac{1}{\e^2}(1 - \cos \theta)) \\ &\quad \quad \quad \quad + \frac{|n|}{\sfrac{1}{\e^2} + 2b + \sfrac{1}{\e^2}(1 - \cos \theta)} K_n (\sfrac{1}{\e^2}+ 2b + \sfrac{1}{\e^2}(1- \cos \theta))\biggr] \,d\theta\Biggr|\\
				&\quad \leqslant  \frac{C}{\e } (1 + |n|)K_n^2 (2b),
			\end{aligned}
		\end{equation*}
		by an easy computation similar to that in the proof of Lemma \ref{l.normal.der}. By a similar argument, the second term in \eqref{e.energyofzetan} satisfies the bound 
		\begin{equation*}
			\biggl|\int_{|x-a_2| = \sfrac{1}{\e^2}+b} \Psi_n^{(1)}(x) \overline{\frac{\partial \Psi_n^{(1)}(x)}{\partial \nu}}\,d \sh^1 \biggr| \leqslant \frac{C}{\e}(1+|n|)K_n^2(2b). 
		\end{equation*}
		Putting these together with the bound in \eqref{e.energyofzetan} and \eqref{e.h-halfbnd}, the proof of the Lemma is completed. 
	\end{proof}
	In the next lemma we obtain the $\e\to 0^+$ asymptotic expansion for the diagonal blocks in the matrix in \eqref{e.quadraticformdef}. 
	\begin{lemma}
		\label{l.diagterms.nonconstant} 
		 {Assume $b\e^2 \ll 1.$} For every $m, n\in \Z,$ the $(m,n)$ entry of the diagonal block of the quadratic form \eqref{e.quadraticformdef} satisfy the expansion 
		\begin{equation}
			\begin{aligned}
				 \int_{\partial B_1} e^{im \theta_1} \overline{\frac{\partial }{\partial \nu_1}(\Psi_n^{(1)} + Z_n^{(1)} )}\,d \sh^1 = -\frac{{ 2}\pi}{\e^2}\frac{K_n^\prime(\sfrac{1}{\e^2})}{K_n(\sfrac{1}{\e^2})} \delta_{mn} + R_{mn},
			\end{aligned}
		\end{equation}
		where 
		\begin{equation}
			\label{e.Rmnest}
			|R_{mn}| \leqslant \frac{C}{\e}K_m(2b)K_n(2b)\sqrt{(1+|m|)(1+|n|)} {\leqslant \frac{Ce^{-4b}}{b\e}\sqrt{(1+|m|)(1+|n|)}}.
		\end{equation}
		The same expansion holds for the bottom right diagonal block.
	\end{lemma}
 {\begin{remark} \label{recent}
     Let us carefully note that in the regime of~$b \gg 1,$ the remainder terms~$R_{mn}$ are of higher order than the interaction energy in the statement of Theorem~\ref{t.mt}; this is because the modified Bessel functions~$K_m$ have the large argument asymptotics~$K_m(z) \sim \frac{e^{-z}}{\sqrt{z}},$ so that~$R_{mn} \sim b^{-1}e^{-4b}.$ We also observe that~$R_{mn}$ are controlled by  Fourier multipliers whose scaling corresponds to that of the~$H^{\sfrac12}$ norm of the Dirichlet data. 
 \end{remark}}
	\begin{proof}
		Let us note that since the normal $\nu_1$ points towards the center $a_1,$ it follows that in the polar coordinates about $a_1,$ the normal derivative $\sfrac{\partial}{\partial \nu_1} = -\sfrac{\partial}{\partial r},$ so that, on $\partial B_1,$
		\begin{align*}
			\frac{\partial}{\partial \nu_1}\Psi_n^{(1)} = - \frac{K_n^\prime(\sfrac{1}{\e^2})}{K_n(\sfrac{1}{\e^2})} e^{in\theta_1},
		\end{align*}
		and, we have 
		\begin{align*}
			\int_{\partial B_1} e^{im \theta_1} \overline{\frac{\partial }{\partial \nu_1}\Psi_n^{(1)} }\,d \sh^1  = -\frac{1}{\e^2}\frac{K_n^\prime(\sfrac{1}{\e^2})}{K_n(\sfrac{1}{\e^2})}   \int_{0}^{2\pi} e^{i(m-n)\theta_1} \,d \theta  = -\frac{2\pi}{\e^2}\frac{K_n^\prime(\sfrac{1}{\e^2})}{K_n(\sfrac{1}{\e^2})} \delta_{mn}. 
		\end{align*}
		As usual, the Kronecker's delta $\delta_{mn} := 1$ if $m =n$ and $\delta_{mn} := 0$ if $m \neq n.$
		
		We focus on the second term, i.e., on estimating 
		\begin{equation} \label{e.Z1diag}
			 { \int_{\partial B_1}} e^{im \theta_1} \overline{\frac{\partial }{\partial \nu_1}Z_n^{(1)} }\,d \sh^1 .
		\end{equation}
		The natural idea to estimate this is to directly use Lemma \ref{l.zn1est}; however, this direct estimate misses the observation that $Z_n^{(1)}$ vanishes on $\partial B_1.$ To obtain a better estimate, we use Green's second identity which, specialized to our setting, asserts that for any pair of suitably smooth functions $U_1, U_2$ that decay at infinity sufficiently fast satisfy the identity
		\begin{equation*}
			\int_{\partial B_1} U_1 \frac{\partial U_2}{\partial \nu_1} - U_2 \frac{\partial U_1}{\partial \nu_1} \,d \sh^1 + \int_{\partial B_2 } U_1 \frac{\partial U_2}{\partial \nu_2} - U_2 \frac{\partial U_1}{\partial\nu_2}\,d \sh^1 = \int_{\Omega_\e} U_1 \Delta U_2 - U_2 \Delta U_1 \,dx.
		\end{equation*}
		Applying this identity to $U_1 = \Psi_m^{(1)}$ and $U_2 = \overline{Z_n^{(1)}},$ and subsequently to the choice $U_1 = \overline{\Psi_m^{(1)}}$, and $U_2 = {Z_n^{(1)}}$, { adding the results}, we notice that the bulk terms on the right-hand side cancel, these choices of $U_1$ and $U_2$ are all equal to their respective Laplacians. We are consequently only left with boundary integrals, and we get 
		\begin{equation} \label{e.Greens2}
			\begin{aligned}
				& { \biggl| }\int_{\partial B_1} e^{im \theta_1} \overline{\frac{\partial }{\partial \nu_1}Z_n^{(1)} }\,d \sh^1\biggr| =   {  \biggl|} \int_{\partial B_1} \Psi_m^{(1)} \overline{\frac{\partial }{\partial \nu_1}Z_n^{(1)} }\,d \sh^1\biggr| \\ &\quad \quad =  { \biggl| }\int_{\partial B_1} \frac{\partial \Psi_m^{(1)}}{\partial \nu_1} \overline{Z_n^{(1)}} \,d \sh^1  + \int_{\partial B_2} \frac{\partial \Psi_m^{(1)}}{\partial \nu_1} \overline{Z_n^{(1)}} - \Psi_m^{(1)} \overline{\frac{\partial Z_n^{(1)}}{\partial \nu_1}} \,d \sh^1\biggr|\\
				&\quad \quad = { \biggl|} \int_{\partial B_2} \overline{\Psi_n^{(1)}}\frac{\partial \Psi_m^{(1)}}{\partial \nu_1} + \Psi_m^{(1)} \overline{\frac{\partial Z_n^{(1)}}{\partial \nu_1}}\,d \sh^1 \biggr|
			\end{aligned}
		\end{equation}
		using the boundary conditions satisfied by $Z_n^{(1)}.$ At this point, estimating as before and using Lemma \ref{l.zn1est}, it is easily seen that 
		\begin{align*}
			 {  \biggl|} \int_{\partial B_1} e^{im \theta_1} \overline{\frac{\partial }{\partial \nu_1}Z_n^{(1)} }\,d \sh^1\biggr| \leqslant \frac{C(1 + |n|)^{\sfrac12}(1+|m|)^{\sfrac12} K_m(2b)K_n(2b)}{\e},
		\end{align*}
		since $\|\Psi_m^{(1)}\|_{H^{\sfrac12}(\partial B_2)} +\|\tfrac{\partial}{\partial \nu_1} \Psi_m^{(1)}\|_{H^{-\sfrac12}(\partial B_2)} \leqslant \frac{C}{\sqrt{\e}}(1 + |m|)^{\sfrac12} K_m(2b) ,$ and the proof of the lemma is completed. 
		
	\end{proof}
	
	Finally, we turn to evaluating the off-diagonal blocks in \eqref{e.quadraticformdef}. The evaluation of these is not as straightforward since the terms involved do not have a straightforward separation of scales. To overcome this difficulty, we manipulate the boundary integrals that occur in the off-diagonal blocks using integrations by parts and the PDE solved by the functions involved, and this provides for a representation where the terms involved do have a separation of scales. At that point we can proceed very similarly to the proof of Lemma \ref{l.offdiag.terms} in the case of constant boundary conditions. 
	
	\begin{lemma}
		\label{l.offdiag.nonconstant}
		For each $m,n \in \Z,$ the $(m,n)$ term in each of the off-diagonal blocks satisfies the expansion 
		\begin{equation}
			\begin{aligned}
				\Biggl|   {  \int_{\partial B_1}} e^{im\theta_1} \overline{\frac{\partial }{\partial \nu_1}(\Psi_n^{(2)} + Z_n^{(2)})} \,d \sh^1- { e^{-\frac{i(n+m)\pi}{2}}} \frac{e^{-2b} \sqrt{\pi}}{\e}\Biggr| \leqslant \frac{(1 + |m|) (1+|n|)e^{-4 b}}{\e}\, .
			\end{aligned}
		\end{equation}
	\end{lemma}
	\begin{proof}
		Arguing exactly like in the proof of the Lemma \ref{l.diagterms.nonconstant} using Green's second identity, we find that 
		\begin{equation*}
			\begin{aligned}
				&{ \int_{\partial B_1} e^{im\theta_1} \overline{\frac{\partial }{\partial \nu_1}(\Psi_n^{(2)} + Z_n^{(2)})} \,d \sh^1}\\
				&\quad = \int_{\partial B_1} \Psi_m^{(1)}(x) \overline{\frac{\partial }{\partial \nu_1}(\Psi_n^{(2)} + Z_n^{(2)})} \,d \sh^1\\
				&\quad = \int_{\partial B_1} \frac{\partial}{\partial \nu_1} \Psi_m^{(1)} \overline{(\Psi_n^{(2)} + Z_n^{(2)})} \,d \sh^1 \\ &\quad \quad + \int_{\partial B_2} \frac{\partial \Psi_m^{(1)}}{\partial \nu_2} \overline{(\Psi_n^{(2)} + Z_n^{(2)})} \,d \sh^1 - \int_{\partial B_2} \Psi_m^{(1)} \overline{\frac{\partial }{\partial \nu_2} (\Psi_n^{(2)} + Z_n^{(2)})}\,d \sh^1.
			\end{aligned}
		\end{equation*}
		By the boundary conditions of $Z_n^{(2)},$ the first integral on $\partial B_1$ vanishes, and in the second integral, the function $Z_n^{(2)} $ vanishes on $\partial  B_2.$ Therefore, it follows from the preceding display and the definition of $\Psi_n^{(2)}$ on $\partial B_2$ that 
		\begin{equation} \label{e.offdiag.simp1}
			\begin{aligned}
				&{ \int_{\partial B_1} e^{im\theta_1} \overline{\frac{\partial }{\partial \nu_1}(\Psi_n^{(2)} + Z_n^{(2)})} \,d \sh^1}\\
				&=  \int_{\partial B_2} \frac{\partial \Psi_m^{(1)}}{\partial \nu_2} \overline{\Psi_n^{(2)}} \,d \sh^1  \\
				&\quad \quad +  \int_{\partial B_2} \Psi_m^{(1)} \overline{\frac{\partial }{\partial \nu_2} (\Psi_n^{(2)} + Z_n^{(2)})}\,d \sh^1.
			\end{aligned}
		\end{equation}
		Using \eqref{l.zn1est} and arguing as before using the $H^{\sfrac12}-H^{-\sfrac12}$ estimate, it is clear that the last term satisfies the bound 
		\begin{equation}
			\begin{aligned}
				{ \biggl|   \int_{\partial B_2} \Psi_m^{(1)} \overline{\frac{\partial }{\partial \nu_2} (\Psi_n^{(2)} + Z_n^{(2)})}\,d \sh^1\biggr|} \leqslant \frac{C}{\e} \sqrt{(1+|m|)(1+|n|)} K_m(2b)K_n(2b). 
			\end{aligned}
		\end{equation}
  {We point out that by the same arguments as in Remark~\ref{recent}, this constitutes a higher-order contribution in the parameter regime~$b \gg 1.$}
		Therefore in order to complete the proof of the lemma it remains to evaluate the first term on the  right-hand side of \eqref{e.offdiag.simp1}. Toward this end, we parametrize $\partial B_2$ via $x =\{ (0, - \sfrac{1}{\e^2}-b) + \sfrac{1}{\e^2}(\cos \theta_2,\sin \theta_2): \theta_2 \in [0,2\pi)\},$ and notice that in this parametrization $\Psi_n^{(2)}(x) = e^{in\theta_2},$ and $\nu_2 = - (\cos \theta_2,\sin\theta_2).$ 
		
	The main contribution then is that it remains to evaluate the first term in \eqref{e.offdiag.simp1}. We proceed identically as in the proof of Lemma \ref{l.offdiag.terms}. As in that argument, it suffices once again, to evaluate the portion of the integral on $\partial B_2^+,$ and to do this we parametrize $\partial B_2^+$ as in the proof of that lemma (as a function of $x \in [0,\frac{1}{\e^2}])$, and split the associated integral in $[0,\frac{M}{\e}],$ and $[\frac{M}{\e}, \frac{1}{\e^2}].$ We note that on $\partial B_2^+,$ we have that 
		\begin{equation*}
			{ \Psi_n^{(2)} (\theta_2)= e^{in\theta_2}  = \cos n \theta_2+i \sin n\theta_2 = T_n (\cos \theta_2)+i U_{n-1} (\cos \theta_2) \sin\theta_2}\,.
		\end{equation*}
	Here, for any $n \in \Z,$ the \emph{Chebyshev polynomial of the first kind} $T_n$ is defined via 
\begin{equation*}
	T_n(\cos \theta) =  \cos(n \theta)\,, \quad \theta \in \R\,.
\end{equation*}
Similarly, the \emph{Chebyshev polynomial of the second kind} $U_n$ is defined via 
\begin{equation*}
	U_n(\cos \theta) \sin \theta = \sin\bigl( (n+1)\theta\bigr)\,, \quad \theta \in \R\,. 
\end{equation*} 
Introducing these special functions permits us to express multiple angle trigonometric functions of $\theta_2$ in terms of $x$ and $y$. Indeed, since we have 
	\begin{equation*}
		(x,y) = (0, -\frac{1}{\e^2} - b) + \frac{1}{\e^2}(\cos \theta_2, \sin \theta_2)\,,
	\end{equation*}
it follows by rewriting $\Psi_n^{(2)}$ in terms of $x,$ that 
\begin{equation*}
	{ \Psi_n^{(2)}(x) = T_n(\e^2 x) +i \sqrt{1-\e^4x^2}U_{n-1} (\e^2 x)\,.}
\end{equation*}
Before computing $\nu_2 \cdot \nabla \Psi_m^{(1)},$ we record that for $(x,y) \in \partial B_2^+,$ arguing as in Step 3 of the proof of Lemma \ref{l.offdiag.terms}, we find 
\begin{align*}
	(\cos \theta_1, \sin \theta_1) &= \frac{(x,y) - (0, \frac{1}{\e^2} + b)}{\sqrt{x^2 + \Bigl(y - \frac{1}{\e^2} - b \Bigr)^2}}\\
	&= \frac{\Bigl( x, -2b - \frac{2}{\e^2} + \sqrt{\frac{1}{\e^4} - x^2}\Bigr)}{\sqrt{\frac{1}{\e^4} + C_M(x)}}\\
	&= \frac{\Bigl( \e^2 x, - 2b - 2 + \sqrt{1 - \e^4 x^2}\Bigr)}{\sqrt{1 + \e^4 C_M(x)}}\,. 
\end{align*}
It follows that for $(x,y) \in \partial B_2^+,$ we have 
\begin{align*}
	&\Psi_m^{(1)} (x,y) = \frac{K_m\Bigl(\sqrt{x^2 + \Bigl(y - \frac{1}{\e^2} - b \Bigr)^2}\Bigr)}{K_m(\frac{1}{\e^2})} { (\cos(m \theta_1)+i \sin (m\theta_1))}\\
	&\quad= \frac{K_m\Bigl(\sqrt{x^2 + \Bigl(y - \frac{1}{\e^2} - b \Bigr)^2}\Bigr)}{K_m(\frac{1}{\e^2})} \\
	&\quad  \times  { \Biggl( T_m\Biggl(\frac{\e^2 x}{\sqrt{x^2 + \Bigl(y - \frac{1}{\e^2} - b \Bigr)^2}}\Biggr)+i U_{m-1}\Biggl(\frac{\e^2 x}{\sqrt{x^2 + \Bigl(y - \frac{1}{\e^2} - b \Bigr)^2}}\Biggr)\frac{y - \frac{1}{\e^2} - b}{\sqrt{x^2 + \Bigl(y - \frac{1}{\e^2} - b \Bigr)^2}}\Biggr)}\,,
\end{align*}
and we recall from the computations in Step 3 of the proof of Lemma \ref{l.offdiag.terms} that 
\begin{equation*}
	\nu_2(x,y) = - \e^2 (x, y + b + \frac{1}{\e^2})\,. 
\end{equation*}
{  We compute} 
\begin{align*}
	&\nu_2 \cdot \nabla \left({ \mathfrak{Re}}\left\{\Psi_m^{(1)}\right\}\right) (x,y) \\
	&= \frac{K_m\Bigl( \sqrt{\frac{1}{\e^4} + C_M(x)}\Bigr)}{K_m(\frac{1}{\e^2})} T_m^\prime \Biggl( \frac{\e^2 x}{ \sqrt{\frac{1}{\e^4} + C_M(x)}} \Biggr) \nu_2 \cdot \nabla \frac{\e^2 x}{\sqrt{x^2 + \Bigl(y - \frac{1}{\e^2} - b \Bigr)^2}}\Bigg\vert_{y = -b - \frac{1}{\e^2} + \sqrt{\frac{1}{\e^4} - x^2}} \\
	& + \frac{K_m^\prime \Bigl( \sqrt{\frac{1}{\e^4} + C_M(x)}\Bigr) }{K_m(\frac{1}{\e^2})} T_m \Biggl(  \frac{\e^2 x}{ \sqrt{\frac{1}{\e^4} + C_M(x)}}\Biggr) \nu_2 \cdot \nabla  \sqrt{x^2 + \Bigl(y - \frac{1}{\e^2} - b \Bigr)^2}\Bigg\vert_{y = -b - \frac{1}{\e^2} + \sqrt{\frac{1}{\e^4} - x^2}} \,. 
\end{align*}
A tedious computation yields that for $(x,y) \in \partial B_2^+,$ we have 
\begin{align*}
	& \nu_2 \cdot \nabla \frac{\e^2 x}{\sqrt{x^2 + \Bigl(y - \frac{1}{\e^2} - b \Bigr)^2}}\Bigg\vert_{y = -b - \frac{1}{\e^2} + \sqrt{\frac{1}{\e^4} - x^2}} \\
	&= -  (\e^2 x, \sqrt{1 - \e^4 x^2}) \cdot \begin{pmatrix}
		\frac{\e^4 \Bigl( \sqrt{1 - \e^4 x^2} - 2 - 2b\e^2\Bigr)^2 - \e^6 x^2 }{(1 +\e^4 C_M(x))^{\sfrac32}}\\
		\frac{- \e^6 x(- 2 - 2b\e^2 + \sqrt{1 - \e^4 x^2})}{(1 +\e^4 C_M(x))^{\sfrac32}}
	\end{pmatrix}\\
&= O(\e^4)\,.
\end{align*}
Another computation yields that for $(x,y) \in \partial B_2^+,$ 
\begin{align*}
	\nu_2 \cdot \nabla \sqrt{x^2 +  \Bigl(y - \frac{1}{\e^2} - b \Bigr)^2}  &= - ( \e^2 x, \sqrt{1 - \e^4 x^2}) \cdot  \frac{(x, \sqrt{\frac{1}{\e^4} - x^2} - \frac{2}{\e^2} - 2b)}{\sqrt{\frac{1}{\e^4} + C_M(x)}}\\
	&= - (\e^2 x, \sqrt{1 - \e^4 x^2}) \cdot \frac{ (\e^2 x, \sqrt{1 - \e^4 x^2} - 2 - 2b\e^2)}{\sqrt{1 + \e^4 C_M(x)}}\\
	&= - \frac{1  - (2 + 2b\e^2) \sqrt{1 - \e^4 x^2} }{\sqrt{1 + \e^4 C_M(x)}} \approx 1\,, 
\end{align*}
when $|x| \leqslant \frac{C}{\e}.$ These approximations yield 
\begin{align*}
    \nu_2 \cdot \nabla \left({ \mathfrak{Re}}\left\{\Psi_m^{(1)}\right\}\right) (x,y)\approx & \frac{K_m^\prime \Bigl( \sqrt{\frac{1}{\e^4} + C_M(x)}\Bigr) }{K_m(\frac{1}{\e^2})} T_m \Biggl(  \frac{\e^2 x}{ \sqrt{\frac{1}{\e^4} + C_M(x)}}\Biggr)\\
    \approx & \frac{K_m^\prime \Bigl( \sqrt{\frac{1}{\e^4} + C_M(x)}\Bigr) }{K_m(\frac{1}{\e^2})} T_m (0)
\end{align*}

We turn to computing $\frac{\partial { \left(\Im\left\{\Psi_m^{(1)}\right\}\right)}}{\partial \nu_2}.$ For this we observe that 
\begin{align*}
\nu_2 \cdot 	{ \left(\Im\left\{\Psi_m^{(1)}\right\}\right)} (x,y) &=  \frac{1}{K_m(\frac{1}{\e^2})}\nu_2 \cdot \nabla \Biggl[ K_m \Bigl( \sqrt{x^2 + \bigl( y - \frac{1}{\e^2} - b \bigr)^2}\Bigr) U_{m-1} \Biggl(\frac{\e^2 x}{\sqrt{x^2 + \Bigl(y - \frac{1}{\e^2} - b \Bigr)^2}}\Biggr)\\
	&\quad \quad \times \frac{y - \frac{1}{\e^2} - b}{\sqrt{x^2 + \Bigl(y - \frac{1}{\e^2} - b \Bigr)^2}} \Biggr]\\
	& \quad \quad \approx -\frac{1}{K_m(\frac{1}{\e^2})} U_{m-1} \Biggl(\frac{\e^2 x}{\sqrt{\frac{1}{\e^4} + C_M(x)}}\Biggr)\frac{\sqrt{\frac{1}{\e^4} - x^2} - \frac{2}{\e^2} - 2b}{\sqrt{\frac{1}{\e^4} + C_M(x)}}\\
	&\quad \quad \times  K_m^\prime  \Bigl( \sqrt{\frac{1}{\e^4} + C_M(x)}\Bigr) \frac{(x, \sqrt{\frac{1}{\e^4} - x^2} - \frac{2}{\e^2} - 2b)}{  \sqrt{\frac{1}{\e^4} + C_M(x)}} \cdot  (\e^2 x, \sqrt{1 - \e^4 x^2})\\
	&\quad \quad = - \frac{1}{K_m ( \frac{1}{\e^2})} U_{m-1} \Biggl( \frac{\e^4 x}{\sqrt{1 + \e^4 C_M(x)}}\Biggr) \frac{\sqrt{1 - \e^4 x^2} - 2 - 2b \e^2}{\sqrt{1 + \e^4 C_M(x)}} \\
	&\quad \quad \quad  \times K_m^\prime  \Bigl( \sqrt{\frac{1}{\e^4} + C_M(x)}\Bigr) \frac{ 1  - (2 + 2b\e^2) \sqrt{1 - \e^4 x^2}}{\sqrt{1 + \e^4 C_M(x)}}\\
	&\quad \quad \approx - U_{m-1} (0) \frac{K_m^\prime \Bigl( \sqrt{\frac{1}{\e^4} + C_M(x)}\Bigr)}{K_m (\frac{1}{\e^2})}\,,
\end{align*}
and 
\begin{equation*}
	{ (\Psi_n^{(2)}) (x,y) =T_n(\e^2 x)+i  \sqrt{1 - \e^4x^2} U_{n-1}(\e^2 x)  \approx T_n(0)+i U_{n-1}(0)\,.} 
\end{equation*}
In the above, $\approx$ means that the left and right-hand sides of the equality differ by $O(\e^2)$ in magnitude, as can be checked from a straightforward calculation (by noting that $\nu_2 \approx (0,-1),$ and so only the derivatives along the ``radial" direction contribute)

\smallskip 
Therefore, we obtain 
\begin{align*}
&{   \int_{\partial B_2^+}} \frac{\partial \Psi_m^{(1)}}{\partial \nu_2} \overline{\Psi_n^{(2)}}\,d\sh^1 \\
&= { \int_0^{\frac{M}{\e}} \Bigl[ \frac{\partial \Psi_m^{(1)}}{\partial \nu_2} \overline{\Psi_n^{(2)}}  \Bigr] \frac{1}{\sqrt{1 - \e^4 x^2}}\,dx} + O\Bigl(\exp \bigl( - \frac{c}{\e^2} \bigr)\Bigr)\\
&\approx {  \Bigl( (T_m(0)-iU_{m-1}(0)) (T_n(0)-iU_{n-1}(0))\Bigr)} \int_0^{\frac{M}{\e}}\frac{K_m^\prime \Bigl( \sqrt{\frac{1}{\e^4} + C_M(x)}\Bigr) }{K_m(\frac{1}{\e^2})} \,dx \\
& \approx  e^{-\frac{i(n+m)\pi}{2}}. \frac{e^{-2b}\sqrt{\pi}}{\e} \,,
\end{align*}
where we argue exactly like in the proof of Lemma~\ref{l.offdiag.terms}. 

	\end{proof}
	With the foregoing lemmas at hand, just like in the constant boundary case, the proof of the main theorem of this section is then immediate.

	\begin{proof}[Proof of Theorem \ref{t.mt}]
		The proof is a combination of the preceding lemmas, and sending $M_0 \to \infty.$ We note that for fixed $M_0 < \infty,$ from \eqref{e.quadraticformdef}, expanding the quadratic form we obtain
		\begin{equation}
			\begin{aligned}
				&2F_\e(U) =  {\mathfrak{Re} }\biggl[ \int_{\partial B_1} \biggl( \sum_{m \in \Z} g_m^{(1)} e^{im\theta_1}\biggr)\cdot \overline{\biggl( \sum_{n \in \Z} \sum_{j=1}^2 g_n^{(j)} \frac{\partial}{\partial \nu_1} (\Psi_n^{(j)} + Z_n^{(j)})\biggr)}\,d \sh^1\biggr] \\ &\quad \quad +  {\mathfrak{Re} }\biggl[\int_{\partial B_2 }  \biggl( \sum_{m \in \Z} g_m^{(2)} e^{im\theta_2}\biggr)\cdot \overline{\biggl( \sum_{n \in \Z} \sum_{j=1}^2 g_n^{(j)} \frac{\partial}{\partial \nu_2} (\Psi_n^{(j)} + Z_n^{(j)})\biggr)}\,d \sh^1 \biggr] \,.
			\end{aligned}
		\end{equation} 
       \begin{comment}
            \begin{align*}
			\G^T \mathcal{A}_{M_0}G
   &{= - \sum_{i=1}^2 \sum_{m,n}  \frac{{ \pi}}{\e^2} \Bigl(\frac{K_n'(\sfrac{1}{\e^2})}{K_n(\sfrac{1}{\e^2})}\Bigr) \delta_{mn} g_m^{(i)}\overline{g_n^{(i)}}}
   \\
   &{+} 2\sum_{m,n} c_{m+n} g_m^{(1)}\overline{g_n^{(2)}} \\
			&{=}\mathfrak{R}\Bigl( 2 g_0^{(1)} \overline{g_0^{(2)}} + g_0^{(1)}\overline{\sum_{m = 0, n \neq 0} c_n g_n^{(2)}} + g_0^{(2)}\overline{\sum_{n=0, m \neq 0} c_m g_m^{(1)}} +  \sum_{m\neq 0, n \neq 0} c_{m+n} g_m^{(1)} \overline{g_n^{(2)}}\Bigr)\,.
		\end{align*} 
       \end{comment} 
       	where, since all the sums are finite (as~$|g_m^{(i)}| = 0$ if~$|m| > M_0$), we can freely rearrange terms in the summation, and carry out various differentiation and integration operations term-by-term. %{We point out that the terms on the first line correspond to the ``diagonal terms'' (and arise from  Lemma~\ref{l.diagterms.nonconstant}), whereas, the terms on the remaining two lines represent the ``off-diagonal'' terms which are estimated in the previous lemma.}  
        Now we use the definition of Fourier coefficients: 
	\begin{equation*}
		g_m^{(i)} = \fint_0^{2\pi} e^{-im \theta} g_i(\theta)\,d\theta, \quad m \in \Z, i \in \{1,2\}\,. 
	\end{equation*}
Inserting this in the prior expression, and invoking Lemma 3.3 for the diagonal blocks, and Lemma 3.5 for the off-diagonal (i.e. interaction) terms,  we obtain, for the interaction terms 
\begin{align*}
\frac{2\sqrt{\pi}\,e^{-2b}}{\e}  \mathfrak{R}\biggl\{ \sum\limits_{n,m}e^{-\frac{i(n+m)\pi}{2}}g^1_{n}\overline{g^2_{m}}\biggr\} \,.
\end{align*}
 We compute, 
\begin{multline*}
      \sum\limits_{n,m}e^{-\frac{i(n+m)\pi}{2}}g^1_{n}\overline{g^2_{m}}=\left(\sum\limits_{n}e^{-\frac{in\pi}{2}}g^1_{n}\right)\overline{\left(\sum\limits_{m}e^{\frac{im\pi}{2}}g^2_{m}\right)}\\
      =g_1\left(\left(\left(0, \frac{1}{\e^2} + b\right) + \frac{1}{\e^2}\left(\cos \left(-\frac{\pi}{2}\right), \sin \left(-\frac{\pi}{2}\right)\right)\right)\right)\\\times{g_2\left(\left(\left(0, -\frac{1}{\e^2} - b\right) + \frac{1}{\e^2}\left(\cos \left(\frac{\pi}{2}\right), \sin \left(\frac{\pi}{2}\right)\right)\right)\right)}={  g_1(p)g_2(q)}\,,
\end{multline*}
 as the point $\theta = \frac{3\pi}2\sim-\frac{\pi}{2}$ corresponds to the bottom tip of the upper circle $\partial B_1$, and $\phi = -\frac{3\pi}2\sim\frac{\pi}{2}$ corresponds to the upper tip of the lower circle $\partial B_2$, and we recall that each $g_i$ is real-valued. The proof of the theorem is completed when boundary conditions $g_1 $ and $g_2$ have no more than the first $M_0$ modes in Fourier space. Sending $M_0 \to \infty$ completes the proof. 
	\end{proof}

\section{Interaction energies of multiple particles}

\label{sec:multiple}
In this subsection we demonstrate how the analysis of the present paper can be extended to multiple particles. We will also indicate how to modify the arguments to permit polydisperse collections of particles. As a first step toward these generalizations, we consider unit balls $\{B_i\}_{i=1}^N,$ with disjoint closures 
\begin{equation*}
    \overline{B}_i \cap \overline{B}_j = \emptyset\,.
\end{equation*}
Denoting the center of the disk $B_i$ via $a_i^\e,$ for any $i,j \in \{1,\cdots, N\}$ we define $b_{ij} >0$  via
\begin{equation*}
    b_{ij} := \frac{|a_i^\e - a_j^\e| - 2}{\e^2}\,.
\end{equation*}
We consider obtaining an energy expansion to the solution $u_\e \in H^1\bigl(\R^2 \setminus \bigcup_{i=1}^N \overline{B}_i\bigr)$ of the problem
\begin{equation} \label{e.manyparticles}
\begin{aligned}
     \Delta u_\e&=\frac{1}{\e^4} u_\e \quad \mbox{ in } \R^2 \setminus \bigcup_{i=1}^N \overline{B}_i\\
     u_\e&= g_i \quad \mbox{ on } \partial B_i\,.
\end{aligned}
\end{equation}
For simplicity, we focus on the case where the boundary conditions $g_\e^i$ are all constant; the generalization of the discussion here to nonconstant $g_\e^i$ can then be easily carried out. 

Following~\eqref{eq:sipa} we let 
\begin{equation*}
    \Psi_i(x) := \frac{K_0\bigl( \frac{|x-a_i^\e|}{\e^2} \bigr)}{K_0(\e^{-2})}\,, \quad i = 1, \cdots, N; x \in \R^2 \setminus \bigcup_{i=1}^N \overline{B}_i\,.
\end{equation*}
Here $a_i^\e$ is the center of the disk $B_i.$ We also introduce $R_i \in H^1\bigl(\R^2 \setminus \bigcup_{i=1}^N \overline{B}_i\bigr)$ denote the unique solution to 
\begin{equation*}
    \begin{aligned}
        \Delta R_i &= \frac1{\e^4} R_i \quad \mbox{ in } \R^2 \setminus \bigcup_{i=1}^N \overline{B}_i\\
        R_i &= 1 \qquad  \mbox{ on } \partial B_i\\
        R_i &= 0 \qquad \mbox{ on } \partial B_j, j \neq i\,.
    \end{aligned}
\end{equation*}
Finally, we define 
\begin{equation*}
    Z_i = R_i - \Psi_i\,.
\end{equation*}
Then, it is clear that the unique solution to~\eqref{e.manyparticles} is given by 
\begin{equation*}
    u_\e = \sum_{i=1}^N g_i R_i = \sum_{i=1}^N g_i (\Psi_i + Z_i)\,.
\end{equation*}
Then, by analogy with~$\eqref{e.ZjPDE},$ the function $Z_i$ introduced here is a solution to the linear PDE of interest, which vanishes on the $i$th disk, and is equal to the negative of the single particle solution $\Psi_i$ on all other balls. The analogy of~\eqref{e.sol} and~\eqref{e.decompofenergy}  is then apparent, and we find that 
\begin{equation*}
2F_\e (u_\e) = 
    \begin{pmatrix}
        g_1 & \cdots & g_N
    \end{pmatrix}\begin{pmatrix}
        \|\Psi_1 + Z_1\|_\e^2 & \cdots & \langle \Psi_1 + Z_1, \Psi_n + Z_n\rangle_\e \\
        \vdots & \ddots & \vdots \\
        \langle \Psi_n + Z_n , \Psi_1 + Z_1 \rangle_\e & \cdots & \|\Psi_n + Z_n\|_\e^2
    \end{pmatrix}\begin{pmatrix}
        g_1\\
        \vdots \\
        g_n
    \end{pmatrix}
\end{equation*}
From this, arguing as in the two particle case, it is clear that the energy of the minimiser is concentrated in the necks to first order in an energy expansion: the leading order is
\begin{equation*}
    \frac{2\pi}{\e^2} \sum_{i=1}^N \int_{\partial B_i}|g_i|^2\,d\sh^1\,,
\end{equation*}
the next order contribution is $O(K_0(2b)/\e)$ arises from nearest neighbors from the neck in between such neighbors. 

The case of polydisperse particles is also similar to handle: namely, if the particle radii vary between $[\rho_{min}, \rho_{max}]$ for some $\rho_{min}, \rho_{max} = O(1)$ in $\e,$ then one simply defines $\Psi_i$ and $R_i$ as above, accordingly. 

%\bigskip 
%\dg{Systems of paranematic particles: Assume Jamie's ansatz of a single mode on $\partial B$ and let all particles be of the same mode; what type of periodic lattices can be constructed of such particles? Conjecture: n=even - a square lattice? n=3k - a triangular lattice? n=6k - hexagonal lattice? For other n's - almost periodic? Amorphous? Chains? Any ideas from granular matter that might be applicable? Possible MD-like simulations?}

	\section{Numerics and comparison to nonlinear models }
 \label{sec:numerics}

In this section we use numerical simulations in COMSOL, \cite{comsol}, to verify the asymptotics established in the previous sections. We begin by considering the asymptotic expansion \eqref{eq:expans}. Recalling that the self-energy of a single particle is given by
\[\frac{2\pi}{\e^2}\frac{K_1\left(\frac{1}{\e^2}\right)}{K_0\left(\frac{1}{\e^2}\right)}=\frac{2\pi}{\e^2}\left(1+\frac12\e^2+o\left(\e^2\right)\right)=\frac{2\pi}{\e^2}+\pi+o(1),\]
we will replace the $O(1)$-term in \eqref{eq:expans} by $\pi$ and let 
\begin{multline}
\label{eq:expans1}
	\bar\kappa_\e=\frac{2\pi}{\e^2}\left(g_1^2+g_2^2\right)+\frac{\sqrt{\pi/2}}{\e}\left[\Li\left(e^{-4b}\right)+\Theta_4\left(e^{-2b}\right)+e^{-4b}\right](g_1^2+g_2^2)\\-\frac{4\sqrt{\pi}}{\e}\left[e^{-2b}+\frac{1}{\sqrt{2}}\Theta_3\left(e^{-2b}\right)\right]g_1g_2+\pi\,.
\end{multline}
Fig.~\ref{fig:res1} shows the dependence of $\kappa_\e$ and $\bar\kappa_\e$ on $b$ for $\e=0.15$ when $g_1\equiv1$ on $\partial B_1$ and $g_2\equiv0$ on $\partial B_2.$
\begin{figure}
    \centering
    \includegraphics[width=3.5in]{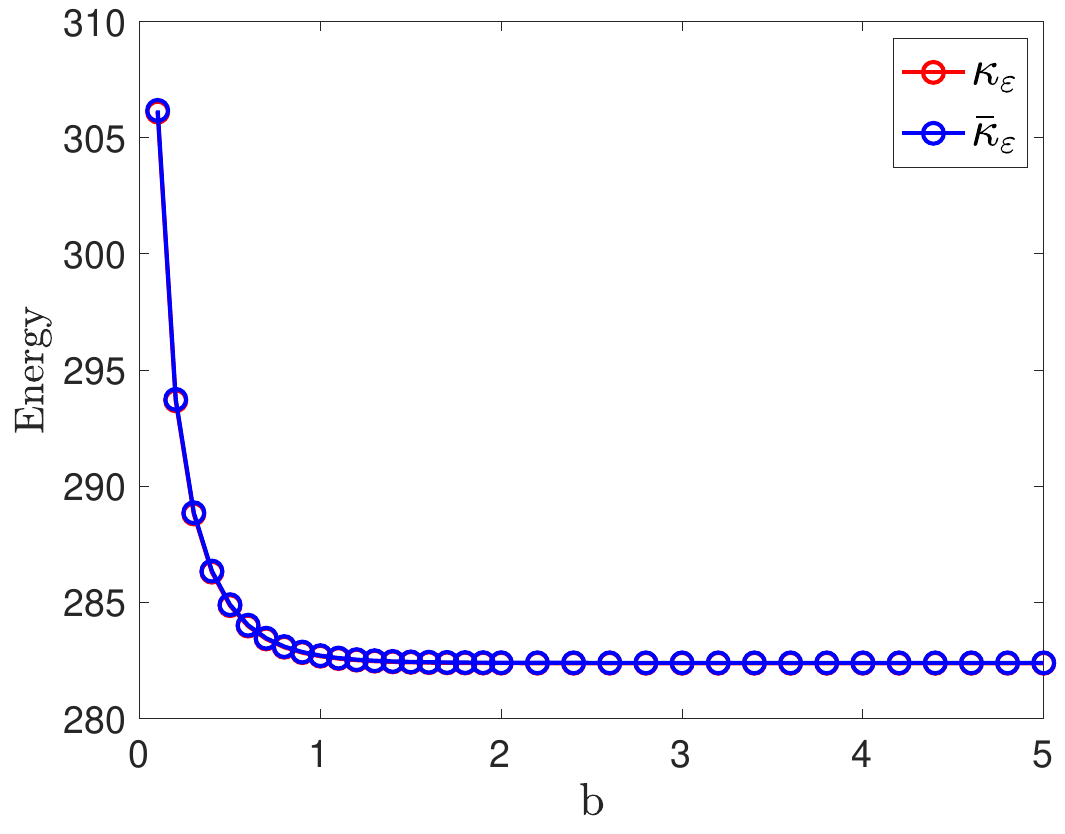}
    \caption{Comparison between $\kappa_\e$ and $\bar\kappa_\e$ when $g_1\equiv1$ and $g_2\equiv0$.}
    \label{fig:res1}
\end{figure}

Fig.~\ref{fig:res2} shows the dependence of $\kappa_\e$ and $\bar\kappa_\e$ on $b$ for $\e=0.15$ when $g_1\equiv1$ on $\partial B_1$ and $g_2\equiv1$ on $\partial B_2.$
\begin{figure}
    \centering
    \includegraphics[width=3.5in]{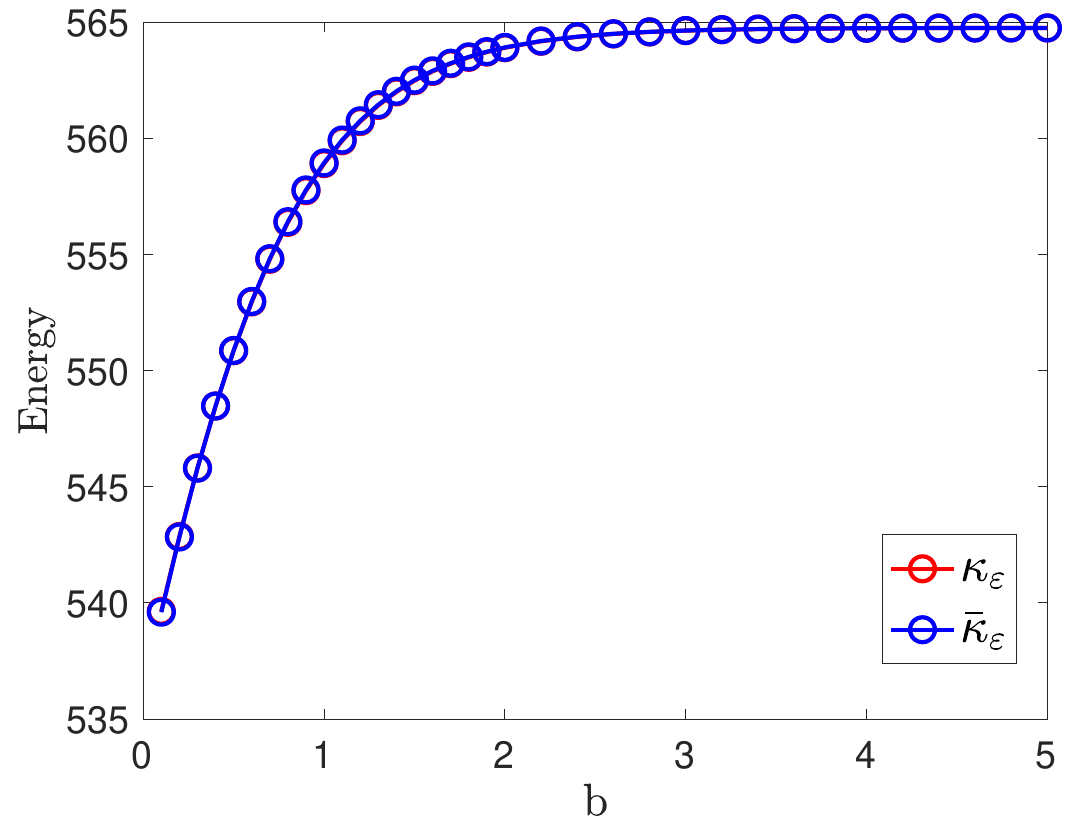}
    \caption{Comparison between $\kappa_\e$ and $\bar\kappa_\e$ when $g_1\equiv1$ and $g_2\equiv1$.}
    \label{fig:res2}
\end{figure}

Fig.~\ref{fig:res3} shows the dependence of $\kappa_\e$ and $\bar\kappa_\e$ on $b$ for $\e=0.15$ when $g_1\equiv1$ on $\partial B_1$ and $g_2\equiv-1$ on $\partial B_2.$
\begin{figure}
    \centering
    \includegraphics[width=3.5in]{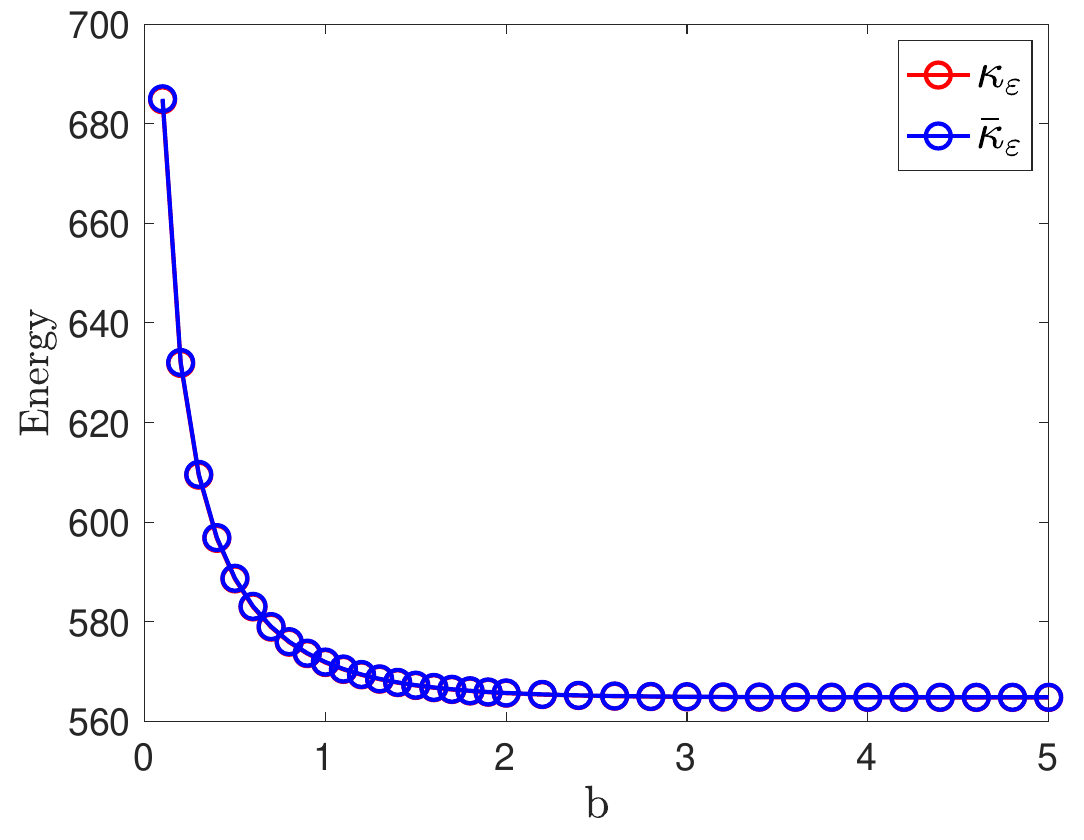}
    \caption{Comparison between $\kappa_\e$ and $\bar\kappa_\e$ when $g_1\equiv1$ and $g_2\equiv-1$.}
    \label{fig:res3}
\end{figure}
From  Figs.~\ref{fig:res1}-\ref{fig:res3} we conclude that our asymptotics are, in fact, accurate up to $o(1).$ We also observe that for certain combinations of $g_1$ and $g_2$ the form of the $(b,\bar\kappa_\e)$-dependence can be of the Lennard-Jones-type as shown in Fig.~\ref{fig:res12}.
\begin{figure}
    \centering
    \includegraphics[width=3.5in]{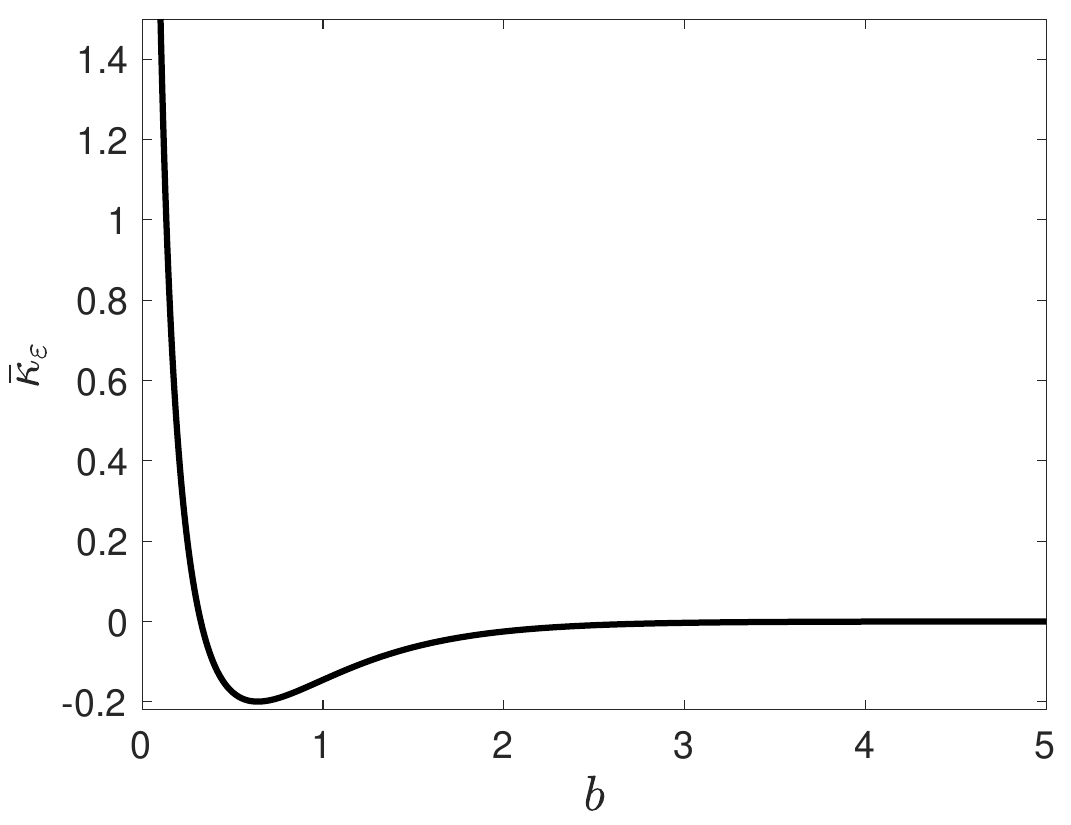}
    \caption{$\bar\kappa_\e(b)$ when $g_1\equiv1$ and $g_2\equiv0.2$.}
    \label{fig:res12}
\end{figure}

The comparison between the energies of minimizers of the full nonlinear and linear problems are shown in Fig.~\ref{fig:res11} for $\e=0.15$ and two different choices of boundary data when $k(T)=2.$ 
\begin{figure}
    \centering
    \includegraphics[width=2.9in]{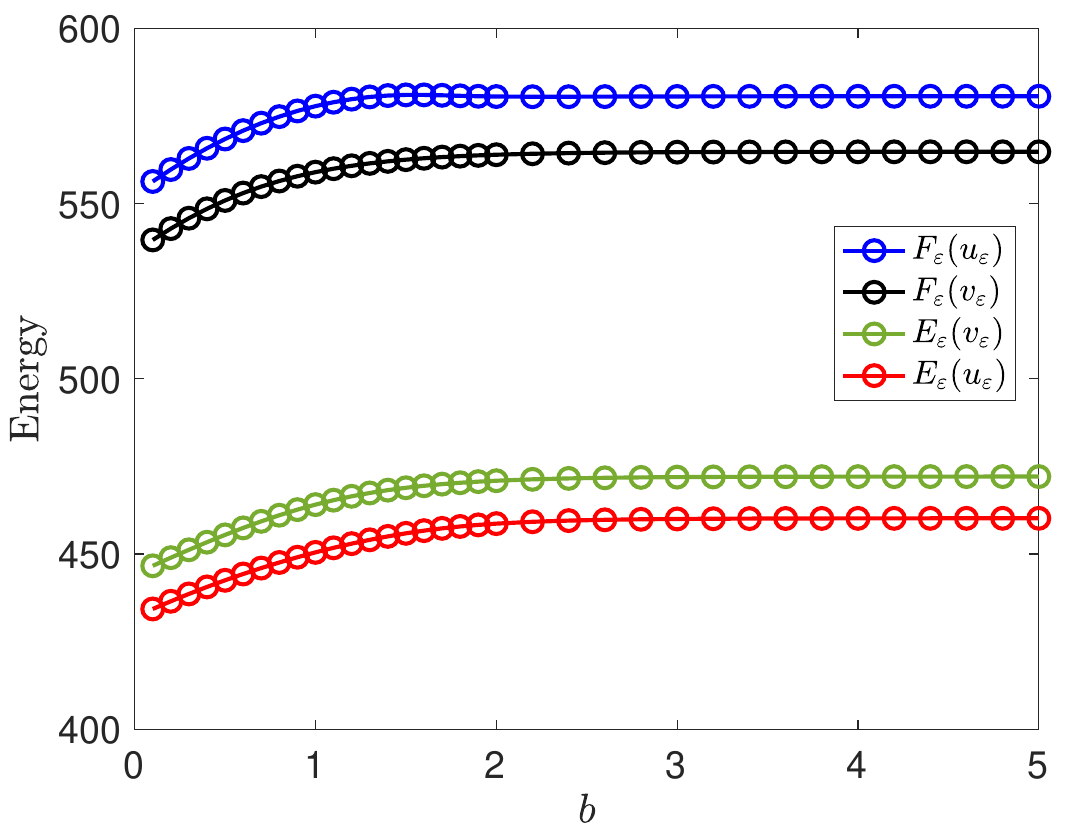}\quad \includegraphics[width=2.9in]{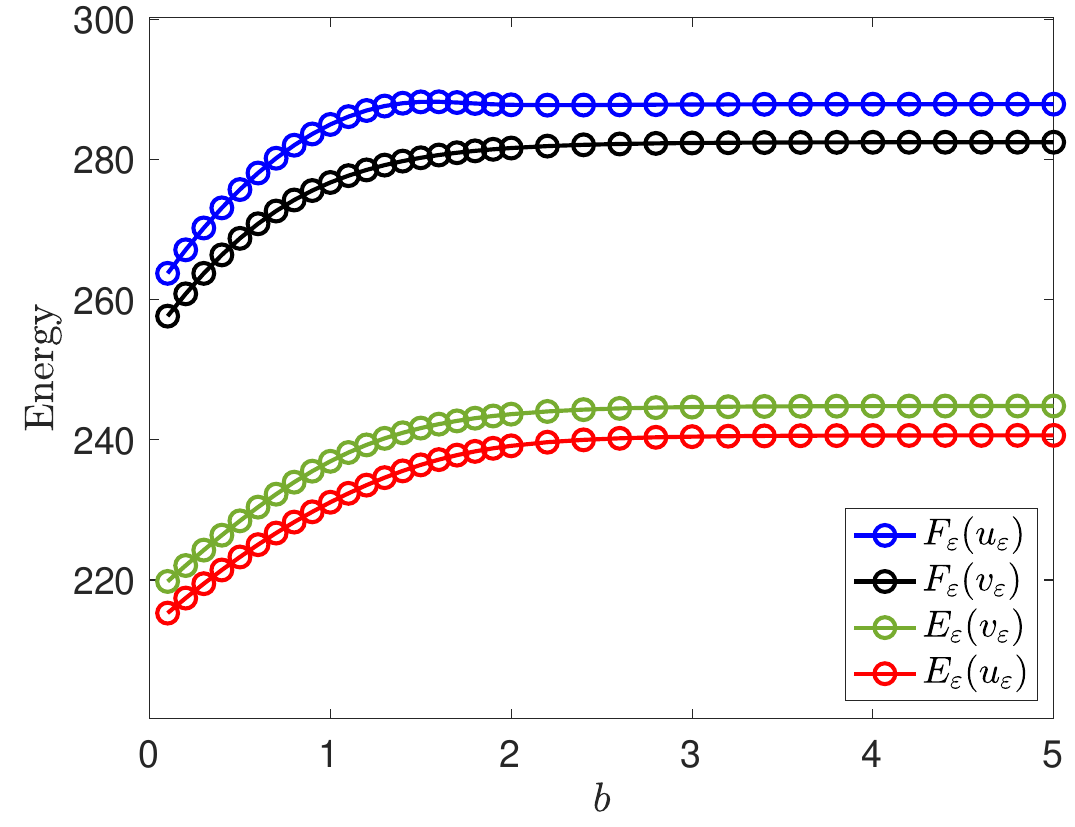}
    \caption{Comparison between the energies of minimizers for the nonlinear and linear problems when $g_1\equiv1$ and $g_2\equiv1$ (left) and $g_1\equiv\cos{\theta}$ and $g_2\equiv-\cos{\theta}$ (right). Here $u_\e=\argmin{E_\e(u)}$ and $v_\e=\argmin{F_\e(u)}.$}
    \label{fig:res11}
\end{figure}
Note that the difference between the minimum $F_\e(v_\e)$ of the quadratic energy and $E_\e(u_\e)$ of its nonlinear counterpart is roughly a constant, hence the interaction forces between the two particles in the nonlinear and linear regimes are approximately the same.
\begin{comment}
The difference between the $L^2$-norms of these minimizers and their gradients are depicted in Fig.~\ref{fig:res11}. The data shows that ${\|u_\e-v_\e\|}_{L^2}=\sout{\mathcal{O}}\red{O}(\e)$ and ${\|\nabla u_\e-\nabla v_\e\|}_{L^2}=\sout{\mathcal{O}}\red{O}\left(\e^{-1}\right).$ Here $u_\e=\argmin{E_\e(u)}$ and $v_\e=\argmin{F_\e(u)}.$
\begin{figure}
    \centering
    \includegraphics[width=2.9in]{l2u.pdf}\quad \includegraphics[width=2.8in]{l2gradu.pdf}
    \caption{Comparison between the minimizers for the nonlinear and linear problems when $g_1\equiv\cos{\theta}$ and $g_2\equiv-\cos{\theta}$ and $b=1$. Here $u_\e=\argmin{E_\e(u)}$ and $v_\e=\argmin{F_\e(u)}.$}
    \label{fig:res4}
\end{figure}
\end{comment}

Next, we consider a system of three particles with $g_1=g_2=g_3\equiv1$ on $\partial B_1,$ $\partial B_2$ and $\partial B_3,$ respectively. To this end, we denote by 
\[\kappa_\e^0:=\frac{3\pi}{2}\left(\frac{4}{\e^2}+1\right)\]
the self-energy of three particles and 
\[\kappa_\e^1(b):=\frac{\sqrt{2\pi}}{\e}\left[\Li\left(e^{-4b}\right)+\Theta_4\left(e^{-2b}\right)+e^{-4}\right]\\-\frac{4\sqrt{\pi}}{\e}\left[e^{-2b}+\frac{1}{\sqrt{2}}\Theta_3\left(e^{-2b}\right)\right]\]
the interaction energy of a single neck between the two particles on the distance $2\e^2b$ from each other (cf. \eqref{eq:expans1}).

In the first numerical experiment, we assume that three particles are positioned at the vertices of a equilateral triangle, where the distance between the centers of any pair of particles is $2+2b\e^2.$ Assuming that the interactions are restricted to the necks, the minimum energy of this configuration should be
\[\kappa_\e\sim\bar\kappa_\e:=\kappa_\e^0+3\kappa_\e^1.\]
Fig.~\ref{fig:res5} demonstrates that this is indeed the case as the graphs of $\kappa_\e$ and $\bar\kappa_\e$ as functions of $b$ are essentially indistinguishable.
\begin{figure}
    \centering
    \includegraphics[width=2.9in]{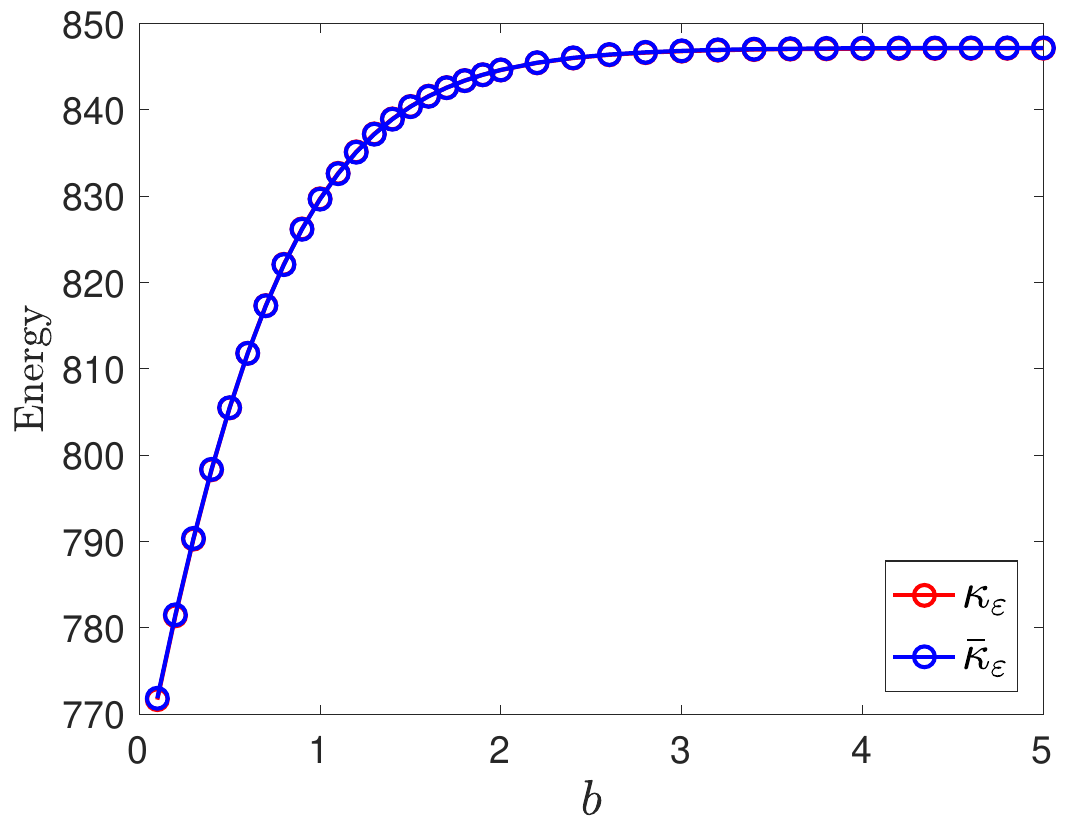}\quad \includegraphics[width=2.9in]{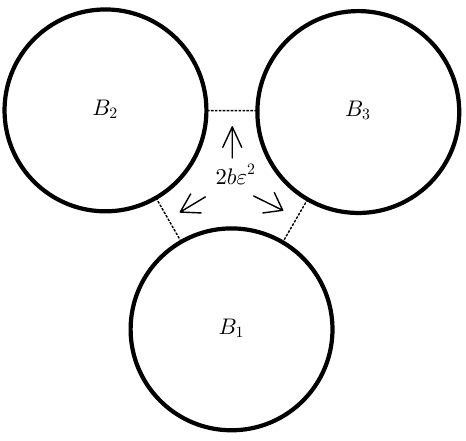}
    \caption{Comparison between $\kappa_\e$ and $\bar\kappa_\e$ (left) for a three-particle triangular configuration depicted on the right.}
    \label{fig:res5}
\end{figure}
In Fig.~\ref{fig:res6},
\begin{figure}
    \centering
    \includegraphics[width=3in]{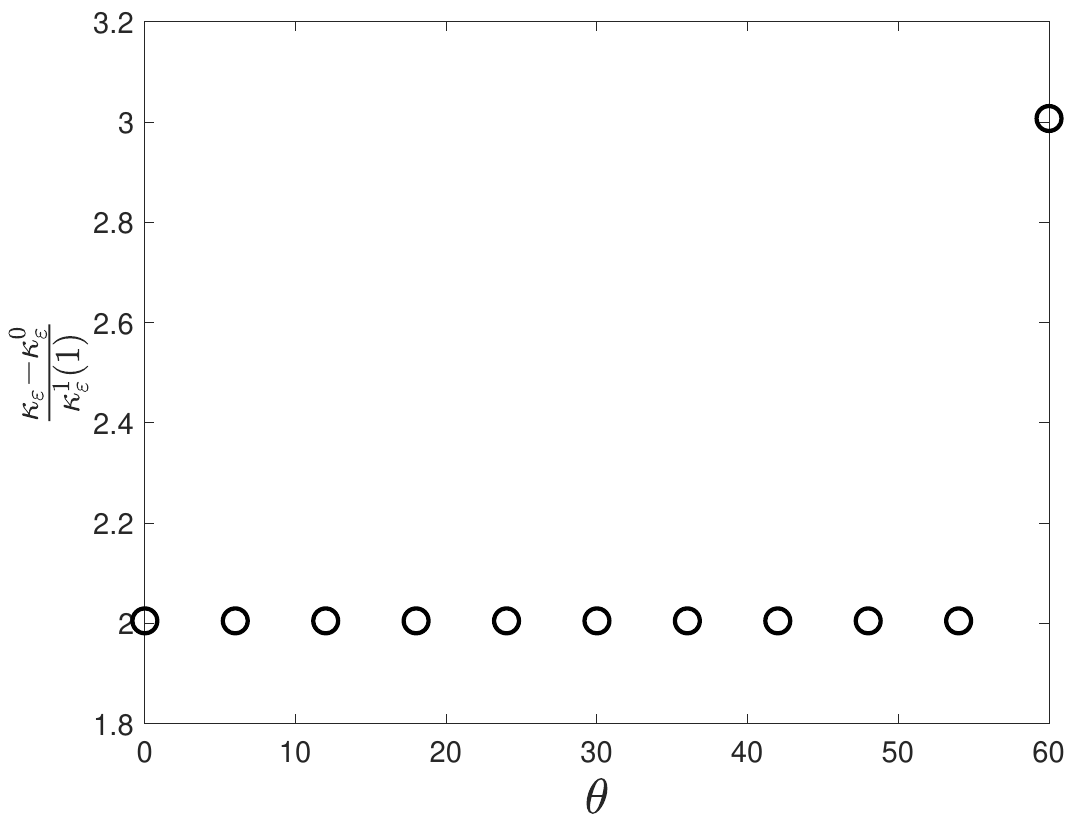}\\ \includegraphics[width=4in]{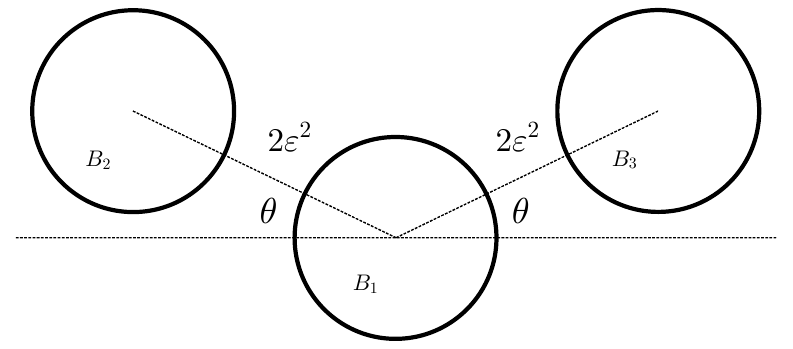}
    \caption{The interaction energy between the particles forming the configuration shown on the bottom as a function of $\theta$ (top).}
    \label{fig:res6}
\end{figure}
we consider a configuration of three particles where the distances between $B_1$ and $B_2$ and $B_1$ and $B_3$ are fixed and equal to $2\e^2,$ while the angle $\theta$ ranges from $0$ to $60$ degrees. We plot the ratio between the interaction energy between the three particles and the energy of a single neck. As it can be seen from Fig.~\ref{fig:res6}, when the angle is less than $60$ degrees, then there are exactly two necks and, indeed, the interaction energy is equal to exactly two neck energies. When the angle is equal to $60$ degrees, then the third neck forms and the interaction energy is equal to the three energies of a single neck. We conjecture that for all boundary conditions on the surfaces of particles, the energy of interaction between the particles is concentrated in the necks.

\section{Monte Carlo simulation of the pairwise energy}
\label{sec:MonteCarlo}
%\red{The stuff on single modes has all been commented out, does that mean that we want to remove it, or re-write according to the localised new version of the two-particle energy? - JMT}

In this section we will consider a many-body system of identical particles satisfying given canonical degree $d$ boundary data $g_k(\theta) = e^{i(d\theta - (d-1)\omega_k)}$ and which we assume are free to move and rotate, so that each particle has degrees of freedom corresponding to its centre of mass and the angle $\omega_k$. Following our discussion in the preceding section, we assume that there are only pairwise interactions in this system. Assuming that the distances between particles are large enough and plugging in special choices of $g_1(\theta) = e^{i(d\theta - (d-1)\omega_1)}$ and $g_2(\theta) = e^{i(d\theta - (d-1)\omega_2)}$ in the statement of Theorem~\ref{t.mt}, we obtain that the interaction energy between a pair of particles of radius $1$ at distance $2b\e^2$  is given by 
\begin{equation}
\label{eq:j1}
    V_{12}= -(-1)^d\frac{2e^{-2b}\sqrt{\pi}}{\e}\cos((d-1)(\omega_1 - \omega_2))\,. 
\end{equation}
We note a qualitative difference in behaviour depending on the parity of $d$. When $d$ is even, the energy is minimised at parallel configurations, with particles at relative angle of $0$. Thus, we expect it to be favourable for particles to be closely packed with similar orientations. If $d$ is odd, however, then the energy is minimised at anti-parallel configurations, where particles are at a relative angle of $\frac{\pi}{d-1}$, modulo $\frac{2\pi}{d-1}$. As the interactions are short-range, we expect only interactions with nearest neighbours to be significant. Heuristically, it seems clear that configurations of square-like lattices with second-nearest-neighbours having the same orientation, whilst nearest neighbours are at a relative angle of $\frac{\pi}{d-1}$, should be relatively stable. 

We will consider the pairwise interaction energy for particles with orientations $\omega_1,\omega_2$ and centres of mass separated by $r$ as given by 
\begin{equation}
V(r,\omega_1,\omega_2)=\left\{\begin{array}{l l}
(-1)^{d+1}\exp\left(-\frac{|r|-2}{\epsilon^2}\right)\cos((d-1)(\omega_1-\omega_2)) & |\hat{r}|>2\\
+\infty & |\hat{r}|\leq 2\end{array}\right. 
\end{equation}
Multiplicative factors that do not affect minimisers of the energy are neglected for simplicity. The infinite energy for $|\hat{r}|\leq 2$ corresponds to the the particles being unable to interpenetrate. Of course, the pairwise interaction in \eqref{eq:j1} corresponds to an asymptotic limit, and thus we are required to introduce an appropriate length-scale for the simulation, corresponding to the choice of $\epsilon$, which we take to be $\sqrt{\frac{2}{5}}$, as we found in preliminary studies that even marginally smaller values of $\epsilon$ lead to interactions too weak to produce any noticeable structure. The total pairwise energy is then given by
\[
\frac{1}{2}\sum\limits_{i\neq j}V(x_i-x_j,\omega_i,\omega_j).
\]

We employ a simulated annealing algorithm, with the transition probabilities taken from the corresponding Gibbs' distribution of the system, that is, those of a Metropolis-Hastings algorithm, using 256 particles.  At each temperature, we perform one Monte Carlo iteration to each particle, randomly permuting the order of the particles at each temperature. We perturb the centre of mass of the particle with index $i$ according to a Gaussian distribution with mean $0$ and standard deviation given by $\min(\max(0.025,\delta_i),0.5)$, where $\delta_i=\min\limits_{j\neq i}(|x_i-x_j|-2)$ is the minimal contact distance to another particle. We perturb the angle according to a normal distribution with mean 0 and standard deviation $\frac{2\pi}{50}$. { We} linearly decrease the temperature from $\frac{1}{4}$ to $0$ over 25000 steps{, yielding a total of $25000\times 256$ individual Monte Carlo iterations. The particles are} initialised as a perturbation of a square lattice with nearest-neighbour separation of $2.2$, and orientations taken according to a uniform distribution. Finally, due to the short-range nature of the interactions, at higher temperatures it is easy for particles to drift large distances, at which point their behaviour becomes a random walk and ceases to effectively interact with the rest of the system, so we impose that particles cannot leave a box of size $46\times 46$. { This is effectively imposed as an infinite confining potential $U$, taken as a function of the centre of mass, so that for the domain $\Omega=(0,46)^2$, $U(x)=\infty$ if $x\not\in \Omega$ and $U(x)=0$ otherwise. As the particles have effective radius of 1, the effective area density of the system is approximately $0.38$}

Whilst simulated annealing is generally used to find global minimisers, as the particles are very weakly interacting, we expect a relatively flat energy landscape that permits large fluctuations away from the global minimiser.

In Figure \ref{fig.MHMC} we present the results of the simulations. The disks represent the individual particles, and the lines within them represent their orientation, and are illustrated such as to be consistent with the axes of symmetry of the particles. Furthermore, we colour the particles according to their angle modulo $|\frac{2\pi}{d-1}|$, on an RGB colour-scale, corresponding to the symmetry of their boundary condition. 

In Subfigure \ref{subfig.MHMC.3} we have odd-degree boundary conditions, and thus by the interaction energy \eqref{eq:j1}, we expect to have an anti-parallel configuration, where neighbours in close contact are rotated, but second-nearest neighbours have the same orientation, and this is observed. Even though a square lattice can be expected via a heuristic argument, we observe that the distribution of centres of mass is relatively amorphous, with some mild amount of short-range correlation. We observe several chain-like structures, with anti-parallel alignment with nearest neighbours. Due to the short-range nature of the interaction, these are expected to be locally stable, whilst denser configurations in a square-like lattice would have lower energy. This is similar to the situation in Subfigure \ref{subfig.MHMC.5}, which also exhibits clearly visible well-aligned domains with anti-parallel configurations within. Notably, we observe that the degree 5 configuration is more amorphous with more chain-like structures, which we aim to explain via a heuristic argument. In the degree 5 case, we need to consider relative angles modulo $\frac{\pi}{2}$, whilst in the degree 3 case we consider relative angles modulo $\pi$. This smaller range of angles would suggest a higher sensitivity to small perturbations in the orientation, making it more difficult for the particles to align into their optimum states and leading to amorphousness. 

In Subfigure \ref{subfig.MHMC.2}, our pairwise energy favours nearest neighbours having the same orientation, and many contacts with neighbours, which is the observed behaviour. Although particles are generally well-aligned with their neighbours, we observe a kind of polycrystalline structure with clearly identifiable domains. These are expected to be locally stable, as reorienting a single grain would require simultaneous reorientation of many particles. 

To demonstrate more clearly the local parallel and anti-parallel configurations, we include histograms in Figure \ref{fig.MHMC_angles} of the relative orientations of particles with their nearest neighbours and second-nearest neighbours below, taken modulo $\frac{2\pi}{|d-1|}$. Explicitly, we say that two particles are nearest neighbours if the separation of their centres of mass is less than $2.05$, and that two particles are second-nearest neighbours if they are distinct and share a nearest neighbour. We observe a clear tendency for nearest neighbours to be either parallel or anti-parallel according to the parity of the degree, and parallel alignment of next-nearest neighbours in all cases.

\begin{figure}[ht]\begin{center}
\begin{subfigure}[h]{0.3\textwidth}
\begin{center}
\includegraphics[width=0.99\textwidth]{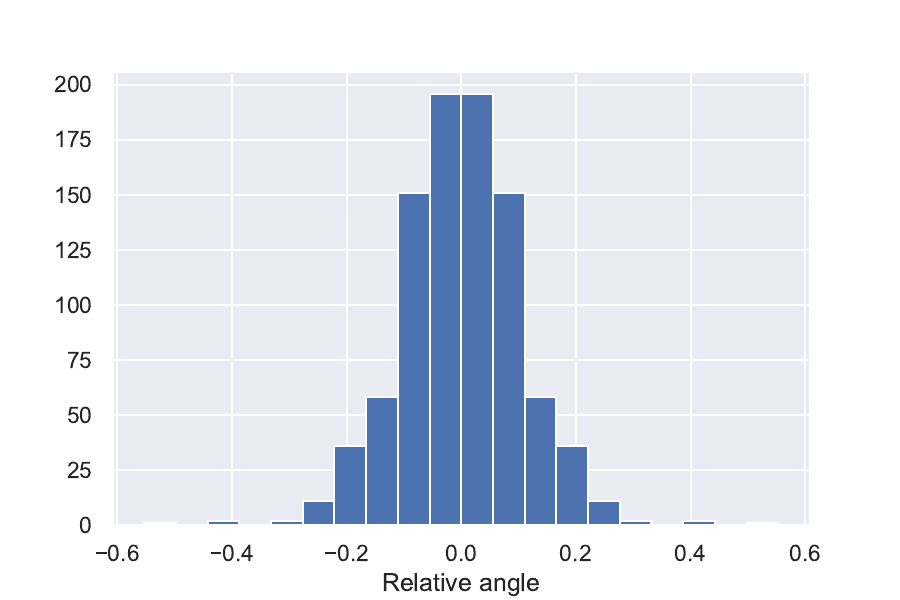}
\caption{Degree 2: Relative angle of nearest neighbours}
\end{center}
\end{subfigure}\hspace{0.1cm}
\begin{subfigure}[h]{0.3\textwidth}
\begin{center}
\includegraphics[width=0.99\textwidth]{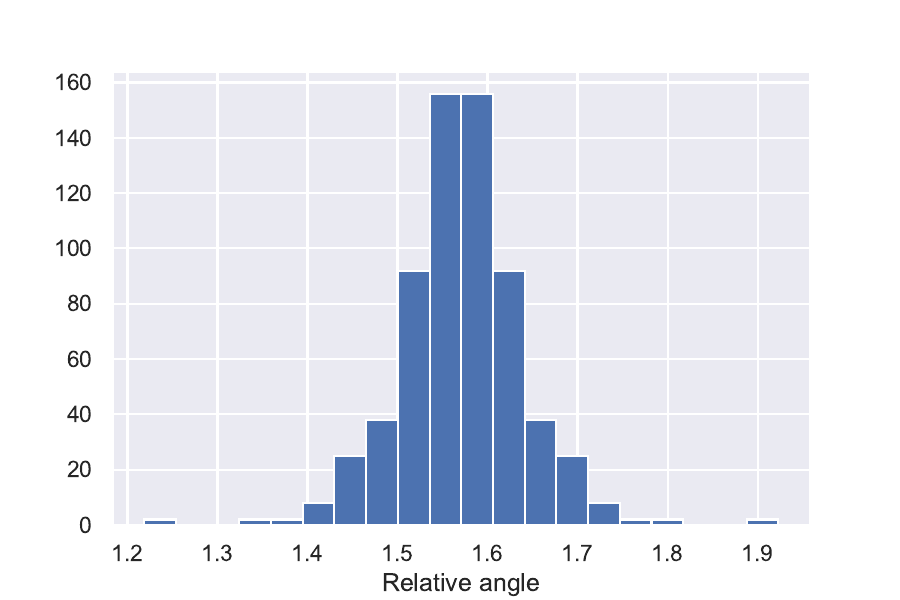}
\caption{Degree 3: Relative angle of nearest neighbours}
\end{center}
\end{subfigure}\hspace{0.1cm}
\begin{subfigure}[h]{0.3\textwidth}
\begin{center}
\includegraphics[width=0.99\textwidth]{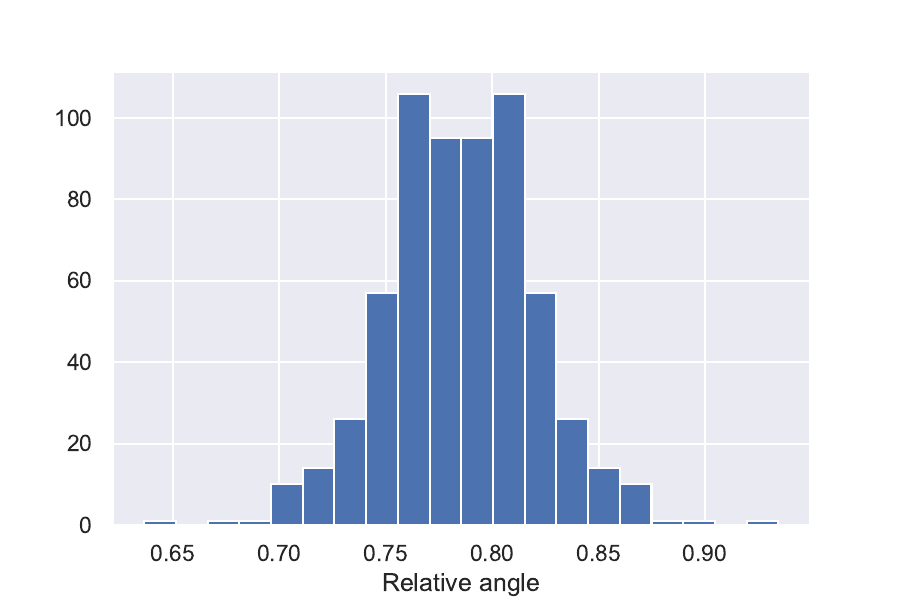}
\caption{Degree 5: Relative angle of nearest neighbours}
\end{center}
\end{subfigure}
\begin{subfigure}[h]{0.3\textwidth}
\begin{center}
\includegraphics[width=0.99\textwidth]{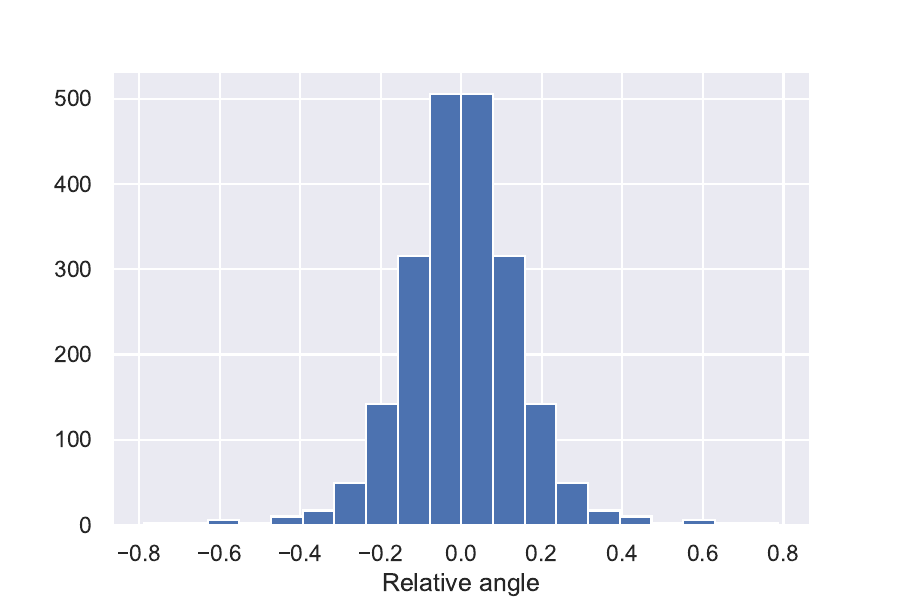}
\caption{Degree 2: Relative angle of next-nearest neighbours}
\end{center}
\end{subfigure}\hspace{0.1cm}
\begin{subfigure}[h]{0.3\textwidth}
\begin{center}
\includegraphics[width=0.99\textwidth]{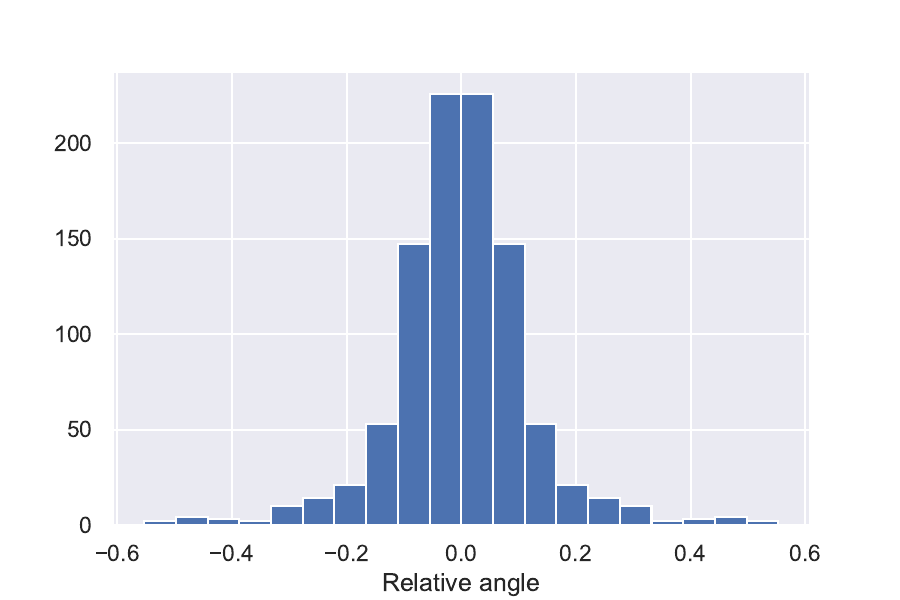}
\caption{Degree 3: Relative angle of next-nearest neighbours}
\end{center}
\end{subfigure}\hspace{0.1cm}
\begin{subfigure}[h]{0.3\textwidth}
\begin{center}
\includegraphics[width=0.99\textwidth]{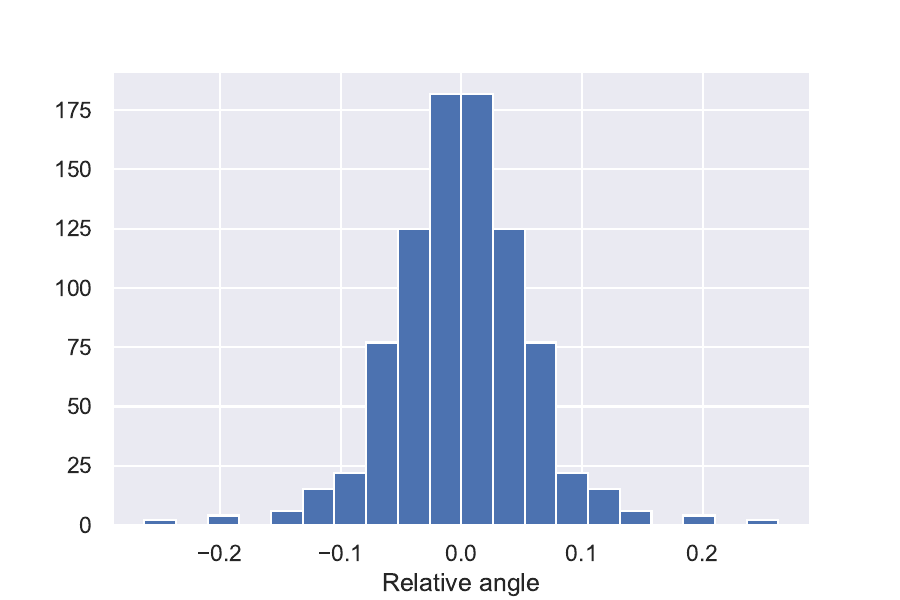}
\caption{Degree 5: Relative angle of next-nearest neighbours}
\end{center}
\end{subfigure}
\caption{Distributions of relative angles of nearest- and next-nearest neigbours}\label{fig.MHMC_angles}
\end{center}
\end{figure}

It is a similarly straightforward exercise to evaluate the angular component of the interaction energy between particles of distinct degrees. If we consider two particles, whose centres of mass are at relative angle $\alpha$, of degrees $d_1,d_2$ and with orientations described by angles $\omega_1,\omega_2$, then we have that 
\begin{equation}
\begin{split}
\Re(g_1(p)\bar{g_2}(q)) =& \Re(e^{i(d_1\alpha-(d_1-1)\omega_1)}e^{-i(d_2(\alpha-\pi)-(d_2-1)\omega_2)})\\
=&\cos((d_1\alpha-(d_1-1)\omega_1)-(d_2(\alpha-\pi)-(d_2-1)\omega_2))\\
=&(-1)^{d_2}\cos((d_1-d_2)\alpha+(d_2-1)\omega_2-(d_1-1)\omega_1),
\end{split}
\end{equation}
where $p$ and $q$ correspond to the closest points on the surface of each respective particle to the other. We remark that unlike the case where both degrees are equal, this depends on the relative position of the particles via $\alpha$, and not just the orientations $\omega_1,\omega_2$. 

We consider a mixed system of degree 1 and degree 3 particles. As seen before, we have that degree 3 particles prefer an anti-parallel alignment. Degree 1 particles are purely repulsive, and due to their rotational symmetry, there is no orientational dependence. For the interactions between degree 1 and degree 3 particles, taking particle 1 to be of degree 3 and particle 2 to be of degree 1, the angular component of the interaction energy is $-\cos(2(\alpha-\omega))$. In particular, their optimal configuration is to have the degree 1 particle at either of the two poles of the degree 3 particle where the director is perpendicular to the surface. We employ a simulated annealing algorithm with the same experimental setup as the previous experiments to obtain the results in Figure \ref{fig.MHMC_mix}. As before, we colour the degree 3 particles according to their angle, modulo $\pi$, with the illustrated diameter spanning the two points where the boundary data is perpendicular to the surface. The degree 1 particles are rotationally symmetric and thus coloured in white.

\begin{figure}[ht]\begin{center}
\begin{subfigure}[h]{0.3\textwidth}
\begin{center}
\includegraphics[width=0.99\textwidth]{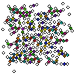}
\caption{Configuration at the end of the simulation}\label{fig.MCMH_sol}
\end{center}
\end{subfigure}\hspace{0.1cm}
\begin{subfigure}[h]{0.3\textwidth}
\begin{center}
\includegraphics[width=0.99\textwidth]{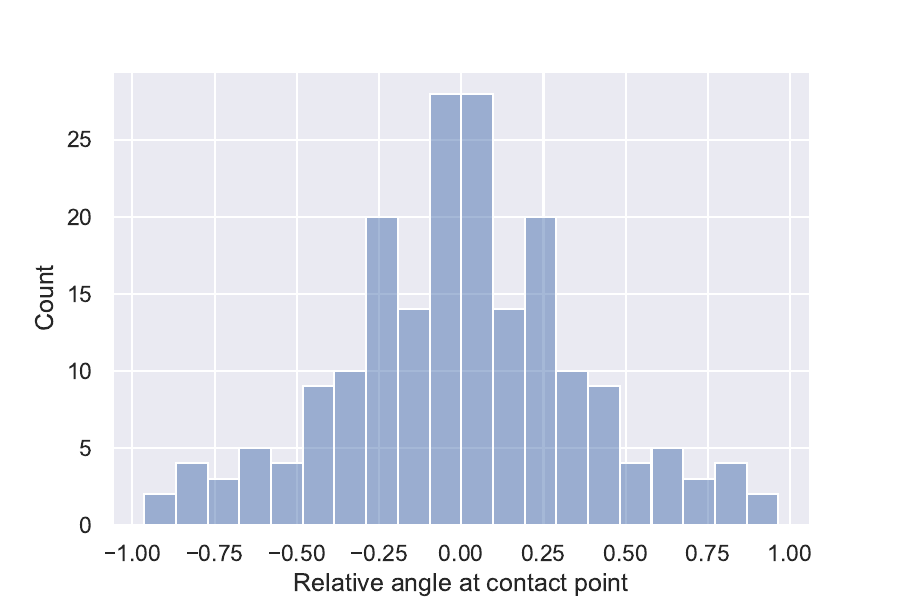}
\caption{Relative angle of the boundary data at the contact point for pairs of degree 1 and degree 3 particles.}\label{fig.MHMC_13}
\end{center}
\end{subfigure}\hspace{0.1cm}
\begin{subfigure}[h]{0.3\textwidth}
\begin{center}
\includegraphics[width=0.99\textwidth]{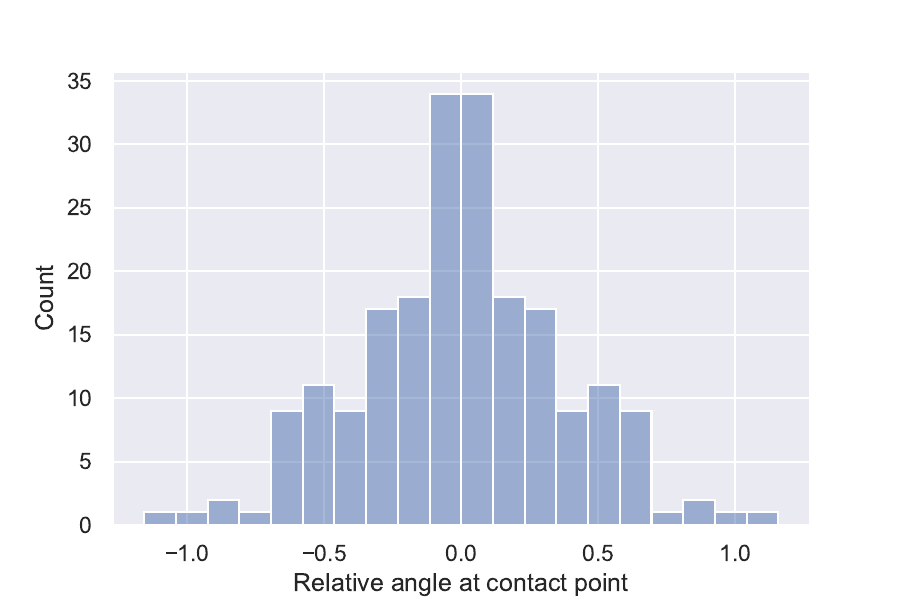}
\caption{Relative angle of the boundary data at the contact point for pairs of degree 3 particles.}
\label{fig.MHMC_33}\end{center}
\end{subfigure}
\caption{Results for a mixed system of degree 1 and degree 3 particles.}\label{fig.MHMC_mix}
\end{center}
\end{figure}

In Subfigure \ref{fig.MCMH_sol}, we see the configuration at the end of the simulation. By eye, we observe qualitatively the expected behaviour of neighbours, where degree 1 particles are separated due to repulsive interactions, neighbouring degree 3 particles tend to be at near-right-angles to each other, and degree 3 and degree 1 particles are roughly aligned along the illustrated diameter, whose end-points correspond to the regions of the surface with perpendicular director. Nonetheless, we observe that the particles are not so well-aligned as in the pure-state case. In particular, we see many triangles consisting of two degree 3 particles and one degree 1 particle, and geometrically such a triangle cannot be pairwise-minimising for the energy. We demonstrate the local orientational ordering graphically by considering the relative angles of the boundary director at the contact point of nearest neighbours in Subfigure \ref{fig.MHMC_13} for pairs of degree 1 and degree 3 particles, and in Subfigure \ref{fig.MHMC_33} for pairs of degree 3 particles, taken modulo $\pi$ in each case. We observe a central tendency at zero, but greater variation than in the case of pure systems. 

\section{Acknowledgements}
 The authors acknowledge the hospitality of HIM: Hausdorff Center for Mathematics where some of the research on this project was conducted during  the Trimester Program Mathematics for Complex Materials funded by the Deutsche Forschungsgemeinschaft
(DFG, German Research Foundation) under Germany Excellence Strategy – EXC-2047/1
– 390685813. DG was supported in part by the NSF grant DMS-2106551. R.V. was partially supported by the Simons Foundation (Award \# 733694) and an AMS-Simons travel award. He also acknowledges the hospitality provided by the Department of Mathematics at the University of Akron when R.V. visited D.G. to complete parts of this project. A.Z. has been partially supported by the Basque
Government through the BERC 2022-2025 program and by the Spanish State Research
Agency through BCAM Severo Ochoa excellence accreditation SEV-2017-0718 and through
project PID2020-114189RB-I00 funded by Agencia Estatal de Investigacion (PID2020-
114189RB-I00 / AEI / 10.13039/501100011033). A.Z. was also partially supported by a
grant of the Ministry of Research, Innovation and Digitization, CNCS - UEFISCDI, project
number PN-III-P4-PCE-2021-0921, within PNCDI III. 

{The authors would thank the anonymous referees for their helpful suggestions on our manuscript.}
	
\bigskip
\par\noindent{\bf\Large Declarations}

\smallskip
\par\noindent {\bf Conflict of interests} The authors have no competing interests to declare that are relevant to the content of this
article.

\par\noindent{\bf Data Availability Statement} The datasets generated during and/or analysed during the current study are available from the corresponding author on reasonable request.

 \appendix

 \vskip 5ex
\par\noindent{\bf\huge Appendix}	

	 	\section{Sobolev spaces and trace theory}
	\label{sec:Sobolev} 
	 There are various definitions for the norms of trace spaces of functions in $H^1(\Omega)$, which are equivalent for sufficiently regular $\Omega$ \cite{gagliardo1957caratterizzazioni}.  In this work we work with the following definitions.
	 
	 \begin{definition}
	 	Let $\Omega$ be a Lipschitz, possibly unbounded, domain with boundary $\partial\Omega$. We define $H^{\sfrac{1}{2}}(\partial\Omega)$ to be the range of the trace operator on $H^1(\Omega)$. For $u_0\in H^{\sfrac{1}{2}}(\partial\Omega)$, we define its norm as 
	 	\begin{align}\label{eq.def.h12}
	 		\|u_0\|_{H^{\sfrac{1}{2}}(\partial\Omega)}=\inf\left\{\|u\|_{H^1(\Omega)}:u|_{\partial\Omega}=u_0 \right\}.
	 	\end{align}
	 \end{definition}
	 It is then immediate that if $u\in H^1(\Omega)$, is the weak solution to  $\Delta u = u$ on $\Omega$, then $\|u\|_{H^1(\Omega)}=\left|\left|u|_{\partial\Omega}\right|\right|_{{H^{\sfrac{1}{2}}} (\partial\Omega)}$. We note that this definition is distinct from the typical one employing the Gagliardo (semi-)norm,
	 \begin{align*}
	 	\|u_0\|_{H^{\sfrac{1}{2}}_G(\partial\Omega)}^2=&\|u_0\|_{L^2(\partial\Omega)}^2+\int_{\partial\Omega}\int_{\partial\Omega} \frac{|u_0(x)-u_0(y)|^2}{|x-y|^2}\,dx\,dy,
	 \end{align*}
	 and instead corresponds to the interpretation of the trace space of $H^1(\Omega)$ as the quotient space $H^1(\Omega)/H^1_0(\Omega)$, where \eqref{eq.def.h12} corresponds to the induced norm on the quotient space. In the case of bounded and Lipschitz domains, these norms are known to be equivalent \cite{gagliardo1957caratterizzazioni}, however in the case of unbounded domains, relevant in this work, this appears to be a folklore theorem, so we include a proof for completeness. 
	 
\begin{proposition}
Let $\Omega\subset\mathbb{R}^n$ be an exterior domain, i.e., $\Omega^c$ is a bounded, Lipschitz domain.  Then $||\cdot||_{H^{\sfrac{1}{2}}_G(\Gamma)}\sim||\cdot||_{H^{\sfrac{1}{2}}(\partial\Omega)}$. 
\end{proposition}
\begin{proof}

Take $B$ to be a disk such that $\Omega^c\subset\subset B$. We define $\hat{\Omega}=\Omega\cap B$, which is then a bounded, Lipschitz domain. Given $u_0\in H^\frac{1}{2}(\partial\Omega)$, define $Eu_0$ to be its extension by $0$ to $\partial\hat{\Omega}=\hat{\Gamma}$, so that $Eu_0|_{\partial\Omega}=u_0$ and $Eu_0|_{\partial B}=0$. Our proof strategy is to show the chain of equivalences, 
\begin{equation}
||u_0||_{H^{\sfrac{1}{2}}_G(\Gamma)} \sim ||Eu_0||_{H^{\sfrac{1}{2}}_G(\hat{\Gamma})}\sim ||Eu_0||_{H^{\sfrac{1}{2}}(\partial\hat{\Omega})}\sim ||u_0||_{H^{\sfrac{1}{2}}(\partial\Omega)},
\end{equation}
where $A\sim B$ implies the existence of some $C>1$ with $\frac{1}{C}A\leq B \leq CA$. 

First, we turn to $||u_0||_{H^{\sfrac{1}{2}}_G(\Gamma)} \sim ||Eu_0||_{H^{\sfrac{1}{2}}_G(\hat{\Gamma})}$. It is immediate, following the definition of the norm, that $||Eu_0||_{H^{\sfrac{1}{2}}_G(\hat{\Gamma})}\geq ||u_0||_{H^{\sfrac{1}{2}}_G(\Gamma)}$. To obtain the converse estimate, we note that 
\begin{equation}
\begin{split}
||Eu_0||_{H^{\sfrac{1}{2}}_G(\hat{\Gamma})}^2=&||u_0||_{L^2(\Gamma)}^2+\int_\Gamma\int_\Gamma \frac{|u(x)-u(y)|^2}{|x-y|^2}\,dx\,dy+2\int_\Gamma\int_{\partial B} \frac{|u(x)|^2}{|x-y|^2}\,dx\,dy\\
\leq &||u_0||_{L^2(\Gamma)}^2+\int_\Gamma\int_\Gamma \frac{|u(x)-u(y)|^2}{|x-y|^2}\,dx\,dy+\frac{2|\partial B|}{d(\partial B,\Gamma)^2}\int_{\Gamma}|u(x)|^2\,dx\\
\leq & \left(1+\frac{2|\partial B|}{d(\partial B,\Gamma)^2}\right)||u_0||_{H^{\sfrac{1}{2}}_G(\Gamma)}. 
\end{split}
\end{equation}

The relationship $||Eu_0||_{H^{\sfrac{1}{2}}_G(\hat{\Gamma})}\sim ||Eu_0||_{H^{\sfrac{1}{2}}(\partial\hat{\Omega})}$ is given in \cite{gagliardo1957caratterizzazioni}, as $\hat{\Omega}$ is a bounded, Lipschitz domain.

Finally, we demonstrate $||Eu_0||_{H^{\sfrac{1}{2}}(\partial\hat{\Omega})}\sim ||u_0||_{H^{\sfrac{1}{2}}(\partial\Omega)}$. For $u_0$ in the trace space of $H^1(\hat{\Omega})$ with $u_0|_{\partial B}=0$, there exists a minimiser for the infima that defines $||u_0||_{H^{\sfrac{1}{2}}(\partial\hat{\Omega})}$, which has trace equal to zero on $\partial B$. In particular, it may be extended by zero to give a $W^{1,2}(\Omega)$ function of equal $H^1$ norm, and may be used as a trial function for $||u_0||_{H^{\sfrac{1}{2}}(\partial{\Omega})}$. This implies that $||u_0||_{H^{\sfrac{1}{2}}(\partial\Omega)}\leq||Eu_0||_{H^{\sfrac{1}{2}}(\partial\hat{\Omega})}$.

For the converse estimate, let $\varphi\in C^\infty(\Omega)$ satisfy $\varphi=1$ in a vicinity of $\Gamma$, and $\text{supp }(\varphi)\subset\subset B$. Now for any $u\in H^1(\Omega)$ with trace $u_0$ on $\Gamma$, $\varphi u$ is an acceptable trial function for the minimisation problem defining $||Eu_0||_{H^{\sfrac{1}{2}}(\partial\hat{\Omega})}$. As $\varphi$ is smooth with compact support, however, this means that $||\varphi u||_{H^1(\hat{\Omega})}=||\varphi u||_{H^1(\Omega)}\leq C||u||_{H^1(\Omega)}$, where $C$ depends only on the $C^1$ norm of $\varphi$ and $B$. Thus $||Eu_0||_{H^{\sfrac{1}{2}}(\partial\hat{\Omega})}\leq C||u_0||_{H^{\sfrac{1}{2}}(\partial{\Omega})}$.

\end{proof}

	 \begin{definition}\label{def.normal.comp}
	 	The space $H^{-\sfrac{1}{2}}(\partial\Omega)$ is defined to be the dual space of $H^{\sfrac{1}{2}}(\partial\Omega)$. Furthermore, for any vector field $v\in L^2(\Omega)$ with $\text{div}(v)\in L^2(\Omega)$, we define the normal component of $v$ on $\partial\Omega$, $\nu\cdot v\in H^{-\sfrac{1}{2}}(\partial\Omega)$ via its action on elements $u_0\in {H^{\sfrac{1}{2}}}(\partial\Omega)$ as, with mild abuse of notation,
	 	\begin{align*}
	 		\int_{\partial\Omega}u_0 (v\cdot\nu)\,d\mathcal{H}^1= \int_{\Omega}\nabla u\cdot v+\text{div}(v)u\,dx,
	 	\end{align*}
	 	where $u\in H^1(\Omega)$ is any arbitrary extension of $u_0$.
	 \end{definition}
	 
	 \begin{proposition}\label{prop.normal.deriv.equal.h1}
	 	Let $Z\in H^1(\Omega)$ satisfy $\Delta Z = Z$ weakly. Then we define $\frac{\partial Z}{\partial\nu}\in H^{-\sfrac{1}{2}}(\partial\Omega) $ as $\nabla Z\cdot\nu$ according to Definition \ref{def.normal.comp}, which satisfies 
	 	\begin{align*}
	 		\Bigl\|\frac{\partial Z}{\partial \nu}\Bigr\|_{H^{-\sfrac12}(\partial \Omega)}=\|Z\|_{H^{1}(\Omega)}\,.
	 	\end{align*}
	 \end{proposition}
	 \begin{proof}
	 	We turn directly to the definition of the dual norm and the normal derivative and see that 
	 	\begin{align*}
	 		\Bigl\|\frac{\partial Z}{\partial \nu}\Bigr\|_{H^{-\sfrac12}(\partial \Omega)} =& \sup\limits_{\substack{u_0\in {H^{\sfrac{1}{2}}}(\partial\Omega)\\ \|u_0\|_{{H^{\sfrac{1}{2}}}(\partial\Omega)}\leq 1 }}\int_{\partial\Omega} u_0\frac{\partial Z}{\partial\nu}\,d\mathcal{H}^1\\
	 		=& \sup\limits_{\substack{u\in H^1(\Omega), \Delta u = u \\ \|u\|_{H^1(\Omega)}\leq 1}}\int_{\Omega} \nabla u\cdot \nabla Z+\text{div}(\nabla Z)u\,dx\\
	 		=& \sup\limits_{\substack{u\in H^1(\Omega), \Delta u = u \\ \|u\|_{H^1(\Omega)}\leq 1}}\int_{\Omega} \nabla u\cdot \nabla Z+Zu\,dx\\
	 		=&\sup\limits_{\substack{u\in H^1(\Omega), \Delta u = u \\ \|u\|_{H^1(\Omega)}\leq 1}}\langle u,Z\rangle_{H^1(\Omega)}=\|Z\|_{H^1(\Omega)},
	 	\end{align*}
	 	since, by Cauchy-Schwarz, we see that $u =\frac{Z}{\|Z\|_{H^1(\Omega)}}$ is admissible, and attains the supremum.
	  \end{proof}

   \section{The case of a single particle} \label{ss.single}
	 In order to obtain an expression for the energy in this setting, we first fix $m \in \mathbb{Z}.$ We first compute the contribution to the self-energy associated with the $m$th mode. To be precise, 
	 \begin{lemma} \label{l.singleKm}
	 	Define $\Phi_m$ to be the solution to 
	 	\begin{equation} \label{e.phimeq}
	 		\begin{aligned}
	 			&\Delta \Phi_m = \frac{1}{\e^4} \Phi_m, \quad \quad x\in\RR^2 \setminus B(0,r_1),\\
	 			& \Phi_m(x) = e^{im\theta}, \quad \quad |x| = r_1.
	 		\end{aligned}
	 	\end{equation}
	 	Then, 
	 	\begin{align}
	 		\int_{\RR^2 \setminus B(0,r_1)} \left( |\nabla \Phi_m|^2 + \frac{1}{\e^4}|\Phi_m|^2 \right)\,dx = -\frac{2\pi r_1}{\e^2} \frac{K_{m}^\prime \big( \frac{r_1}{\e^2} \big) }{K_m(\frac{r_1}{\e^2})},
	 	\end{align}
   where $K_m$ is the modified Bessel function of the second kind and order $m$ (see Appendix~\ref{s.app1}).
	 \end{lemma}
	 \begin{proof}
	 	The proof is by construction of a radial profile. Specifically, we seek $\Phi_m (x) := f_m(r) e^{im\theta},$ with $f_m(r) = 1$ when $r = r_1.$ Then $f_m$ solves the ODE 
	 	\begin{equation}
	 		f_m^{\prime \prime} + \frac{1}{r} f_m^\prime - \frac{m^2}{r^2} f_m = \frac{1}{\e^4} f_m, \quad \quad f_m(r_1) = 1. 
	 	\end{equation}
	 	Then, arguing as before and rescaling, it is easy to see that the solution that decays at infinity is given by 
	 	\begin{align*}
	 		f_m(r) = \frac{K_m \big( \frac{r}{\e^2}\big)}{K_m \big( \frac{r_1}{\e^2} \big)}. 
	 	\end{align*}
	 	In particular, by the strict convexity of the energy, and the associated uniqueness for \eqref{e.phimeq}, we conclude that 
	 	\begin{align*}
	 		\Phi_m(x) = \frac{K_m \big( \frac{|x|}{\e^2}\big)}{K_m \big( \frac{r_1}{\e^2} \big)} e^{im\theta}.
	 	\end{align*}
	 	The energy of this function is then easily computed: using the divergence theorem, we write 
	 	\begin{equation}
	 		\begin{aligned}
	 			&\int_{\RR^2 \setminus B(0,r_1)} \left( |\nabla \Phi_m|^2 + \frac{1}{\e^4}|\Phi_m|^2 \right)\,dx  = -\int_{\partial B(0,r_1)} \Phi_m \cdot {\frac{\partial \Phi_m}{\partial \nu}} \,d \sh^1\\
	 			&\quad = \frac{2\pi r_1}{\e^2} \frac{K_m^\prime \left( \frac{r_1}{\e^2} \right)}{K_m \left( \frac{r_1}{\e^2}\right)}\,.
	 		\end{aligned}
	 	\end{equation}
	 The proof of the lemma is complete. 
	 \end{proof}

	\section{Estimates of modified Bessel functions of the second kind} \label{s.app1}
	For each $m \in \N,$ the homogeneous ordinary differential equation 
	\begin{equation*}
		t^2 u^{\prime \prime} + t u^\prime - (t^2 + m^2)u = 0\,, \quad t > 0\,,
	\end{equation*}
has two linearly independent solutions: $I_m$ and $K_m.$ The former, $I_m$ the modified Bessel function of the first kind, is exponentially growing, and is not used in the sequel, while the latter, $K_m$, the modified Bessel function of the second kind, is exponentially decaying. In this appendix, we summarize certain estimates on these functions in the form that we will need them. 

\begin{lemma}
	\label{l.besselfct1}
	Let $m \in \Z$ be fixed and $R > 1$. Then, there exists a constant $C> 0$ independent of $m$ and $R$, such that  the Bessel function $K_m$ satisfies the following pointwise estimate that for all $t \in [R,2R]:$ 
	\begin{equation*}
		\Bigl|\frac{K_m(t)}{K_m(R)} - \sqrt{\frac{R}{t}} e^{-(t-R)} \Bigr| \leqslant  \frac{C(1+|m|)}{R}\,.
	\end{equation*}
\end{lemma}
\begin{proof}
    We refer the reader to~\cite{AS}. 
\end{proof}
{We also repeatedly used the following relations satisfied by the modified Bessel functions:
\begin{equation} \label{e.besselderivatives}
    K_m' = -\frac12(K_{m-1} + K_{m+1})\,, \quad m \in \NN; \quad K_0' = - K_1\, , 
\end{equation}
 {
and the large-argument asymptotics that, for each $m\in\mathbb{Z}$ there exists $C>0$ with 
\begin{equation}\label{eq.besseldecay}
    K_m(x)\leq \frac{C}{\sqrt{x}}\exp(-x)
\end{equation}
for sufficiently large $x$. }
The proofs of these results are standard and can be found in any book on special functions (e.g.~\cite{AS}). 
}

\section{ A brief introduction to the Landau-de Gennes model}
\label{sec:LDGmodel}

The main characteristic feature of the nematic liquid crystals is the local preferred orientation of the rod-like molecules. A comprehensive way of modeling this is through a a probability measure  $\mu(x,\cdot):\mathcal{L}(\mathbb{S}^2)\to\mathbb
[0,1]$ for 	 each material point $x$ in the region $\Omega$ occupied
by the liquid crystal. Thus $\mu(x,A)$ assigns  a number between $0$ and $1$ denoting the probability that the molecules with centre of mass in a very small neighborhood of  the point $x\in\Omega$ are pointing in a direction contained in $A\subset
\mathbb{S}^2$.

The significant numerical and analytical challenges associated generated by dealing with parametrised probability measures
have lead Pierre Gilles de Gennes in the 70s to propose replacing the probability measure by one of its moments. Due to the physical head-to-tail symmetry of the molecules the first order moment vanishes (see for details \cite{ball2011orientability,mottram2014introduction}).  Thus the  first nontrivial information on $\mu$ comes from the
tensor of second moments:
$$M_{ij}\stackrel{\rm{def}}{=}\int_{\mathbb{S}^2}p_ip_j\,d\mu(p),\,i,j=1,2,3.$$
We have $M=M^T$ and
$\textrm{tr}\,M=\int_{\mathbb{S}^2}d\mu(p)=1$. \par If the orientation of the molecules is equally distributed in
all directions we say that the distribution is {\it isotropic} and
then $\mu=\mu_0$ where $d\mu_0(p)=\frac{1}{4\pi}dA$. The
corresponding second moment tensor is
$$M_0\stackrel{\rm{def}}{=}\frac{1}{4\pi}\int_{\mathbb{S}^2} p\otimes
p\,dA=\frac{1}{3}Id$$ (since $\int_{\mathbb{S}^2}
p_1p_2\,d\mu(p)=0,\,\int_{\mathbb{S}^2}p_1^2\,d\mu(p)=\int_{\mathbb{S}^2}p_2^2\,d\mu(p)=\int_{\mathbb{S}^2}p_3^2\,d\mu(p)$
and $\textrm{tr}\,M_0=1$).
\par The de Gennes order-parameter tensor $Q$ is defined as

\begin{equation}
Q\stackrel{\rm{def}}{=}M-M_0=\int_{\mathbb{S}^2}\left(p\otimes
p-\frac{1}{3}Id\right)\,d\mu(p)\label{defq}
\end{equation} and measures the deviation of the second moment
tensor from its isotropic value.

By extension we call a $Q$-tensor any symmetric, traceless, three-by-three real-valued matrix and denote the space of such $Q$-tensors by $\mathcal{S}_0$. The 
configuration of the nematic material is then described by maps $Q:\Omega\to \mathcal{S}_0$. The simplest theory that produces physically meaningful predictions is a variational one. In it equilibrium configurations of liquid crystals are obtained, for instance, as
energy minimizers, subject to suitable boundary conditions. The simplest commonly used energy functional is
\begin{equation}
\label{energy}
 \mathcal{F}_{LG}[Q]=\int_{\Omega}\left[ \frac{L}{2}\sum_{i,j,k=1}^3 Q_{ij,k}Q_{ij,k}+\frac{a}{2}\textrm{tr}\,Q^2 +
\frac{ b}{3}\textrm{tr}\,Q^3 +\frac{ c}{4}\left(\textrm{tr}\,Q^2\right)^2\right]\,dx
\end{equation} where $a,b,c$ are temperature and material dependent constants and $L>0$ is the elastic constant. The ``elastic part" $\frac{L}{2}\sum_{i,j,k=1}^3 Q_{ij,k}Q_{ij,k}$ models the spatial variations of the material while the ``bulk term" $f_B(Q)=\frac{a}{2}\textrm{tr}\,Q^2 +
\frac{ b}{3}\textrm{tr}\,Q^3 +\frac{ c}{4}\left(\textrm{tr}\,Q^2\right)^2$ models the phase transition from the isotropic state (no local preferred orientation of the molecules) to the nematic state of material. 

The bulk term is required to respect physical invariances of the material and thus can only be a function of $\textrm{tr}(Q^2)$ and $\textrm{tr}(Q^3)$. Following Landau's intuition it a polynomial chosen to be of the lowest possible order such that the mathematical predictions match the physical ones. Out of the three coefficients only $a$ depends on the temperature and varying $a$ provides different types of minimisers (see \cite{mottram2014introduction}, Section $II.A$ for details) with a  negative enough $a$ giving a nematic-type minimiser, that is an element in the set $\{s_+(a,b,c)\left(n\otimes n-\frac{1}{3}Id\right); n\in\mathbb{S}^2\}$ with $s_+(a,b,c)$ an explicitly computable scalar and $Id$ the three-by-three identity matrix. We will be interested in the paranematic situation when the parameter 
$a$ positive and large enough provides  a zero Q-tensor as minimiser for $f_B(Q)$. It should be noted that in this setting  the bulk term behaves qualitatively  as a perturbation of the quadratic term so it is expected, as in \cite{galatola2003interaction} for instance, that replacing $f_B(Q)$ by  the quadratic $ g(Q)=\textrm{tr}(Q^2)$. In this case the different components of the $Q$-tensor are not coupled hence problem can be reduced to independent scalar problems as will be the focus of most of the paper.

	\bibliographystyle{acm}
	\bibliography{ref}
	
\end{document}